%% file: arxiv.tex
\documentclass[letter,11pt]{article}

\input{math_commands.tex}

\usepackage[margin=1in]{geometry}
\usepackage{subfigure}
\usepackage{hyperref}
\usepackage{url}
\usepackage{amsthm}
\usepackage{algorithm, algpseudocode}
\usepackage{graphicx}
\usepackage{subfigure}
\usepackage{enumitem}

\newtheorem{theorem}{Theorem}

\newtheorem{lemma}[theorem]{Lemma}
\newtheorem{corollary}[theorem]{Corollary}
\theoremstyle{definition}
\newtheorem{definition}[theorem]{Definition}
\newtheorem{assumption}[theorem]{Assumption}

\theoremstyle{remark}
\newtheorem{remark}[theorem]{Remark}

\title{\vspace{-8mm}\textbf{Beyond Convexity: Proximal-Perturbed Lagrangian Methods for Efficient Functional Constrained Optimization}}

\date{}
\vspace{-0.05in}
\author{Sang Bin Moon\footnotemark[1] \footnotemark[2] \and Jong Gwang Kim\footnotemark[1] \footnotemark[3] \and Ashish Chandra\footnotemark[4] \and Christopher Brinton\footnotemark[2] \and Abolfazl Hashemi\footnotemark[2]}
\begin{document}
\maketitle
\footnotetext[1]{These authors contributed equally to this work.}
\footnotetext[2]{School of Electrical and Computer Engineering, Purdue University, West Lafayette, IN.}
\footnotetext[3]{Coles College of Business, Kennesaw State University, Kennesaw, GA.}
\footnotetext[4]{Department of Management and Quantitative Methods, College of Business, Illinois State University, Normal, IL.}
\begin{abstract}
Non-convex functional constrained optimization problems have gained substantial attention in machine learning and data science, addressing broad requirements that typically go beyond the often performance-centric objectives. An influential class of algorithms for functional constrained problems is the class of primal-dual methods which has been extensively analyzed for convex problems. Nonetheless, the investigation of their efficacy for non-convex problems is under-explored. This paper develops a primal-dual algorithmic framework for solving such non-convex problems. This framework is built upon a novel form of the Lagrangian function, termed the {\em Proximal-Perturbed Augmented Lagrangian}, which enables the development of simple first-order algorithms that converge to a stationary solution under mild conditions. Notably, we study this framework under both non-smoothness and smoothness of the constraint function and provide three key contributions: (i) a single-loop algorithm that does not require the continuous adjustment of the penalty parameter to infinity; (ii) a non-asymptotic iteration complexity of $\widetilde{\mathcal{O}}(1/\epsilon^2)$; and (iii) extensive experimental results demonstrating the effectiveness of the proposed framework in terms of computational cost and performance, outperforming related approaches that use regularization (penalization) techniques and/or standard Lagrangian relaxation across diverse non-convex problems.
\end{abstract}
\section{Introduction} \label{sec:intro}

We consider the following non-convex optimization problem with functional constraints: 
\begin{equation} \label{eq:op}
\begin{aligned}
\underset{\x \in \mathbb{R}^n}{\mathrm{min}}  \ \ f(\x) + r(\x) \quad \mathrm{s.\:t.} \quad  g(\x) \leq \0,
\end{aligned}
\end{equation}
where $f: \mathbb{R}^n \rightarrow \mathbb{R}$ and $g: \mathbb{R}^n \rightarrow \mathbb{R}^m$ are continuous and possibly non-convex mappings; and $r: \mathbb{R}^n \rightarrow \mathbb{R} \cup \{ + \infty\}$ is a proper, closed, and convex (possibly non-smooth) function.

Problems of this form \eqref{eq:op} appear in a wide range of applications in machine learning, data science, and signal processing, e.g., wireless transmit/receive beamforming design \cite{scutari2016parallel2,shi2020penalty}, constrained classification/detection problems \cite{huang2019stable,rigollet2011neyman,zafar2019fairness}, and optimization for deep neural network \cite{bai2023efficient,jiang2023optimality}. Solving non-convex problems, even those without constraints, is generally challenging, as finding even an approximate global minimum is often computationally intractable \cite{nemirovskij1983problem}. The presence of functional constraints $g(\mathbf{x})$ in \eqref{eq:op} that can potentially be non-convex is critical for many of the applications mentioned above, yet it makes the problem even more challenging. A further complication arises since in many of these applications, problem \eqref{eq:op} tends to be large-scale, i.e., with large variable dimension $n$ \cite{boyd2011distributed}. Hence, developing first-order methods that can find stationary solutions with lower complexity bounds is highly desirable.

Augmented Lagrangian (AL)-based algorithms are a prevailing class of approaches for constrained optimization problems. The foundational AL method, introduced by \cite{hestenes1969multiplier} and \cite{powell1969method}, has been a powerful algorithmic framework built on by many contemporary algorithms. In particular, the Alternating Direction Method of Multipliers (ADMM) scheme has been widely employed for solving constrained optimization problems based on the AL framework; see \cite{bertsekas2014constrained,birgin2014practical} and recent works for constrained convex settings \cite{ouyang2015accelerated,lan2016iteration,xu2017accelerated,liu2019nonergodic,xu2021iteration,upadhyay2025optimization}. 

However, AL-based methods remain fairly limited for problems in the general form of \eqref{eq:op}. Key challenges arise from the non-convexity of the objective and constraint functions, which can lead to complicated updates with no closed-form or compute-efficient updates and
require carefully updating the penalty parameters to ensure the solution remains near feasible. Consequently, existing analyses of AL-based methods, with the best-known guarantees of $\mathcal{O}(1/\epsilon^3)$ for a given $\epsilon > 0$, require increasing penalty parameters to infinity to ensure feasibility, leading to demanding iteration complexity.

Motivated by the above discussion, we aim to answer the question:

\vspace{0.1in}
\begin{quote}
    {\em Can we design algorithms to solve problems of the form \eqref{eq:op} with an improved iteration complexity bound and efficient update rules?}
\end{quote}
\vspace{0.1in}

This paper answers this question in the affirmative. Notably, we develop an efficient single-loop first-order primal-dual method for solving problem \eqref{eq:op} such that based on a new augmented Lagrangian, for a given accuracy $\epsilon > 0$ to compute an $\epsilon$-approximate stationary solution (see Definition \ref{def_epsilon-kkt}). It achieves an iteration complexity of $\widetilde{\mathcal{O}}(1/\epsilon^{2})$ in terms of the number of gradient evaluations.\footnote{In this paper, the notation $\widetilde{\mathcal{O}}(\cdot)$ suppresses all logarithmic factors from the big-$\mathcal{O}$ notation.}

\subsection{Related Work} 
We review the literature on iteration complexity and convergence of AL and penalty-based methods for non-convex constrained problems.

\textbf{Linearly constrained non-convex problems.} Many existing works have focused on the class of problems where $g(\mathbf{x})$ in~\eqref{eq:op} is linear. \cite{hajinezhad2019perturbed} introduced a perturbed-proximal primal-dual algorithm, with an iteration complexity of $\widetilde{\mathcal{O}}(1/\epsilon^4)$, under the assumption of a feasible initialization. \cite{kong2023iteration} proposed proximal AL methods that obtain the improved complexity result of $\widetilde{\mathcal{O}}(1/\epsilon^3)$ under Slater’s condition. Finally, \cite{zhang2020proximal,zhang2022global} proposed a first-order single-loop proximal AL method that achieves $\mathcal{O}(1/\epsilon^2)$ iteration complexity, which relies on error bounds that depend on the Hoffman constant of the polyhedral constraints.\footnote{The Hoffman constant $\kappa$ is the smallest number such that for any $\x$, $\text{dist}(\x, \{ \y \mid A\y \leq b \}) \le \kappa \| (A\x - b)_+ \|$,
where $(A\x - b)_+$ denotes the positive part of $A\x - b$.
} However, estimating the Hoffman constant is known to be difficult in practice.

\textbf{Non-convex functional constrained problems.} There are several recent works that focus on the iteration complexity of first-order AL-based methods or penalty methods to solve~\eqref{eq:op}  \cite{cartis2011evaluation,scutari2016parallel,li2021rate,lin2022complexity,lu2022single,kong2022iteration,sahin2019inexact}. \cite{scutari2016parallel} proposed double-loop distributed primal-dual algorithms with asymptotic convergence guarantees,
under the coercivity assumption and Mangasarian-Fromovitz constraint qualification (MFCQ). 
More recently, a set of methods have emerged employing the regularity condition (Assumption 4 of \cite{lin2019inexact}, (8) of \cite{sahin2019inexact}) for ensuring solution feasibility. \cite{sahin2019inexact} proposed a double-loop inexact AL method (iALM) that achieves an $\widetilde{\mathcal{O}}(1/\epsilon^4)$ iteration complexity. \cite{li2021rate} improved the iteration complexity to $\widetilde{\mathcal{O}}(1/\epsilon^3)$, which is obtained using a triple loop iALM. \cite{kong2022iteration} established an $\widetilde{\mathcal{O}}(1/\epsilon^3)$ complexity bound of the proximal AL method (NL-IAPIAL) for non-convex problems with nonlinear convex constraints. \cite{lu2022single} proposed the first single-loop gradient-based algorithm that achieves the best-known iteration complexity $\mathcal{O}(1/\epsilon^3)$ for \eqref{eq:op}. However, the regularity condition is non-standard and rather strong as it forces a relationship between feasibility of the generated iterates and first-order optimality. We are thus motivated to develop an algorithm that improves iteration complexity without requiring this assumption. While our proposed framework shares similarity with the above influential works, it further provides certain key components which we delineate in Section \ref{sec:conc}.

\subsection{Our Contributions}\label{sec:conc}
In this paper, we develop a novel algorithmic framework designed to solve challenging non-convex functional constrained optimization problems. Our paper offers the following key contributions:

\begin{itemize}
\item We propose a single-loop first-order algorithm that finds an $\epsilon$-solution with $\widetilde{\mathcal{O}}(1/\epsilon^2)$ iteration complexity for both non-smooth and smooth constraint functions without requiring the strong regularity condition used in other AL-based algorithms  \cite{li2021rate,lin2022complexity,lu2022single,sahin2019inexact}.

\item To establish the above results, we conduct a comprehensive convergence analysis of our method for both non-smooth and smooth constraint functions, establishing simple and concise proofs compared to existing works. Notably, the analysis does not impose assumptions on the surjectivity of the Jacobian $J_g(\x)$ \cite{bolte2018nonconvex,boct2020proximal}, or boundedness of penalty parameters \cite{grapiglia2021complexity}. It also does not require the feasibility of initialization as in \cite{boob2022stochastic,sun2021dual,xie2021complexity}, which itself is a non-convex and challenging task. 

\item By using a constant penalty parameter, our algorithm achieves improved computational efficiency and ease of implementation compared to existing schemes. Specifically, we neither require linear independence constraint qualification (LICQ) to ensure boundedness of penalty parameters \cite{solodov2009global}, nor computational efforts for careful updating scheme of the penalty parameters.
Our numerical results validate that compared with existing methods, using a fixed penalty parameter achieves more consistent progress toward solution stationarity and feasibility.

\item The algorithmic framework is flexible, enabling it to effectively handle various non-convex, smooth, and non-smooth functional constraints. Experimental results demonstrate its high effectiveness in terms of computational cost and performance, outperforming related algorithms that use regularization techniques and/or standard Lagrangian relaxation.
\end{itemize}

\subsection{Outline}
Section \ref{sec:prelim} provides the notation, definitions, and assumptions that we use throughout the paper. In Section \ref{sec:non_smooth} and \ref{sec:smooth}, we propose novel first-order primal-dual algorithms and establish their convergence results for non-smooth and smooth functional constraints, respectively. Section \ref{sec:experiments} presents numerical results on commonly encountered problems in signal processing and machine learning to demonstrate the effectiveness of the proposed algorithm. Detailed derivations are provided in the supplementary material due to space limitations.

\section{Preliminaries} \label{sec:prelim}
This section provides the notation, formal definitions, and assumptions utilized throughout this paper, forming the foundation for our proposed algorithmic approach and its convergence analysis.

We adopt the following notation: $\mathbb{R}^n$ denotes the $n$-dimensional Euclidean space, and $\mathbb{R}^n_+$ represents the non-negative orthant. We use $[m]$ to denote the set $\{1,\ldots,m\}$. The inner product between two vectors is denoted by $\left\langle \cdot, \cdot \right\rangle$, and the Euclidean norm of matrices and vectors is denoted by $\left\| \: \cdot \: \right\|$. The distance function between a vector $\x$ and a set $\mathcal{X} \subseteq \mathbb{R}^n$ is defined as $\text{dist}(\x,\mathcal{X}) := \inf_{\y \in \mathcal{X}} \|\y -\x \|.$
For a proper extended real-valued function $r$, its {\em domain} is defined as $\textrm{dom}(r) := \left\lbrace \mathbf{x} \in \mathbb{R}^n: r(\x) < +\infty  \right\rbrace$. A function $r$ is considered {\em proper} if $\text{dom}(r) \neq \emptyset$ and it does not take the value $- \infty$. It is {\em closed} if it is lower semicontinuous, meaning $\liminf_{\x \rightarrow \x^0}  r(\x) \geq r(\x^0)$ for any $\x^0 \in \mathbb{R}^n$. For a convex function $r$ at $\x$, its {\em subgradient} is denoted by {\small $\partial r(\x):= \left\lbrace \dd \in \mathbb{R}^n: r(\y) \geq r(\x) + \left\langle \dd, \y - \x \right\rangle,   \forall \y \in \mathbb{R}^n, \ \x \in \text{dom}(r) \right\rbrace$}. The {\em proximal map} associated with a proper, closed, and convex function $r:\mathbb{R}^n \rightarrow \mathbb{R} \cup \left\lbrace + \infty \right\rbrace$ at $\x \in \mathbb{R}^n$ with $\eta>0$ is uniquely defined by $\text{prox}_{\eta r}(\x)={\text{argmin}_{\y \in \mathbb{R}^n}}\left\lbrace r(\y) + \frac{1}{2\eta} \| \x - \y\|^2 \right\rbrace.$

Next, we provide the formal definitions and assumptions for the class of functions, and the optimality measure under consideration. 

Assuming a suitable constraint qualification (CQ) holds, the stationary solutions of problem \eqref{eq:op} can be characterized by the points $(\x^{\ast}, \boldsymbol{\nu}^{\ast})$ satisfying the Karush-Kuhn-Tucker (KKT) conditions \cite{bertsekas1999nonlinear}.
\begin{definition}[The KKT point] \label{def_kkt}
A point $\mathbf{x}^\ast$ is called a \emph{KKT point} for problem \eqref{eq:op} if there exists $\boldsymbol{\nu}^\ast \in \mathbb{R}^m_+$ such that
\begin{equation} \label{eq:def_kkt}
\begin{cases}
    \0 \in \nabla f(\x^\ast) + \partial r(\x^\ast) + J_g(\x^\ast) \boldsymbol{\nu}^\ast, \\   
    g_j(\x^\ast) \leq 0, \ \ \boldsymbol{\nu}_j g_j(\x^\ast)  =0, \ \ j \in [m]. 
\end{cases}
\end{equation}
\end{definition}

A suitable CQ is necessary for the existence of multipliers that satisfy the KKT conditions (e.g., MFCQ, Constant Positive Linear Dependence (CPLD), and others; see \cite{andreani2022scaled, bertsekas1999nonlinear}). In practice, it is difficult to find an exact KKT solution $(\x^{\ast}, \boldsymbol{\nu}^{\ast})$ that satisfies \eqref{eq:def_kkt}. Thus, one typically aims to find an approximate KKT solution, defined as an $\epsilon$-KKT solution next.

\begin{definition}[$\epsilon$-KKT solution, Definition 2 of \cite{lu2022single}] \label{def_epsilon-kkt}
Given $\epsilon>0$, a point $\mathbf{x}^\star$ is called an {\em $\epsilon$-KKT solution} for problem \eqref{eq:op} if there exists $\boldsymbol{\nu}^\star \in \mathbb{R}^m_+$ such that
\begin{equation} \label{eq:def_epsilon-kkt}
\begin{cases}
\vv^\star \in \nabla f(\x^\star) + \partial r(\x^\star) + J_g(\x^\star)\boldsymbol{\nu}^\star, \quad \|\vv^\star\| \leq \epsilon, \\
\| \max\{0,g(\x^\star)\} \| \leq \epsilon, \quad \sum_{j=1}^m |\nu_j g_j(\x^\star)|\le\epsilon ,\notag 
\end{cases}
\end{equation}
where $\max \{\0, g(\x^\star)\}$ denotes the component-wise maximum of $g(\x^\star)$ and the zero vector $\0$ at $\x^\star$. 
\end{definition}

To establish the ensuing analysis, we introduce the following standard assumptions for problem \eqref{eq:op}: 

\begin{assumption} \label{assumption_kkt}
There exists a point $(\x, \boldsymbol{\nu}) \in \text{dom}(r) \times \mathbb{R}^m$ satisfying the KKT conditions \eqref{eq:def_kkt}.
\end{assumption}

\begin{assumption}\label{assumption_lipschitz_f}
$\nabla f$ is $L_f$-Lipschitz continuous on $\text{dom}(r)$. That is, there exist $L_f >0$ such that 
\begin{equation}
    \|\nabla f(\x) -\nabla f(\x^\prime)\| \leq L_f \|\x - \x^\prime\|, \ \forall \x, \x^\prime \in \text{dom}(r), \notag 
\end{equation}
\end{assumption}

\begin{assumption}\label{assumption_lipschitz_g}
$\nabla g$ is $L_g$-Lipschitz continuous on $\text{dom}(r)$. That is, there exist $L_g >0$ such that 
\begin{equation}
    \|\nabla g(\x) -\nabla g(\x^\prime)\|  \leq L_g \|\x - \x^\prime\|, \ \forall \x, \x^\prime \in \text{dom}(r). \notag
\end{equation}
\end{assumption}

\begin{assumption} \label{assumption_bounded_domain}
The domain of $r$ is compact, i.e.,
$D_{\x} := {\max}_{\x, \x^\prime \in \text{dom}(r)} \|\x - \x^\prime\| < + \infty.$ 
\end{assumption}

\begin{assumption}[Assumption 3 of \cite{na2023inequality}, Assumption 1 of \cite{na2023adaptive}, Assumption 3.1 of \cite{hong2023constrained}] \label{assumption_bounded_multiplier}
The iterates $\{\la_k\}$ generated by iterative methods for problem \eqref{eq:op} estimating $\boldsymbol{\nu}^\ast$ satisfying Definition \ref{def_kkt} are contained in a convex compact subset $\Lambda\subset\mathbb{R}^m$.
\end{assumption}

These assumptions are considered standard in the optimization literature \cite{boob2022stochastic,huang2023oracle} and are satisfied by a broad range of practical problems in signal processing and machine learning \cite{bolte2018nonconvex,li2021augmented,li2021rate,lu2022single,kong2022iteration,lu2023iteration}. Our work distinguishes itself by not requiring certain restrictive assumptions beyond those stated above, such as the surjectivity of $\nabla g(\x)$ (or $\nabla g(\x) \nabla g(\x)^{\top}$ being positive definite) \cite{bolte2018nonconvex,bot2019proximal,boct2020proximal,li2015global},  Slater’s condition \cite{boob2022stochastic,kong2022iteration}, or more crucially feasibility of initialization \cite{boob2022stochastic,hajinezhad2019perturbed,xie2021complexity} as that by itself is a non-convex problem when $g$ is nonconvex.
For problems with an unbounded $\text{dom}(r)$, they can be reformulated to satisfy Assumption \ref{assumption_bounded_domain}; for instance, if $f$ is bounded below and $r$ is coercive, the problem can be transformed to one with $f + r$ for some $r$ (e.g., norm functions) with a compact domain  \cite{lu2023iteration}. Notably, this can be implemented in  practice for machine learning problem using the standard practice of weight or gradient clipping.
Moreover, Assumption \ref{assumption_bounded_multiplier} is commonly used in the convergence analysis of constrained optimization algorithms \cite{nocedal2006numerical,bertsekas2014constrained,birgin2014practical,hong2016convergence,hong2023constrained,na2023adaptive,na2023inequality}. From certain constraint qualification, such as MFCQ or CPLD, it can also be derived that the set of KKT multipliers corresponding to a local minimum is bounded.

Furthermore, under Assumption \ref{assumption_bounded_domain}, there exist constants $B_g > 0$ and $M_g > 0$ such that 
\begin{align}
\underset{\x \in \text{dom}(r)}{\max}\| g(\x) \| \leq B_g \ \ \text{and} \ \
 \underset{\x \in \text{dom}(r)}{\max}\| \nabla g(\x) \| \leq M_g, \label{eq:bound_jacobian_gx}
\end{align}
which implies the Lipschitz continuity of $g$ \cite[Chapter 9.B]{rockafellar2009variational}: $\| g(\x) - g(\x^\prime) \|  \leq M_g \|\x - \x^\prime\|, \ \ \forall \x, \x^\prime \in \text{dom}(r).$

In the next section, we consider a non-convex optimization problem with non-smooth functional constraint. In this case, Assumption \ref{assumption_bounded_subgradient} implies the Lipschitz continuity of the subgradient, instead of \eqref{eq:bound_jacobian_gx} or Assumption \ref{assumption_lipschitz_g}.

\begin{assumption}\label{assumption_bounded_subgradient}
$g$ is continuous with $\partial g(\x) \neq \emptyset$  on $\text{dom}(r)$, and there exists a constant $M_g > 0$ such that ${\max}_{\x \in \text{dom}(r)}\| \partial g(\x) \| \leq M_g$.
\end{assumption}

\section{Non-convex Non-smooth Constraints} \label{sec:non_smooth}
In this section, we consider the non-convex optimization problem \eqref{eq:op} with a non-smooth functional constraint $g(\cdot)$. We first introduce a novel Lagrangian with a structure designed for developing an efficient algorithm that solves the non-smooth constrained problem. A critical feature of the resulting algorithm is its reliance on fixed parameters $\alpha$, $\beta$ and $\rho=\alpha/(1+\alpha\beta)$), which eliminates the need for the sensitive and iterative adjustments required by many existing schemes. This design not only simplifies implementation but also enhances computational efficiency. Empirical results demonstrate that the algorithm's performance is not sensitive to the choice of $\alpha$ and $\beta$, further highlighting its robustness.

\subsection{A Variant of Proximal-Perturbed Lagrangian}
Motivated by the reformulation techniques in \cite{bertsekas1999nonlinear, bertsekas2014constrained}, we employ {\em perturbation} variables $\z \in \mathbb{R}^m$ and slack variables $\uu \in \mathbb{R}^m_+$. By setting $g(\x) + \uu = \z$ and $\z = \0$, we first transform problem \eqref{eq:op} into an equivalent equality-constrained formulation: 
\begin{equation} \label{eq:ep}
\begin{aligned}
\underset{\x \in \mathbb{R}^n, \uu \in \mathbb{R}^m_+, \z \in \mathbb{R}^m }
	{\mathrm{min}}    f(\x)+r(\x)   \quad 
	\mathrm{s.\:t.}  \quad  g(\x) + \uu = \z, 
	                      \quad \z = \0. 
\end{aligned}
\end{equation}
Clearly, for $\z^\ast = \0$ and $\uu^\ast \geq 0 $, the extended formulation \eqref{eq:ep} is equivalent to problem \eqref{eq:op}. 

For this formulation \eqref{eq:ep}, we define a variant of {\em the Proximal-Perturbed Lagrangian} (P-Lagrangian) from \cite{kim2021equilibrium} as follows:
\begin{align} \label{eq:ppl} 
\mathcal{L}_{\alpha\beta}(\x,\uu,\z,\la,\m)  :=  f(\x) + \left\langle \la, g(\x) + \uu - \z \right\rangle + \left\langle \m, \z \right\rangle + \frac{\alpha}{2}\| \z \|^2  - \frac{\beta}{2}\| \la - \m \|^2 + r(\x),  
\end{align}
where $\la \in \mathbb{R}^m$ is a multiplier (dual) for the constraint $g(\x) + \uu - \z = 0$, $\m \in \mathbb{R}^m$ is an {\em auxiliary multiplier} for the constraint $\z = \0$, $\alpha > 0$ is a penalty parameter, and $\beta>0$ is a proximal parameter. 

Given $(\la,\m)$, $\mathcal{L}_{\alpha\beta}$ can be minimized with respect to $\z$ in closed form: 
$$\z(\la,\m)=(\la-\m)/ {\alpha}.$$ 
Substituting $\z(\la, \m)$ back into $\mathcal{L}_{\alpha\beta}$ yields the reduced P-Lagrangian: 
\begin{align} \label{eq:reduced_ppl}
\mathcal{L}_{\alpha\beta}(\x,\uu,\z(\la, \m),\la,\m) 
& = f(\x) + \left\langle \la, g(\x) + \uu  \right\rangle - \frac{1}{2\rho} \| \la - \m \|^2 + r(\x), 
\end{align}
where $\rho:=\frac{\alpha}{1+\alpha\beta}$. Note that \eqref{eq:reduced_ppl} is $\frac{1}{\rho}$-strongly concave in $\la$ for a fixed $\m$. This property guarantees a unique maximizer for $\la$, which can be found in closed form:
\begin{align} 
\la(\x,\m) = \underset{\la \in \mathbb{R}^m}{\textrm{argmax}} \; \mathcal{L}_{\alpha\beta}(\x,\uu,\z(\la,\m),\la,\m) = \m + \rho (g(\x) + \uu), \label{eq:opt_lambda}
\end{align}
which is well-defined and is used for the update of $\la_{k+1}$ in \eqref{eq:lambda_update}.

\subsection{Description of Algorithm}
Based on the P-Lagrangian, we propose the P-Lagrangian based Alternating Direction Algorithm (PLADA) for solving problem \eqref{eq:op} with non-smooth constraints. The complete procedure is detailed in Algorithm \ref{alg:plada}.

\begin{algorithm} [t!]	
\caption{P-Lagrangian based Alternating Direction Algorithm (PLADA)}  \label{alg:plada}
\begin{algorithmic}[1]

\State \textbf{Input:} Initialization $(\x_0,\uu_0,\z_0,\la_0,\m_0)$, and parameters $\alpha > 1$, $\beta \in (0,1)$, $\rho =\frac{\alpha}{1+ \alpha \beta}$, $0 < \eta < \frac{1}{L_f +   3 \rho M_g^2}$, $0 < \tau < \frac{1}{3\rho}$, $\delta_0\in(0,1]$, and $K$.

\For{$k= 0, 1, \dots, K $} 

\State $\x_{k+1} = \underset{\x \in \mathbb{R}^n}{\mathrm{argmin}} \left\lbrace \left\langle \nabla f(\x_k), \x  \right\rangle + \left\langle \la_k, g(\x) \right\rangle + {1}/{2\eta} \| \x - \x_k \|^2 + r(\x) \right\rbrace$;

\State $\uu_{k+1} = \Pi_{\mathbb{R}_+^m}[\uu_{k}  - \tau \la_{k}]$; 

\State $\m_{k+1} = \m_k + \sigma_k (\la_k - \m_k), \  \sigma_k =  \min\left\{\sigma_0, \frac{\rho\delta_k}{\| \la_k - \m_k \|^2 + 1}\right\}$;

\State $\la_{k+1} = \m_{k+1} + \rho (g(\x_{k+1}) + \uu_{k+1})$;

\State $\z_{k+1} = \frac{1}{\alpha}(\la_{k+1} - \m_{k+1})$;

\EndFor
\end{algorithmic} 
\end{algorithm}

Each iteration of PLADA consists of a sequence of updates for the primal, dual, and auxiliary variables. The primal variable $\x$ is updated using a proximal gradient step:
\begin{equation} 
\begin{aligned}
\x_{k+1} & = \underset{\x \in \mathbb{R}^n}{\mathrm{argmin}} \left\lbrace \left\langle \nabla f(\x_k), \x  \right\rangle + \left\langle \la_k, g(\x) \right\rangle + {1}/{2\eta} \| \x - \x_k \|^2 + r(\x) \right\rbrace. \label{eq:x_update}
\end{aligned}
\end{equation}

The slack variable $\uu$ is updated via projected gradient descent onto $\mathbb{R}_+^m$: 
\begin{align}
\uu_{k+1} 
 & = \underset{\uu \in \mathbb{R}_+^m}{\mathrm{argmin}} \left\lbrace   \left\langle \nabla_{\uu} \mathcal{L}_{\alpha\beta}(\x_k, \uu_k,\z_k,\la_k,\m_k), \uu  - \uu_k \right\rangle + {1}/{2\tau} \| \uu - \uu_k \|^2 \right\rbrace  = \Pi_{\mathbb{R}_+^m}[\uu_{k}  - \tau \la_{k}],  \label{eq:u_update}
\end{align}
where, without loss of generality, we can construct an upper bound $\max_{k\ge1}\{\uu_k\}\le B_g$ as $\|g(\x) \| \leq B_g$.

The auxiliary multiplier $\m$ is updated using a gradient ascent step on \eqref{eq:reduced_ppl}:
\begin{align} 
\m_{k+1}  
= \m_{k} + \sigma_k (z_k + \beta(\la_k - \m_k))  
= \m_{k} + \frac{\sigma_k}{\rho} (\la_k - \m_k), \label{eq:mu_update}
\end{align}
with a diminishing step-size $\sigma_k = \min \left\lbrace \sigma_0, \rho \delta_k/({\|\la_k - \m_k \|^2 +1}) \right\rbrace$, which is governed by a sequence $\delta_k>0$ that satisfies the standard conditions: 
\begin{equation} \label{eq:delta_condition}
 \delta_{0} \in (0,1], \ \ \lim_{k \to \infty} \delta_k = 0, \ \ \text{and} \ \ \sum_{k=0}^{\infty} \delta_k = +\infty. 
\end{equation}
In Algorithm \ref{alg:plada}, we choose $\delta_k = \kappa \cdot (k+1)^{-1}$ with $\kappa > 0,$ so that these conditions hold.

The main dual variable $\la$ is subsequently updated via an exact maximization on \eqref{eq:reduced_ppl}:
\begin{align} 
\la_{k+1}  
 = \m_{k+1} + \rho \left(g(\x_{k+1}) + \uu_{k+1} \right). \label{eq:lambda_update}
\end{align}

Finally the auxiliary variable $\z$ is updated by an exact minimization on \eqref{eq:ppl}: 
\begin{equation} 
\z_{k+1} = (\la_{k+1}-\m_{k+1})/{\alpha}.  \label{eq:z_update}
\end{equation}

A key advantage of this framework is that the parameters $\alpha, \beta $, and the dual step size $\rho$ are constants, independent of the number of iterations $k$. In Section \ref{sec:additional_experiments}, we demonstrate the robustness of the algorithm with respect to the choices of $\alpha$ and $\beta$.

\subsection{Convergence Guarantees}
This subsection establishes the convergence guarantees for Algorithm \ref{alg:plada}. Our analysis begins by defining the necessary concepts from subdifferential calculus for non-smooth functions. We then present a series of technical lemmas that establish convergence of the algorithm's iterates with respect to the Definitions \ref{def_kkt} and \ref{def_epsilon-kkt}. These results culminate in theorems proving the algorithm's asymptotic convergence as well as the non-asymptotic rate of convergence on expectation. All proofs are contained in Supplementary Material for concise main paper.

We first recall some essential definitions from variational analysis. We denote the Jacobian matrix of $g$ at $\x$ by $\partial g(\x)$. For any set $\mathcal{X} \subseteq\mathbb{R}^d$, its indicator function $\mathbb{I}_{\mathcal{X}}$ is defined by $\mathbb{I}_{\mathcal{X}} = 0   \text{ if }   \x \in \mathcal{X} \ \text{ and }  + \infty, \ \text{ otherwise}.$  Note that $\arg\min_{\x \in \mathcal{X}} F(\x) = \arg\min_{\x \in \mathbb{R}^d} \{\varphi(\x) := F(\x) + \mathbb{I}_{\mathcal{X}}(\x)\}.$ 

\begin{definition} [Definition 8.3 of \cite{rockafellar2009variational}]
Let $g_i:\mathbb{R}^d \rightarrow \mathbb{R} \cup \left\lbrace + \infty \right\rbrace$ be a proper and lower semicontinuous function. 
For each $\x \in \mathcal{X}$, the \emph{Frechet subdifferential} of $g$ of $\x$ is given by
\begin{align}
\widehat{\partial} g_i(\x) := \left\lbrace  d_{k} \in \mathbb {R}^d: \underset{w \rightarrow \x}{\mathrm{lim \: inf}} \: \frac{g_i(w) - g_i(\x)-\left\langle  d, w - \x \right\rangle }{\left\|w - \x \right\|} \geq 0 \right\rbrace.     \notag  
\end{align}
\end{definition}

\begin{definition}
The {\em limiting subdifferencial} (or simply the subdifferential) of $g_i$ at $\x \in \mathbb{R}^d$ is defined as
\begin{align}
{\partial} g_i(\x) := \left\lbrace d \in \mathbb {R}^d: \exists \, \x_k {\rightarrow} \x \ \textrm{and} \ d_{k} \in \widehat{\partial} g_i(\x_k) \ \textrm{with} \ d_{k} \rightarrow d \ \textrm{as} \  k \rightarrow \infty \right\rbrace.       \notag
\end{align}
\end{definition}
The inclusion $\widehat{\partial} g_i(\x) \subseteq {\partial} g_i(\x)$ holds for each $\x \in \mathcal{X}$ and we set $\widehat{\partial} g_i(\x) = {\partial} g_i(\x) = \emptyset$ for $\x \notin \mathcal{X}$. Each $ d \in {\partial} g_i(\x)$  is called a subgradient of $g_i$ at $\x$.

We now present the main convergence theorems for Algorithm \ref{alg:plada}, establishing the asymptotic convergence to a KKT solution as defined in Definition \ref{def_kkt}.

\begin{lemma} [Primal Stationarity] \label{lem_primal_convergence_plada}
Let $\left\{\ww_{k} \right\}$ be the sequence generated by Algorithm \ref{alg:plada}, and let $\{\pp_k := (\x_k,\uu_{k},\z_{k})\}$ be the primal sequence. Under Assumptions \ref{assumption_kkt}, \ref{assumption_lipschitz_f}, \ref{assumption_bounded_domain} and \ref{assumption_bounded_subgradient}, the running averaged of the squared primal stationarity residual converges to zero:
\begin{equation}
  \lim_{T \rightarrow \infty} \frac{1}{T}\sum_{k=0}^{T-1} \| \boldsymbol{\zeta}_{\pp}^{k+1}\|^2 = 0,  \ \ \text{with the rate of} \ \mathcal{O}\left(\frac{\log(T)}{T}\right) = \tilde{\mathcal{O}}\left(\frac{1}{T}\right), \notag
\end{equation}
where $\boldsymbol{\zeta}_{\pp}^{k+1} :=(\boldsymbol{\zeta}_{\x}^{k+1},\boldsymbol{\zeta}_{\uu}^{k+1},\boldsymbol{\zeta}_{\z}^{k+1}) \in \partial_{\pp} \mathcal{L}_{\alpha\beta} (\ww_{k+1})$. Hence, any limit point $(\bar{\x},\bar{\la})$ of the sequence $(\x_k,\la_k)$ satisfies the stationarity condition of the original problem: $\0 \in \nabla f(\bar{\x}) + \partial r(\bar{\x}) + \partial g(\bar{\x})^{\top}\bar{\la}$.
\end{lemma}

Lemma \ref{lem_primal_convergence_plada} establishes that the primal iterates in an ergodic sense. The running-average of the squared stationarity residual (first-order optimality) converges to zero at a rate of $\tilde{\mathcal{O}}\left(1/T\right)$\footnote{The notation $\tilde{\mathcal{O}}(\cdot)$ suppresses all logarithmic factors from the big-$\mathcal{O}$ notation.}:
\begin{align}
&  \frac{1}{T}\sum_{k=0}^{T-1}\| \boldsymbol{\zeta}_{\pp}^{k+1} \|^2 = \mathcal{O}\left(\frac{\log(T)}{T}\right)=\tilde{\mathcal{O}}\left(\frac{1}{T}\right). \notag 
\end{align}

\begin{remark} \label{rmk_rate_iterates_plada}
An immediate consequence of Lemma \ref{lem_one_iter_ppl} and \ref{lem_primal_convergence_plada} is that the squared successive difference of the primal iterates also converges at the same rate:
\begin{align}
&  \frac{1}{T}\sum_{k=0}^{T}\left(\| \x_{k+1} - \x_{k} \|^2 + \| \uu_{k+1} - \uu_{k} \|^2 \right) = \tilde{\mathcal{O}}\left(\frac{1}{T}\right). \notag 
\end{align}
Note that Lemma \ref{lem_primal_convergence_plada} states the convergence in an ergodic sense, which involves averaging over the sequence of iterates or employing a randomized output selection from $T$ iterates. Thus, the primal iterates converge with $\tilde{\mathcal{O}}(1/T)$ in an ergodic sense.
\end{remark}

Next, we establish that the iterates converge to a feasible point. This is achieved by showing that the gap between the dual variables vanishes as $k\to\infty$: $\lim_{k \rightarrow \infty} \|\la_{k} - \m_{k} \| = 0$. 

\begin{lemma} [Primal Feasibility]\label{lem_feasibility_plada}
Let $\left\{\ww_{k} \right\}$ be the sequence generated by Algorithm \ref{alg:plada}. Under Assumptions \ref{assumption_kkt}, \ref{assumption_lipschitz_f}, \ref{assumption_bounded_domain}, \ref{assumption_bounded_multiplier} and \ref{assumption_bounded_subgradient}, the gap between the dual variables vanishes:
$$ \underset{k\rightarrow \infty}{\lim} \| \la_{k} - \m_{k} \| = 0.$$
Consequently, any limit point $\bar{\x}$ of the sequence $\{\x_k\}$ is feasible for problem \eqref{eq:op}, satisfying $ g(\bar{\x}) \leq 0$. 
\end{lemma}

\begin{lemma} [Dual feasibility] \label{lem_dual_feasibility_plada}
Let $\bar{\la}$ be a limit point of the sequence $\{\la_k\}$ generated by Algorithm \ref{alg:plada}. Then, $\bar{\la}$ is feasible for the dual problem of \eqref{eq:op}, satisfying $\bar{\la} \ge \0$.
\end{lemma}

\begin{lemma} [Complementary slackness] \label{lem_complementary_slackness_plada}
Let $(\bar{\x},\bar{\la})$ be a limit point of the sequence $\{(\x_k,\la_k)\}$ generated by Algorithm \ref{alg:plada}. Then, $(\bar{\x},\bar{\la})$ satisfies the complementary slackness for problem of \eqref{eq:op}, i.e., $\bar{\la}^{\top}g(\bar{\x}) = 0$.
\end{lemma}

\begin{theorem} [Convergence to a KKT Point] \label{thm_KKT_plada}
Let $\{\mathbf{w}_k = (\x_k, \uu_k, \z_k, \la_k, \m_k)\}$ be the sequence generated by Algorithm \ref{alg:plada}. Under Assumptions \ref{assumption_kkt}, \ref{assumption_lipschitz_f}, \ref{assumption_bounded_domain}, \ref{assumption_bounded_multiplier} and \ref{assumption_bounded_subgradient}, any limit point $\bar{\mathbf{w}}$ of the sequence $\{\mathbf{w}_k\}$ corresponds to a KKT point of the original problem \eqref{eq:op} as defined in Definition \ref{def_kkt}.
\end{theorem}

\begin{proof}
By Lemmas \ref{lem_primal_convergence_plada}, \ref{lem_feasibility_plada}, \ref{lem_dual_feasibility_plada} and \ref{lem_complementary_slackness_plada}, $\bar{\mathbf{w}}$ satisfies the KKT conditions as defined in Definition \ref{def_kkt}.
\end{proof}

To find the non-asymptotic ergodic rate of convergence, we construct a non-negative auxiliary sequence as
\begin{equation}
    \boldsymbol{\nu}_k := \la_k + \frac{1}{\tau}(\uu_{k+1} - \uu_k). \label{eq:nu}
\end{equation}
Note that by the first-order optimality of $\uu_{k+1}$ for \eqref{eq:u_update},
\begin{equation}
    \uu_k-\tau\la_k \le \uu_{k+1} \iff \la_k + \frac{1}{\tau}(\uu_{k+1} - \uu_k) \ge \0, \notag
\end{equation}
and $\boldsymbol{\nu}_k\ge\0$ for all $k\ge0$. To show the ergodic convergence to a $\epsilon$-KKT solution, define the running average of this non-negative multiplier: $\bar{\boldsymbol{\nu}}_T := \frac{1}{T}\sum_{k=0}^{T-1} \boldsymbol{\nu}_k$.

\begin{lemma} [Non-asymptotic rate for primal stationarity] \label{lem_rate_primal_stationarity_plada}
Let $\{\pp_k := (\x_k,\uu_{k},\z_{k})\}$ be the primal sequence generated by Algorithm \ref{alg:plada} using the non-negative multiplier $\bar{\boldsymbol{\nu}}_T$. Under Assumptions \ref{assumption_kkt}, \ref{assumption_lipschitz_f}, \ref{assumption_bounded_domain} and \ref{assumption_bounded_subgradient}, average primal stationarity residual converges as
\begin{equation}
    \frac{1}{T}\sum_{k=0}^{T-1}\| \boldsymbol{\zeta}_{\pp}^{k+1} \|^2 =  \tilde{\mathcal{O}}\left(\frac{1}{T}\right), \notag 
\end{equation}
where $\boldsymbol{\zeta}_{\pp}^{k+1} :=(\boldsymbol{\zeta}_{\x}^{k+1},\boldsymbol{\zeta}_{\uu}^{k+1},\boldsymbol{\zeta}_{\z}^{k+1}) \in \partial_{\pp} \mathcal{L}_{\alpha\beta} (\ww_{k+1})$.
\end{lemma}

\begin{lemma} [Non-asymptotic rate for primal feasibility] \label{lem_rate_primal_feasibility_plada}
Let $\{\x_k\}$ be the primal sequence generated by Algorithm \ref{alg:plada}. Under Assumptions \ref{assumption_kkt}, \ref{assumption_lipschitz_f}, \ref{assumption_bounded_domain}, \ref{assumption_bounded_multiplier} and \ref{assumption_bounded_subgradient}, average primal feasibility violation converges as
\begin{equation}
\frac{1}{T}\sum_{k=0}^{T-1} \|[g(\x_{k+1})]^{+}\|^2 =  \tilde{\mathcal{O}}\left(\frac{1}{T}\right). \label{eq:primal_feasibility_rate_plada}
\end{equation}
\end{lemma}

\begin{lemma} [Non-asymptotic rate for complementary slackness] \label{lem_rate_complementary_slackness_plada}
Let $\{\x_k, \boldsymbol{\nu}_k\}$ be the sequence generated by Algorithm \ref{alg:plada}. Under Assumptions \ref{assumption_kkt}, \ref{assumption_lipschitz_f}, \ref{assumption_bounded_domain}, \ref{assumption_bounded_multiplier} and \ref{assumption_bounded_subgradient}, average complementary slackness for Definition \ref{def_epsilon-kkt} converges as
\begin{equation}
    \frac{1}{T}\sum_{k=0}^{T-1} \sum_{j=1}^m |\nu_{j,k} g_j(\x_{k+1})| =\tilde{\mathcal{O}}(1/\sqrt{T}).
\end{equation}
\end{lemma}

\begin{theorem} [Non-asymptotic Rate of Convergence] \label{thm_rate_plada}
Under Assumptions \ref{assumption_kkt}, \ref{assumption_lipschitz_f}, \ref{assumption_bounded_domain}, \ref{assumption_bounded_multiplier} and \ref{assumption_bounded_subgradient}, there exists a uniformly-at-random iterate $k\in\{0,\cdots,K-1\}$ from the sequence generated by Algorithm \ref{alg:plada} that is a $\epsilon$-KKT solution to problem \ref{eq:op} on expectation as defined in Definition \ref{def_epsilon-kkt}. The total number of iterations required to achieve this is bounded by $\tilde{\mathcal{O}}(1/\epsilon^2)$.
\end{theorem}

\begin{proof}
By Lemmas \ref{lem_rate_primal_stationarity_plada}, \ref{lem_rate_primal_feasibility_plada} and \ref{lem_rate_complementary_slackness_plada} and the construction of the non-negative multiplier sequence, we have the ergodic convergence of $\epsilon$-KKT residuals at a rate of $\tilde{\mathcal{O}}(1/\sqrt{T})$. Hence, all conditions for an $\epsilon$-KKT solution are met with an overall iteration complexity of $\tilde{\mathcal{O}}(1/\epsilon^2)$. Therefore, if one chooses one of the algorithm iterates uniformly at random, that solution will be $\epsilon$-KKT on expectation.
\end{proof}

\section{Non-convex Continuously Differentiable Constraints} \label{sec:smooth}
In this section, we present our novel form of augmented Lagrangian (Section~\ref{ssec:ppal}), termed Proximal-Perturbed Augmented Lagrangian (PPAL), and propose a single-loop primal-dual algorithm based on it (Section~\ref{ssec:algo}). 

\subsection{Proximal-Perturbed Augmented Lagrangian} \label{ssec:ppal}
Our approach for the smooth case is built upon {\em Proximal-Perturbed Augmented Lagrangian} (PPAL). As in the non-smooth case, we work with the equivalent equality-constrained formulation of problem \eqref{eq:op} and define the PPAL as:
\begin{align} \label{eq:ppal} 
\mathcal{L}_{\rho}(\x,\uu,\z,\la,\m)  = \ell_{\rho}(\x,\uu,\z,\la,\m) +  r(\x),  
\end{align}
where 
\begin{equation}
   \ell_{\rho}(\cdot) :=  f(\x) + \left\langle \la, g(\x) + \uu - \z \right\rangle + \left\langle \m, \z \right\rangle + \frac{\alpha}{2}\| \z \|^2 - \frac{\beta}{2}\| \la - \m \|^2 + \frac{\rho}{2}\| g(\x) + \uu \|^2.  
\end{equation}

Analogous to \eqref{eq:reduced_ppl}, substituting the expression for $\z(\la, \m)$ back into \eqref{eq:ppal} yields the reduced PPAL: 
\begin{align} \label{eq:reduced_ppal}
\mathcal{L}_{\rho}(\x,\uu,\z(\la, \m),\la,\m) 
& = f(\x) + \left\langle \la, g(\x) + \uu  \right\rangle - \frac{1}{2\rho} \| \la - \m \|^2 + \frac{\rho}{2} \| g(\x) + \uu \|^2+ r(\x).
\end{align}

\subsection{Description of Algorithm}
\label{ssec:algo}

\begin{algorithm} [t!]	
\caption{PPAL-based first-order Algorithm (\text{PPALA})}  \label{alg:ppala}
\begin{algorithmic}[1]

\State \textbf{Input:} Initialization $(\x_0,\uu_0,\z_0,\la_0,\m_0)$, and parameters $\alpha > 1$, $\beta \in (0,1)$, $\rho =\frac{\alpha}{1+ \alpha \beta}$, $0 < \eta < \frac{1}{L_{\ell} +   3 \rho M_g^2}$, $0 < \tau < \frac{1}{2\rho}$, and $K$.

\For{$k= 0, 1, \dots, K $} 

\State\label{step1} {Compute  $\x_{k+1}$ by the proximal gradient scheme:} 
\begin{align}
	\x_{k+1} & = \underset{\x \in \mathbb{R}^n}{\mathrm{argmin}} \left\lbrace  \left\langle \nabla_{\x} \ell_{\rho}(\x_k,\uu_k,\z_k,\la_k,\m_k), \x - \x_k \right\rangle + ({1}/{2\eta}) \| \x - \x_k \|^2 + r(\x) \right\rbrace; \notag
\end{align}

\State\label{step1-1} {Compute $\uu_{k+1}$ by the projected gradient descent:}
\begin{equation}
    \uu_{k+1} = \Pi_{\mathbb{R}_+^m}[\uu_{k}  - \tau (\la_{k} + \rho (g(\x_{k+1}) + \uu_{k}))]; \label{eq:u_update_ppala}
\end{equation}

\State\label{step2-1} {Update  the auxiliary multiplier $\m_{k+1}$ by:}
\begin{equation}
  \m_{k+1} = \m_k + \sigma_k (\la_k - \m_k), \  \sigma_k =  \frac{\delta_k}{\| \la_k - \m_k \|^2 + 1}; \label{eq:mu_update_ppala}
\end{equation}

\State\label{step2-2} {Update the multiplier $\la_{k+1}$ by}
$$\la_{k+1}  = \m_{k+1} + \rho (g(\x_{k+1}) + \uu_{k+1});$$

\State\label{step3} {Compute $\z_{k+1}$ by}
$$\z_{k+1} = \frac{1}{\alpha}(\la_{k+1} - \m_{k+1});$$

\EndFor 		
\end{algorithmic} 
\end{algorithm}

We propose a single-loop first-order algorithm based on the properties of our PPAL, which computes a stationary solution to problem \eqref{eq:op}. The complete procedure is detailed in Algorithm \ref{alg:ppala}, where $L_{\ell} := L_f + L_g B_{\lambda} + \rho(L_g B_u + L_g B_g + M_g^2)$. 

Each iteration of PPALA involves updating the primal and dual variables. The primal variable $\x$ is updated inexactly by the {\em proximal gradient mapping} (see e.g., \cite{bolte2014proximal}), which can be rewritten as
\begin{equation}
\label{eq:x_update_ppala}
\x_{k+1} = \text{prox}_{\eta r} \left[ \x_k - \eta \nabla_\x \ell_{\rho}(\x_k,\uu_k,\z_k,\la_k,\m_k)\right].
\end{equation}

Next, the slack variable $\uu$ is upated via projected gradient descent: 
\begin{equation}
\begin{aligned}
\uu_{k+1} 
& = \Pi_{\mathbb{R}_+^m} [\uu_{k} - \tau (\nabla_{\uu} \mathcal{L}_{\rho}(\x_k,\uu_k,\z_k,\la_k,\m_k)]  \notag \\
& = \Pi_{\mathbb{R}_+^m} [\uu_{k} - \tau (\la_{k} + \rho (g(\x_{k+1}) + \uu_{k})]. \notag
\end{aligned}
\end{equation}
Note that we can construct an upper bound $\max_{k\ge1}\{\uu_k\}\le B_g$ from~\eqref{eq:bound_jacobian_gx}, since we have $\|g(\x) \| \leq B_g$ for all feasible solutions $\x$.

The {\em auxiliary} multiplier $\m$ is updated as \eqref{eq:mu_update_ppala} with a diminishing sequence $\delta_k$ satisfying the conditions \eqref{eq:delta_condition}. In particular, we employ the form:  
\begin{equation} \label{eq:delta_k}
\delta_k = \frac{1}{p \cdot k^{q}+1}, \quad \ \frac{2}{3} < q \leq 1, \quad \ p > 0.
\end{equation}

Note that several alternatives are available for the sequence $\{\delta_k\}$ satisfying the conditions in \eqref{eq:delta_condition}. Two popular alternative step sizes are: (i) $\delta_k = \frac{\delta_0}{(k + 1)^q}$, where $\delta_0 > 0$ and $0 < q \leq 1$, and (ii) $\delta_k = \frac{\delta_{k-1}}{1 - b\delta_{k-1}}$, where $\delta_0 \in (0, 1]$ and $b \in (0, 1)$; see e.g., \cite{bertsekas1999nonlinear,scutari2014decomposition} for more possibilities for $\{\delta_k\}$. As we will see in Lemma~\ref{lem_stationarity_ppala} and Corollary~\ref{cor1}, a benefit of~\eqref{eq:delta_k} and choosing $q \in (2/3, 1]$ is that it allows our algorithm to achieve improved complexity bounds compared to $\mathcal{O}(1/\epsilon^3)$ found in existing works. 

Then the multipliers $\la$ and $\z$ are updated in the same manner as Algorithm \ref{alg:plada}.

\subsection{Convergence Guarantees}

In this section, we establish the convergence results of Algorithm \ref{alg:ppala}. We prove that the sequence generated by Algorithm \ref{alg:ppala} converges to a KKT point of problem \eqref{eq:op} as defined in \eqref{eq:def_kkt}. The analysis extends to demonstrating the algorithm's non-asymptotic rate of convergence in ergodic sense. Please find the proofs of each Lemma in Sumpplementary Materials.

\begin{lemma} [Primal Stationarity] \label{lem_stationarity_ppala}
Let $\left\{\ww_{k} \right\}$ be the sequence generated by Algorithm \ref{alg:ppala}, and let $\{\pp_k := (\x_k,\uu_{k},\z_{k})\}$ be the generated primal sequence. Under Assumptions \ref{assumption_kkt}-\ref{assumption_bounded_domain}, the running average of the squared primal stationarity residual converges to zero:
\begin{equation}
    \lim_{T \to \infty} \frac{1}{T}\sum_{k=0}^{T-1} \| \boldsymbol{\zeta}_{\pp}^{k+1}\|^2 =0, \ \ \text{with the rate of} \ \begin{cases} 
    \mathcal{O}\left(\frac{\log(T)}{T}\right) = \widetilde{\mathcal{O}}\left(\frac{1}{T}\right) & \text{if } \ q = 1, \\
    \mathcal{O}\left(\frac{1}{T^{q}}\right) & \text{if } \ 2/3 < q <1, 
    \end{cases}, \label{eq:primal_converge_rate}
\end{equation}
where $\boldsymbol{\zeta}_{\pp}^{k+1} \in \partial_{\pp} \mathcal{L}_{\rho} (\ww_{k+1})$ and $\delta_k = \frac{1}{p \cdot k^q + 1}$. Hence, $\0 \in \nabla f(\bar{\x}) + \partial r(\bar{\x}) + \partial g(\bar{\x})^{\top}\bar{\la}$.
\end{lemma}

Thus, a consequence of Lemma \ref{lem_stationarity_ppala} is that $q=1$ gives the fastest primal convergence rate of Algorithm \ref{alg:plada}.

\begin{corollary} \label{cor1}
Consider the sequence $\{\delta_k\}$ with the best choice of $q=1$ in terms of the primal convergence rate of Algorithm \ref{alg:ppala}, i.e., $\delta_k = \frac{1}{p \cdot k + 1}$. For a given tolerance $\epsilon>0$, the number of
iterations required to reach $\epsilon$-primal stationarity,
$\frac{1}{T}\sum_{k=0}^{T-1}\| \boldsymbol{\zeta}_{\pp}^{k+1} \| \leq \epsilon,$
is upper bounded by  $\widetilde{\mathcal{O}}\left({1}/{\epsilon^2}\right).$ 
\end{corollary}

Note that even with the choice of $2/3 < q < 1$ for the sequence $\{\delta_k\}$, we can derive the complexity bound of ${\mathcal{O}}\left({1}/{\epsilon^{2/q}}\right)$ through a similar analysis. This is still an improved complexity bound compared to the best-known complexity of $\mathcal{O}\left(1/\epsilon^3\right)$.

\begin{remark}
As an immediate consequence of results in Lemma \ref{lem_ppal_decrease_converge} and Lemma \ref{lem_stationarity_ppala}, we also have the result:
$\lim_{T \rightarrow \infty} \frac{1}{T}\sum_{k=0}^{T} \left(\| \x_{k+1} - \x_{k} \|^2 +\| \uu_{k+1} - \uu_{k} \|^2  \right)= 0.$ 
This result implies the following rates of the squared running-average successive differences of primal iterates:
\begin{align}
\quad \frac{1}{T}\sum_{k=0}^{T-1}\left(\| \x_{k+1} - \x_{k} \|^2 + \| \uu_{k+1} - \uu_{k} \|^2 \right) = \begin{cases} 
\mathcal{O}\left(\frac{\log(T)}{T}\right) 
 = \widetilde{\mathcal{O}}\left(\frac{1}{T}\right) & \text{if } \ q = 1, \\
\mathcal{O}\left(\frac{1}{T^{q}}\right) & \text{if } \ \frac{2}{3} < q <1, 
\end{cases} \notag
\end{align}
\end{remark}

It remains to prove that $\lim_{k \rightarrow \infty} \|\la_{k} - \m_{k} \|=0$ to show the feasibility guarantees of our algorithm, which will complete our arguement of obtaining an improved iteration complexity among algorithms solving problem \eqref{eq:op}.  This can be easily achieved by the structural properties of Algorithm \ref{alg:ppala}.

\begin{lemma} [Primal Feasibility] \label{lem_feasibility_ppala}
Let $\left\{\ww_{k} \right\}$ be the sequence generated by Algorithm \ref{alg:ppala}. Under Assumptions \ref{assumption_kkt}-\ref{assumption_bounded_multiplier}, the gap between the dual variables vanishes:
$$ \underset{k\rightarrow \infty}{\lim} \| \la_{k} - \m_{k} \| = 0.$$
Consequently, any limit point $\bar{\x}$ of the sequence $\{\x_k\}$ is feasible for problem \eqref{eq:op}, satisfying $ g(\overline{\x}) \leq \0$. Moreover, defining $\boldsymbol{\zeta}_{{\bf d}}^{k+1} :=(\zeta_{\la}^{k+1}, \zeta_{\m}^{k+1}) = (\0, \frac{1}{\rho}(\la_{k+1} - \m_{k+1})) \in \nabla_{{\bf d}} \mathcal{L}_{\rho} (\ww_{k+1})$, we have the running-average feasibility residual: 
\begin{align}
&  \frac{1}{T}\sum_{k=0}^{T-1}\| \boldsymbol{\zeta}_{{\bf d}}^{k+1} \|^2 = \mathcal{O}\left(\frac{\log(T)}{T}\right) =\widetilde{\mathcal{O}}\left(\frac{1}{T}\right). \label{eq:dual_rate}
\end{align}
\end{lemma}

\begin{lemma} [Dual feasibility] \label{lem_dual_feasibility_ppala}
Let $\bar{\la}$ be a limit point of the sequence $\{\la_k\}$ generated by Algorithm \ref{alg:ppala}. Under Assumptions \ref{assumption_kkt}-\ref{assumption_bounded_multiplier},  $\bar{\la}$ is feasible for the dual problem of \eqref{eq:op}, satisfying $\bar{\la}\ge\0$.
\end{lemma}

\begin{lemma} [Complementary Slackness] \label{lem_complementary_slackness_ppala}
Let $(\bar{\x},\bar{\la})$ be a limit point of the sequence $\{(\x_k,\la_k)\}$ generated by Algorithm \ref{alg:ppala}. Then, $(\bar{\x},\bar{\la})$ satisfies the complementary slackness for problem of \eqref{eq:op}, i.e., $\bar{\la}^{\top}g(\bar{\x}) = 0$.
\end{lemma}

\begin{theorem} [Convergence to a KKT Point] \label{thm_kkt_ppala}
Let $\{\mathbf{w}_k = (\x_k, \uu_k, \z_k, \la_k, \m_k)\}$ be the sequence generated by Algorithm \ref{alg:ppala}. Under Assumptions \ref{assumption_kkt}-\ref{assumption_bounded_multiplier}, any limit point $\bar{\mathbf{w}}$ of the sequence $\{\mathbf{w}_k\}$ corresponds to a KKT point of the original problem \eqref{eq:op} as defined in Definition \ref{def_kkt}.
\end{theorem}

\begin{proof}
By Lemmas \ref{lem_stationarity_ppala}, \ref{lem_feasibility_ppala}, \ref{lem_dual_feasibility_ppala} and \ref{lem_complementary_slackness_ppala}, $\bar{\mathbf{w}}$ satisfies the KKT conditions as defined in Definition \ref{def_kkt}.
\end{proof}

Notably, this eliminates the need for strong regularity assumptions, which is often imposed by several AL-based algorithms \cite{li2021rate,lin2022complexity,lu2022single,sahin2019inexact} to ensure feasibility. For Algorithm \ref{alg:ppala}, we construct a non-negative auxiliary sequence as
\begin{equation}
    \boldsymbol{\nu}_k := \la_k + \la_{k+1}-\m_{k+1} + \left(\frac{1}{\tau}-\rho\right)(\uu_{k+1}-\uu_k). \label{eq:nu_ppala}
\end{equation}
Note that $\boldsymbol{\nu}_k\ge\0$ for all $k\ge0$. By the first order optimality of $\uu_{k+1}$ for \eqref{eq:u_update_ppala},
\begin{equation}
    \uu_k - \tau(\la_k+\rho (g(\x_{k+1})+\uu_k)) \le \uu_{k+1}. \notag
\end{equation}
And by the lambda update \eqref{eq:lambda_update},
\begin{gather}
    \0 \le \uu_{k+1} -\uu_k + \tau(\la_k+ \la_{k+1}-\m_{k+1}-\rho\uu_{k+1} +\rho\uu_k) \notag \\
    \0 \le (\uu_{k+1} -\uu_k)(1-\tau\rho) + \tau(\la_k+ \la_{k+1}-\m_{k+1}) \notag \\
    \0 \le \left(\frac{1}{\tau}-\rho\right)(\uu_{k+1} -\uu_k) + \la_k+ \la_{k+1}-\m_{k+1} = \boldsymbol{\nu}_k. \notag
\end{gather}

\begin{lemma} [Non-asymptotic rate for primal stationarity] \label{lem_rate_primal_stationarity_ppala}
Let $\{\pp_k := (\x_k,\uu_{k},\z_{k})\}$ be the primal sequence generated by Algorithm \ref{alg:ppala} using the non-negative multiplier $\bar{\boldsymbol{\nu}}_T$. Under Assumptions \ref{assumption_kkt}-\ref{assumption_bounded_domain}, average primal stationarity residual converges as
\begin{equation}
    \frac{1}{T}\sum_{k=0}^{T-1}\| \boldsymbol{\zeta}_{\pp}^{k+1} \|^2 = \tilde{\mathcal{O}}\left(\frac{1}{T}\right), \notag 
\end{equation}
where $\boldsymbol{\zeta}_{\pp}^{k+1} :=(\boldsymbol{\zeta}_{\x}^{k+1},\boldsymbol{\zeta}_{\uu}^{k+1},\boldsymbol{\zeta}_{\z}^{k+1}) \in \partial_{\pp} \mathcal{L}_{\alpha\beta} (\ww_{k+1})$.
\end{lemma}

\begin{lemma} [Non-asymptotic rate for primal feasibility] \label{lem_rate_primal_feasibility_ppala}
Let $\{\x_k\}$ be the primal sequence generated by Algorithm \ref{alg:ppala}. Under Assumptions \ref{assumption_kkt}-\ref{assumption_bounded_multiplier}, average primal feasibility violation converges as
\begin{equation}
\frac{1}{T}\sum_{k=0}^{T-1} \|[g(\x_{k+1})]^{+}\|^2 =  \tilde{\mathcal{O}}\left(\frac{1}{T}\right). \label{eq:primal_feasibility_rate_ppala}
\end{equation}
\end{lemma}

\begin{lemma} [Non-asymptotic rate for complementary slackness] \label{lem_rate_complementary_slackness_ppala}
Let $\{\x_k, \boldsymbol{\nu}_k\}$ be the sequence generated by Algorithm \ref{alg:ppala}. Under Assumptions \ref{assumption_kkt}-\ref{assumption_bounded_multiplier}, average complementary slackness for Definition \ref{def_epsilon-kkt} converges as
\begin{equation}
    \frac{1}{T}\sum_{k=0}^{T-1} \sum_{j=1}^m |\nu_{j,k} g_j(\x_{k+1})| =\tilde{\mathcal{O}}(1/\sqrt{T}).
\end{equation}
\end{lemma}

\begin{theorem} [Non-asymptotic Rate of Convergence] \label{thm_rate_ppala}
Under Assumptions \ref{assumption_kkt}-\ref{assumption_bounded_multiplier}, there exists a uniformly-at-random iterate $k\in\{0,\cdots,K-1\}$ from the sequence generated by Algorithm \ref{alg:ppala} that is a $\epsilon$-KKT solution to problem \ref{eq:op} on expectation as defined in Definition \ref{def_epsilon-kkt}. The total number of iterations required to achieve this is bounded by $\tilde{\mathcal{O}}(1/\epsilon^2)$.
\end{theorem}

\begin{proof}
The proof is analogous to that of Theorem \ref{thm_rate_plada} using Lemmas \ref{lem_rate_primal_stationarity_ppala}, \ref{lem_rate_primal_feasibility_ppala} and \ref{lem_rate_complementary_slackness_ppala} and the construction of the non-negative multiplier sequence \eqref{eq:nu_ppala}.
\end{proof}

\section{Numerical Experiments}\label{sec:experiments}
This section presents a comprehensive set of numerical experiments designed to validate the theoretical results and demonstrate the practical advantages of our proposed algorithms. We evaluate our framework on a range of non-convex optimization problems, including those with non-smooth constraints and smooth, highly non-convex constraints. Our goals are twofold: (1) to empirically verify the convergence properties and improved efficiency of our methods, and (2) to benchmark their performance against existing state-of-the-art algorithms. The results confirm the robustness and superior performance of our approach, particularly in large-scale and complex settings. Experimental details including implementation details, datasets and setups are provided in Supplemenary Materials.

\subsection{Classification Problems Under Non-smooth Fairness Constraints}
We first evaluate our proposed algorithm \ref{alg:plada} on real-world datasets with non-convex non-smooth fairness constraints. In Supplementary Materials, we also provide experiments on hyperparameter robustness, dual variables convergence and extension to highly stochastic setting.

\subsubsection{Demographic Parity Constraint} \label{sec:exp_dp}

\begin{figure*}[htbp]
\hspace*{88pt}\makebox[10pt]{Adult}%
\hspace*{130pt}\makebox[10pt]{Bank}%
\hspace*{135pt}\makebox[10pt]{COMPAS}%
\\\makebox[20pt]{\raisebox{60pt}{\rotatebox[origin=c]{90}{Loss}}}%
\subfigure{\includegraphics[width=0.305\textwidth]{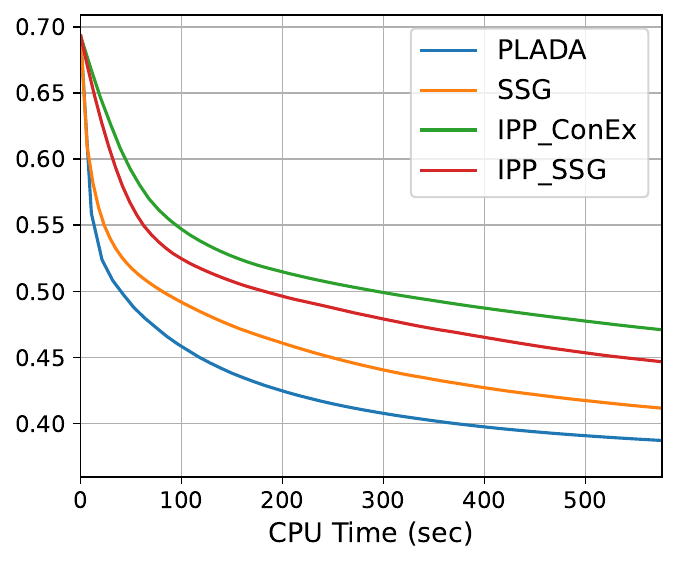} \label{fig:log_dp_a9a_obj}}
\hfill
\subfigure{\includegraphics[width=0.306\textwidth]{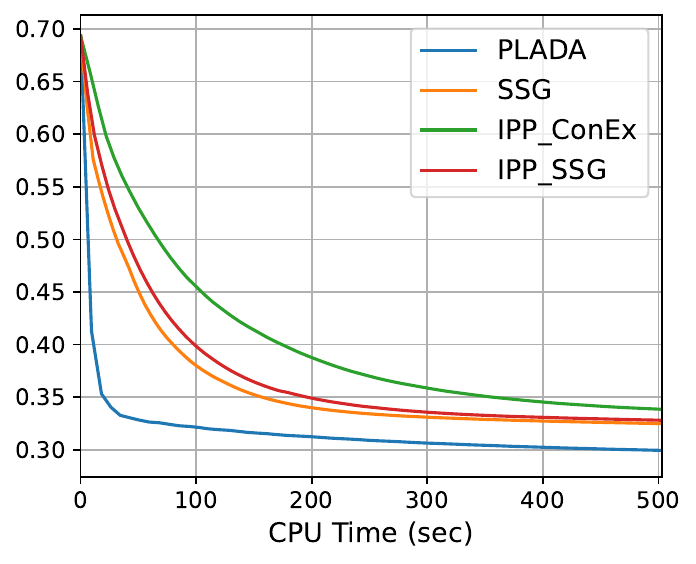} \label{fig:log_dp_bank_obj}}
\hfill
\subfigure{\includegraphics[width=0.303\textwidth]{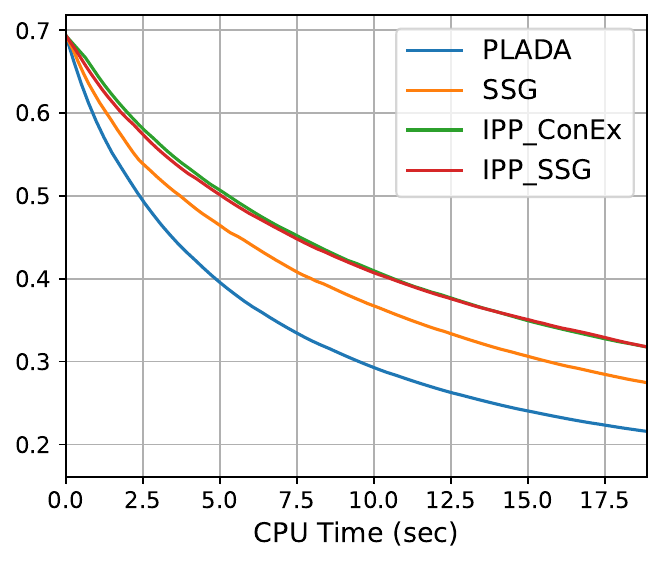} \label{fig:log_dp_compas_obj}}
\makebox[20pt]{\raisebox{55pt}{\rotatebox[origin=c]{90}{DP Violation}}}%
\subfigure{\includegraphics[width=0.303\textwidth]{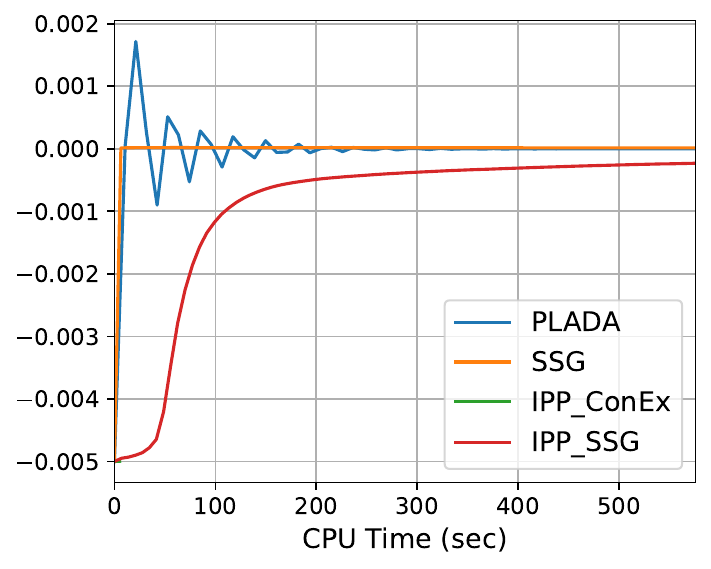} \label{fig:log_dp_a9a_cons}}
\hfill
\subfigure{\includegraphics[width=0.306\textwidth]{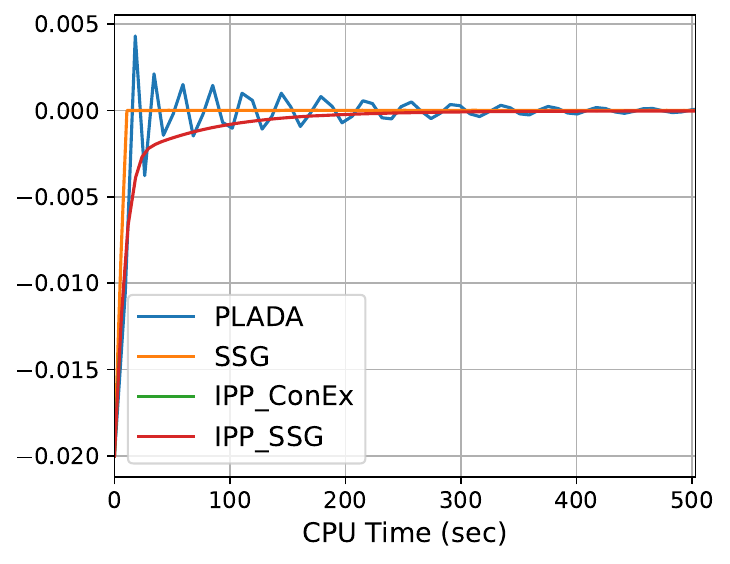} \label{fig:log_dp_bank_cons}}
\hfill
\subfigure{\includegraphics[width=0.303\textwidth]{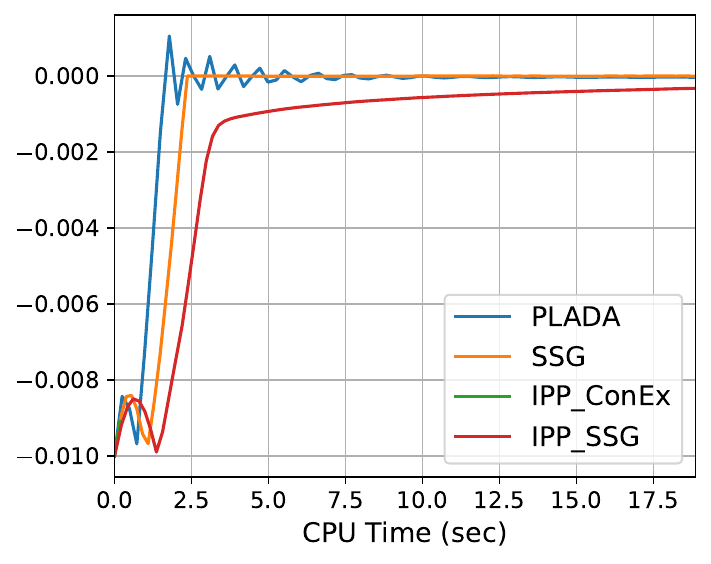} \label{fig:log_dp_compas_cons}}
\makebox[20pt]{\raisebox{60pt}{\rotatebox[origin=c]{90}{Near Stationarity}}}%
\subfigure{\includegraphics[width=0.304\textwidth]{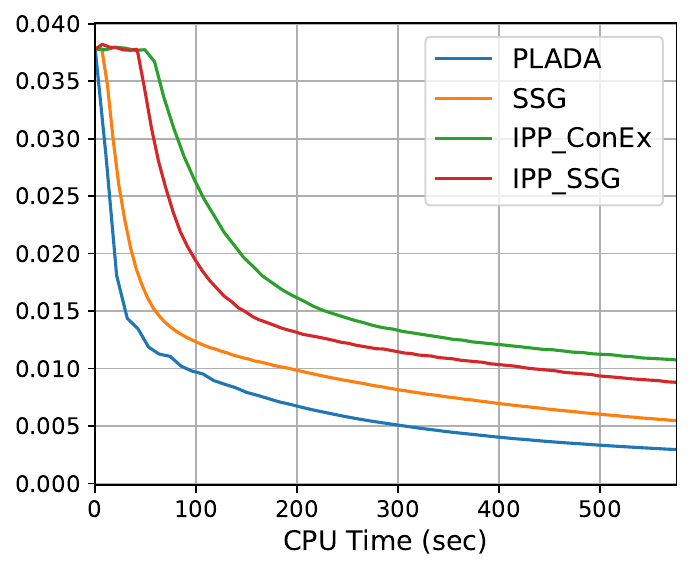} \label{fig:log_dp_a9a_stat}}
\hfill
\subfigure{\includegraphics[width=0.305\textwidth]{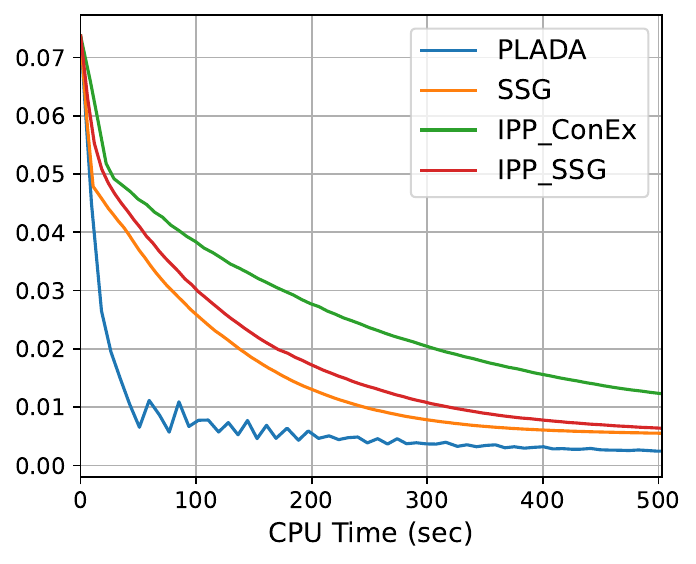} \label{fig:log_dp_bank_stat}}
\hfill
\subfigure{\includegraphics[width=0.304\textwidth]{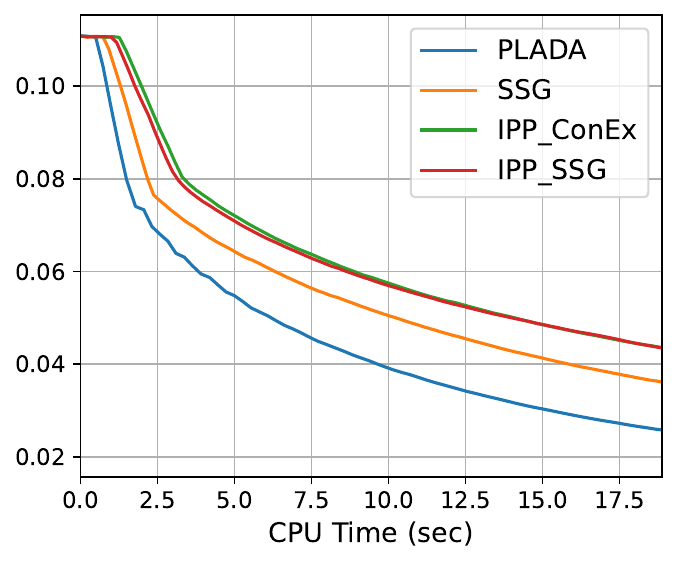} \label{fig:log_dp_compas_stat}}
\caption{Comparison of the performance of PLADA, IPP-ConEx, IPP-SSG and SSG on the logistic loss  \eqref{eq:logistic_loss} with demographic parity (DP) constraint  \eqref{eq:demographic_parity}. The results are presented in terms of their loss values, constraint violation and near stationarity (from top to bottom) on Adult, Bank and COMPAS datasets (from left to right) with respect to CPU time in seconds.}
\label{fig:log_dp}
\end{figure*}

Our first experiment addresses the problem of minimizing the logistic empirical loss:
\begin{equation}
    f(\x) = \frac{1}{N}\sum_{i=1}^N\log(1+e^{-y_i\x^\top {x}_i}), \label{eq:logistic_loss}
\end{equation}
subject to a demographic parity (DP) constraint:
\begin{equation} \label{eq:demographic_parity}
\widehat{\Delta}_D(\x) = \left|\frac{1}{N_p}\sum_{i\in I_p}\sigma(\x^\top {x}_i) - \frac{1}{N_u}\sum_{i\in I_u}\sigma(\x^\top {x}_i)\right|,
\end{equation}
which measures the absolute difference in the positive prediction rates between protected ($I_p$) and unprotected ($I_u$) groups with corresponding sizes of $N_p=|I_p|$ and $N_u=|I_u|$. Equation \eqref{eq:demographic_parity} uses sigmoid $\sigma(\cdot)$ as a surrogate. This results in smooth and convex objective and a weakly convex and non-smooth constraint. Figure \ref{fig:log_dp} depicts the performance of all algorithms across three datasets. The result show that PLADA consistently converges faster with lower loss and smaller constraint violation compared to the benchmark methods. 

\subsubsection{Equalized Odds Constraints} \label{sec:exp_eo}

\begin{figure*}[hbt]
\hspace*{88pt}\makebox[10pt]{Adult}%
\hspace*{130pt}\makebox[10pt]{Bank}%
\hspace*{135pt}\makebox[10pt]{COMPAS}%
\\\makebox[20pt]{\raisebox{60pt}{\rotatebox[origin=c]{90}{Loss}}}%
\subfigure{\includegraphics[width=0.305\textwidth]{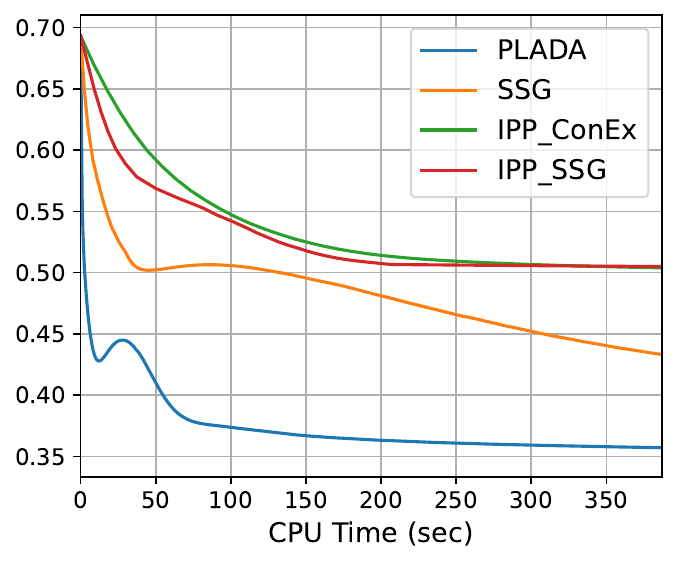} \label{fig:log_eqod_a9a_obj}}
\hfill
\subfigure{\includegraphics[width=0.305\textwidth]{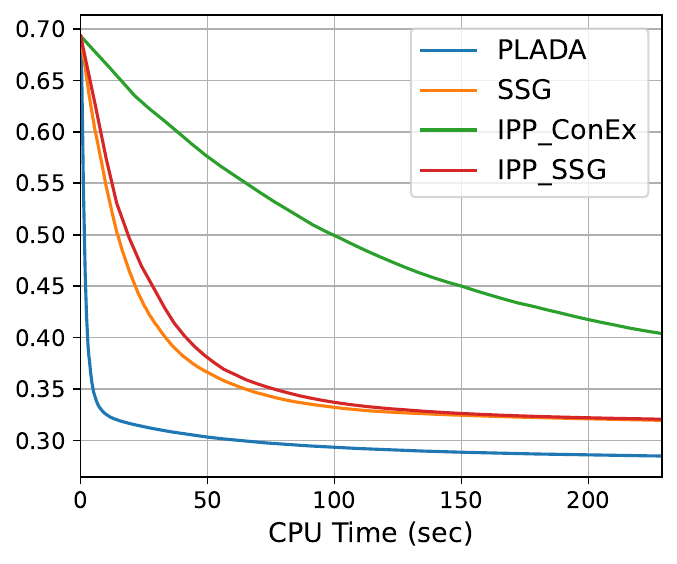} \label{fig:log_eqod_bank_obj}}
\hfill
\subfigure{\includegraphics[width=0.304\textwidth]{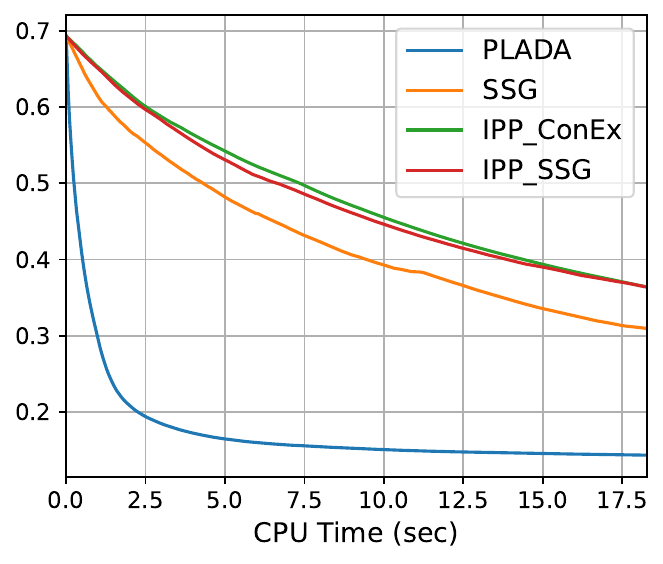} \label{fig:log_eqod_compas_obj}}
\vspace{-1mm}
\makebox[20pt]{\raisebox{55pt}{\rotatebox[origin=c]{90}{EO Violation}}}%
\subfigure{\includegraphics[width=0.304\textwidth]{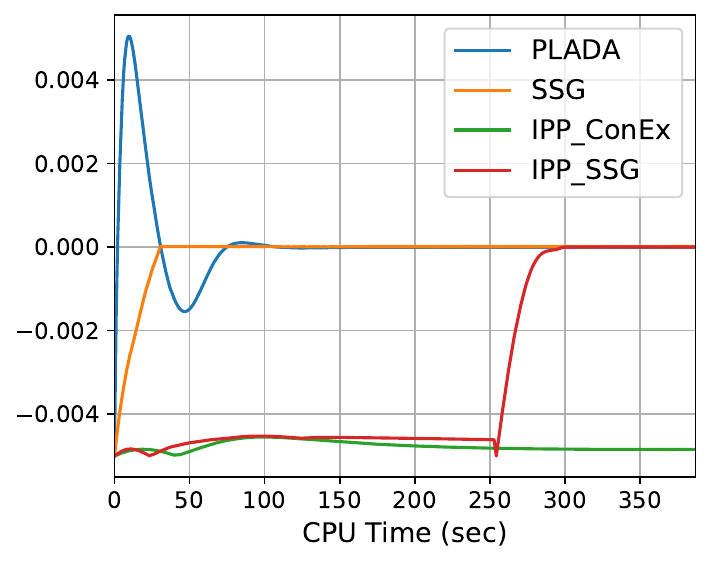} \label{fig:log_eqod_a9a_cons}}
\hfill
\subfigure{\includegraphics[width=0.305\textwidth]{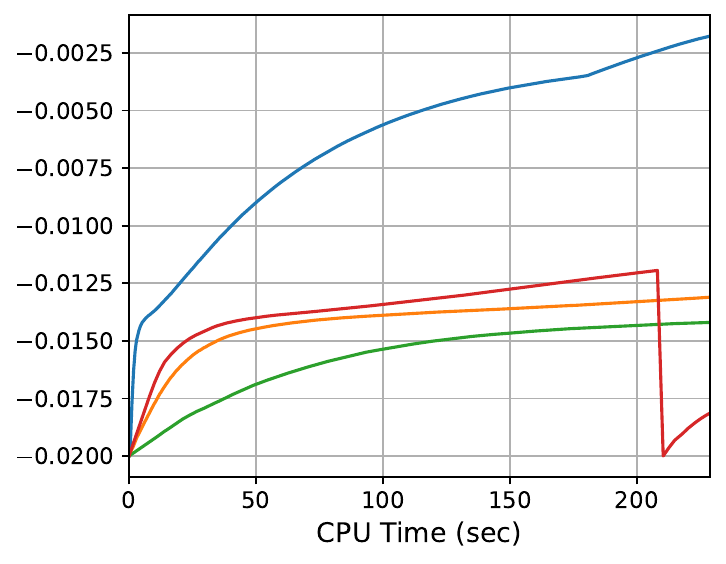} \label{fig:log_eqod_bank_cons}}
\hfill
\subfigure{\includegraphics[width=0.305\textwidth]{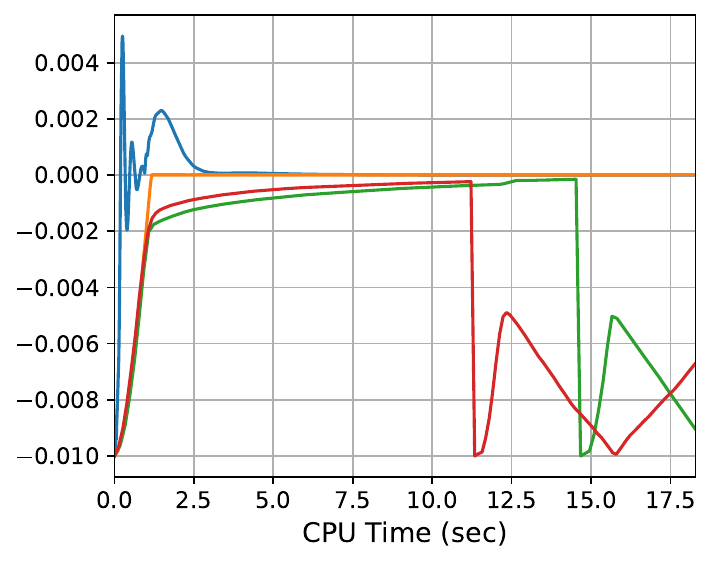} \label{fig:log_eqod_compas_cons}}
\vspace{-1mm}
\makebox[20pt]{\raisebox{60pt}{\rotatebox[origin=c]{90}{Near Stationarity}}}%
\subfigure{\includegraphics[width=0.305\textwidth]{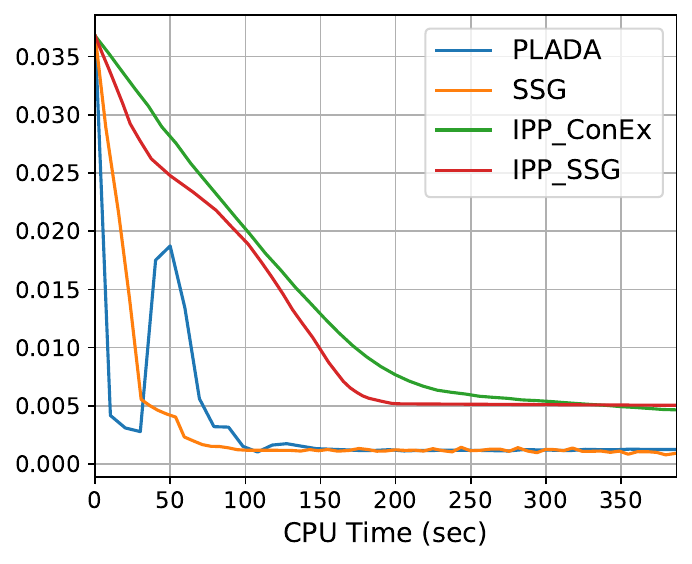} \label{fig:log_eqod_a9a_stat}}
\hfill
\subfigure{\includegraphics[width=0.305\textwidth]{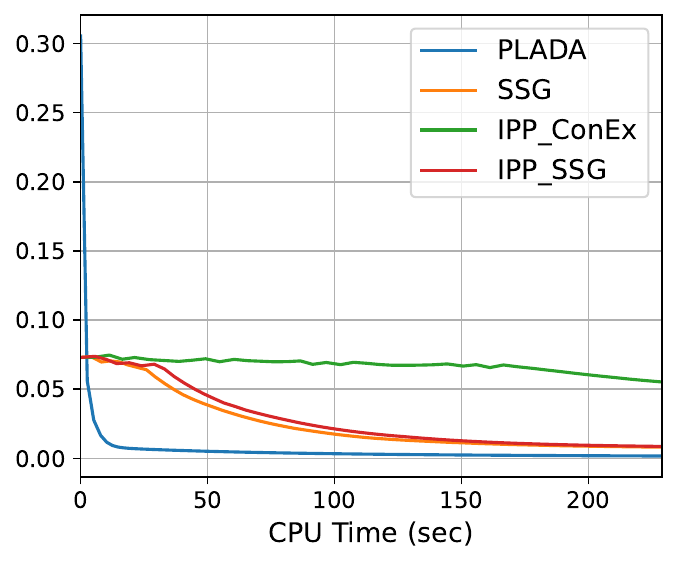} \label{fig:log_eqod_bank_stat}}
\hfill
\subfigure{\includegraphics[width=0.304\textwidth]{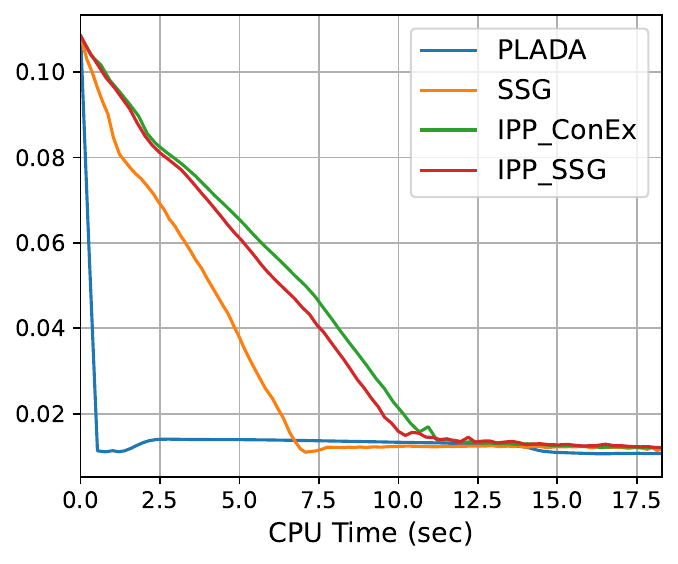} \label{fig:log_eqod_compas_stat}}
\vspace{-2mm}
\caption{Comparison of the performance of PLADA, IPP-ConEx, IPP-SSG and SSG on the logistic loss objective (\ref{eq:logistic_loss}) and the equalized odds (EO) constraint (\ref{eq:equalized_odds}) with respect to CPU time.}
\label{fig:log_eqod}
\end{figure*}

Next, we consider a stricter and more challenging fairness notion: equalized odds (EO). EO constraint \eqref{eq:equalized_odds} requires that the true positive rates and false positive rates are equal across protected and unprotected groups. This results in two separate constraints, which we formulate using a max operator for the benchmark algorithms that only support a single constraint:
\begin{equation}
\begin{aligned}
    \widehat{\Delta}_E(\x) = \max\Biggl(&\left|\frac{1}{N_{pq}}\sum_{i\in I_{pq}}\sigma(\x^\top {x}_i) - \frac{1}{N_{uq}}\sum_{i\in I_{uq}}\sigma(\x^\top {x}_i)\right|, \\
    &\left|\frac{1}{N_{pu}}\sum_{i\in I_{pu}}\sigma(\x^\top {x}_i) - \frac{1}{N_{uu}}\sum_{i\in I_{uu}}\sigma(\x^\top {x}_i)\right|\Biggr). \label{eq:equalized_odds}
\end{aligned}
\end{equation}
A notable advantage of PLADA is its ability to handle multiple constraints by alternatingly optimizing parameters $(\uu,\z,\la, \m)$. Figure \ref{fig:log_eqod} shows that PLADA's advantage is even more pronounced in this more challenging setting.

\subsubsection{Intersectional Group Fairness on Neural Networks} \label{sec:exp_nn_if}

\begin{figure*}[t]
\centering
\subfigure[Error rate]{\includegraphics[width=0.31\textwidth]{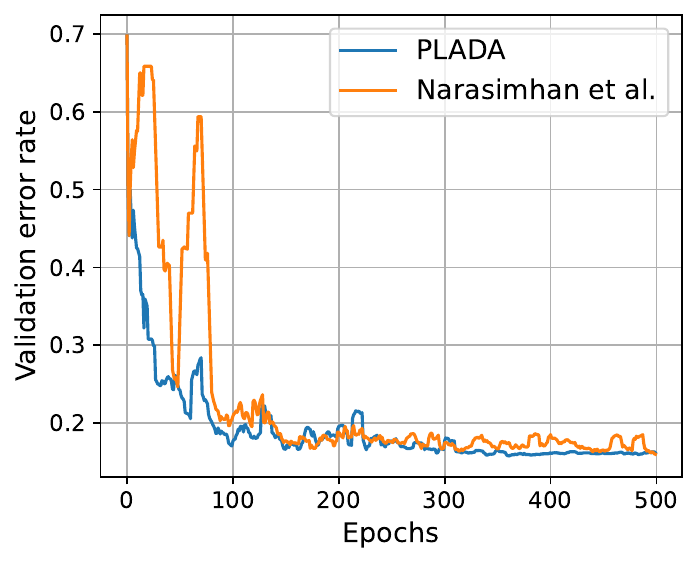} \label{fig:IF_vali_obj}}
\hfill
\subfigure[Average fairness violation  over intersectional groups]{\includegraphics[width=0.325\textwidth]{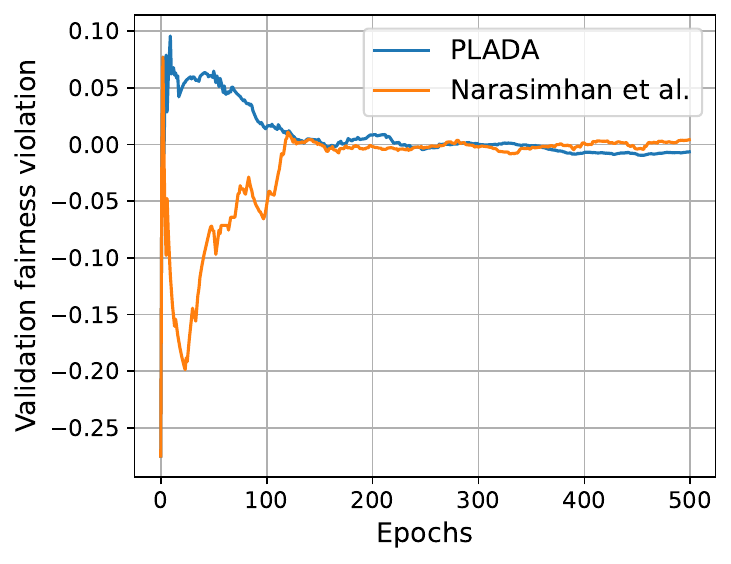} \label{fig:IF_vali_mean_viol}}
\hfill
\subfigure[Maximum fairness violation over intersectional groups]{\includegraphics[width=0.31\textwidth]{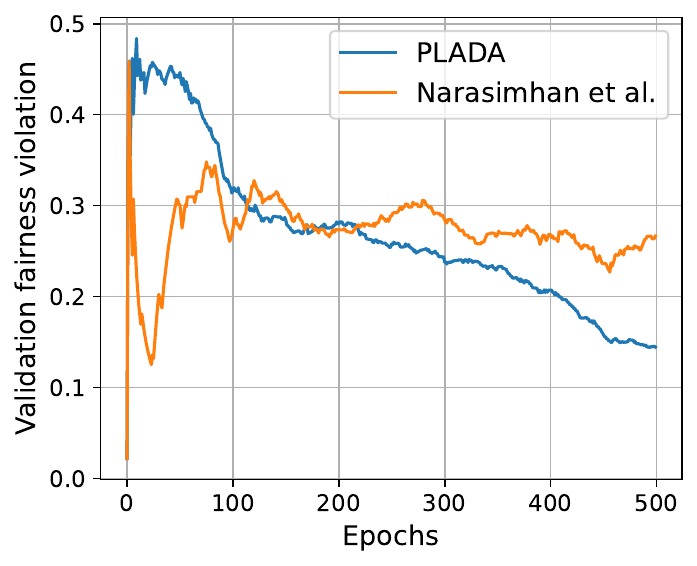} \label{fig:IF_vali_viol}}
\vspace{-2mm}
\caption{Comparison of the validation performance of PLADA and Narasimhan et al., \cite{narasimhan2020approximate} on the intersectional group fairness \eqref{eq:intersectional_fairness} versus Epochs.}
\label{fig:nn_if}
\end{figure*}

The fairness constraint \eqref{eq:intersectional_fairness} is particularly demanding, requiring parity across intersectional groups defined on the Communities and Crime dataset \cite{redmond2009communities}. In particular, the groups are created with ten thresholds on three criteria: the percentages of the Black, Hispanic and Asian populations. Among $10^3$ groups, 535 groups with memberships of more than 1\% of data points are selected. The constraint is formulated as an expectation over the fairness violations for each group:
\begin{equation} \label{eq:intersectional_fairness}
\widehat{\Delta}_I(\x) = \mathbb{E}_G\left[\frac{1}{N_G}\sum_{i\in I_G}\left[1-y_if_\x({x}_i))\right]^+ - \frac{1}{N}\sum_{i=1}^N\left[1-y_if_\x({x}_i)\right]^+\right],
\end{equation}
where $f_\x(\cdot)$ is the neural network classifier, $G$ is a uniformly sampled group, and $[\cdot]^+$ represents a hinge function.

We compare PLADA with the method of \cite{narasimhan2020approximate}, which employs a separate deep neural network to update the Lagrange multipliers. In contrast, PLADA uses a simple, direct update scheme that guarantees the boundedness of the Lagrange multiplier sequence, leading to consistent fairness satisfaction. As shown in Figure \ref{fig:nn_if}, PLADA achieves a lower validation error rate while more effectively reducing fairness violations.

\subsection{Non-convex Multi-class Neyman-Pearson Classification}
This section evaluates the performance of PPALA on a highly non-convex multi-class Neyman-Pearson classification (mNPC) problem using neural networks.

\textbf{Task formulation.} 
The mNPC problem, which aims to minimize the loss for a particular class of interest while ensuring the losses for others remain below given thresholds, is formulated as:
\begin{equation}
\begin{aligned}
\underset{\| \x \| \leq \theta}{\mathrm{min}} & \ \ \frac{1}{\lvert \mathcal{D}_1\rvert} \sum_{j \neq 1}\sum_{ \xi \in \mathcal{D}_1}\phi(f_1(\x_1;\xi)-f_j(\x_j;\xi)) \\
\mathrm{s.\: t.} &  \ \ \frac{1}{\lvert \mathcal{D}_i \rvert}\sum_{j \neq i}\sum_{\xi \in \mathcal{D}_i}\phi(f_i(\x_i; \xi) - f_j(\x_j; \xi )) \leq \kappa_i, \qquad i = 2, \ldots, N, \notag 
\end{aligned}
\end{equation}
where $f_i$ with weights $\x_i$ is a nonlinear classifier for class $i$, $\mathcal{D}_i$ is the corresponding class data, and $\phi$ is a loss function.
\begin{figure}[t!]
\centering
\includegraphics[scale=0.28]{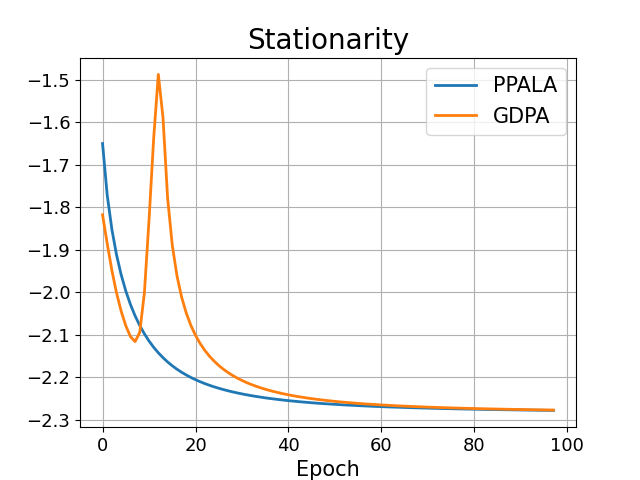}\hspace{-0.15in}
\includegraphics[scale=0.28]{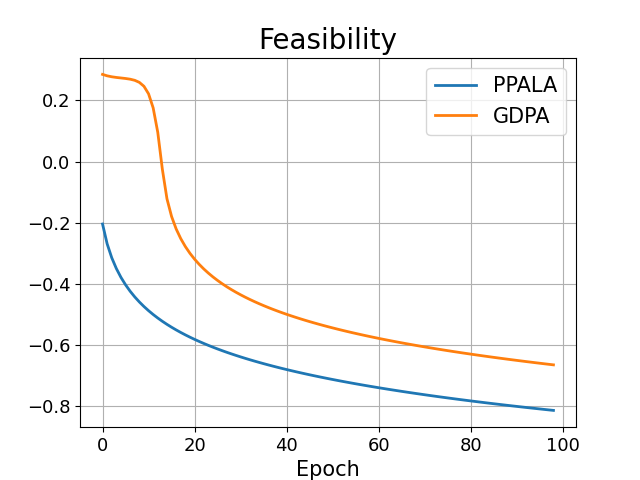}

{{\small (a) Fashion-MNIST}}

\includegraphics[scale=0.28]{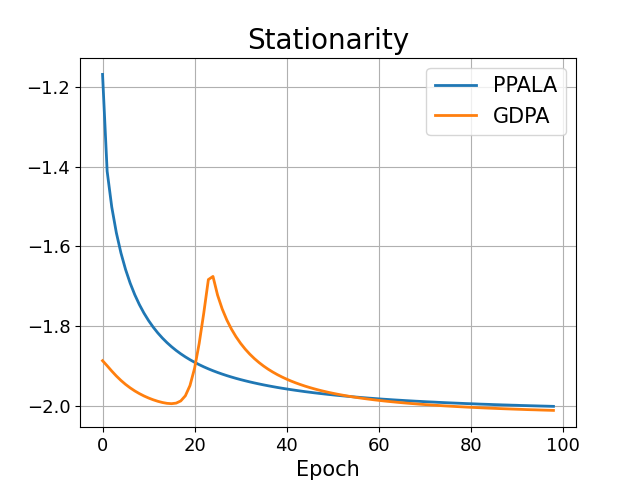}\hspace{-0.15in}
\includegraphics[scale=0.28]{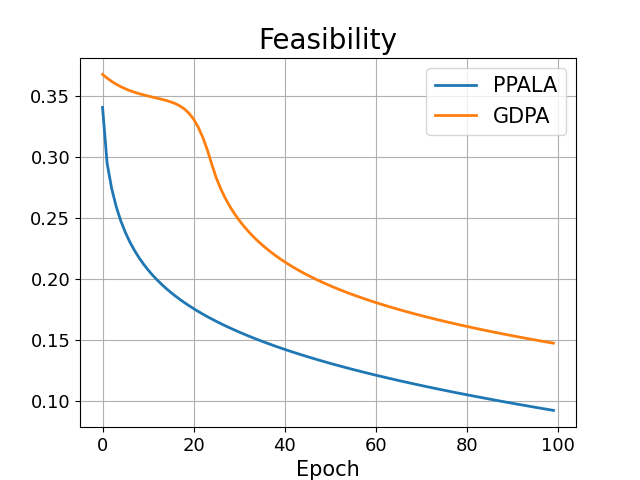}

{{\small (b) CIFAR10}}
\caption{Performance comparison of PPALA and GDPA on Fashion-MNIST and CIFAR10 datasets in terms of obtaining stationarity and feasibility. We see that PPALA provides a consistent reduction of stationarity and feasibility gaps that align with our theoretical expectations. In contrast, GDPA reduces the feasibility gap at a slower rate on Fashion-MNIST and CIFAR10 in our neural network setting.}
  \label{fig:results_mnpc}
\end{figure}

\textbf{Results and discussion.} Figure \ref{fig:results_mnpc} shows that PPALA converges faster compared to GDPA. In particular, The performance advantage is particularly pronounced on the more complex CIFAR-10 dataset. PPALA's robust performance is attributable to its fixed penalty mechanism, which required minimal tuning of parameters across both datasets. In contrast, GDPA exhibited high sensitivity to its penalty parameter update schedule. For instance, we observed that GDPA fails to converge when using large ratio to update the penalty parameter, necessitating careful parameter selection. The gradual update of its penalty parameter hinders GDPA's ability to reduce infeasibility effificiently, while PPALA achieves a fast and consistent reduction in infeasibility with the fixed parameter $\rho=\frac{\alpha}{1+\alpha\beta}$. The results emphasize PPALA's practical advantages in both robustness and computational efficiency when solving more complex problems with highly non-convex constraints.

\section{Conclusions}
In this paper, we have introduced a novel single-loop primal-dual algorithmic framework designed to address non-convex functional constrained optimization problems. A significant contribution of our method is its ability to achieve an improved iteration complexity of $\widetilde{\mathcal{O}}(1/\epsilon^2)$ for computing an $\epsilon$-approximate stationary solution. The proposed algorithmic framework is flexible and robust, capable of effectively handling a variety of non-convex functional constraints. This includes problems with continuously differentiable (smooth) constraints, as well as non-smooth non-convex constraints, which are made tractable through the use of suitable differentiable or sub-differentiable surrogates, particularly relevant in applications like fair classification. Our comprehensive convergence analysis demonstrates that the algorithm ensures a consistent reduction in stationarity and feasibility gaps.

Numerical experiments across diverse non-convex problems, including fairness-constrained classification and multi-class Neyman-Pearson classification (mNPC), consistently demonstrate the algorithm's effectiveness in terms of computational cost and performance. Our algorithm has been shown to outperform related approaches that use regularization techniques and/or standard Lagrangian relaxation, highlighting its superior performance and robustness, especially in large-scale and complex settings.

Future research will explore extending this simple optimization method to stochastic non-convex constrained optimization problems, possibly leveraging variance reduction strategies \cite{cutkosky2019momentum,hashemi2024unified}, which would further broaden its application domain in machine learning and signal processing. Preliminary applications in such highly stochastic settings have already shown promising results in achieving better constraint satisfaction with comparable error rates. Extensions to distributed and federated learning settings is of further interests \cite{kairouz2021advances,das2022faster}. 

\bibliography{references}
\bibliographystyle{ieeetr}

\appendix
\section{Supporting Lemmas for Convergence Analysis of Algorithm \ref{alg:plada}}
The core of our convergence analysis relies on the following technical lemmas, which establish bounds on the iterates and demonstrate the descent properties of the P-Lagrangian. For convenience, we let $\ww_k:=\{(\x_k,\uu_k,\z_k,\la_k,\m_k)\}$ denote the sequence generated by Algorithm \ref{alg:plada}.

\begin{lemma} \label{lem_iterates_relations_plada} 
Let $\{(\x_k,\uu_k,\z_k,\la_k,\m_k)\}$ be the sequence generated by Algorithm \ref{alg:plada}. Then for any $k\ge0$, the following relations hold:
\begin{subequations}
\begin{align}
\| \m_{k+1} - \m_k \|^2 & = ({\sigma_k^2}/{\rho^2}) \|\la_k - \m_k\|^2 \leq {\delta_k^2}/{4}; \label{eq:lem_iter_rel_plada_1} \\
\| \m_{k+1} - \la_k \|^2 & = \left(1 - (\sigma_k/\rho)\right)^2 \| \m_{k} - \la_k  \|^2;  \label{eq:lem_iter_rel_plada_2} \\
\| \la_{k+1}-\la_k \|^2 & \leq 3 \rho^2 M_g^2  \| \x_{k+1} - \x_k \|^2 + 3\rho^2 \| \uu_{k+1} - \uu_k \|^2 + 3(\sigma_k^2/\rho^2) \|\la_k - \m_k\|^2. \label{eq:lem_iter_rel_plada_3} 
\end{align}
\end{subequations}
\end{lemma}

\begin{proof}
From the $\m$-update \eqref{eq:mu_update} and noting that ${a} + {b} \geq 2 \sqrt{{ab}}$ for any ${a,b} \geq 0$, we immediately obtain the relations in \eqref{eq:lem_iter_rel_plada_1}: 
\begin{equation}
\| \m_{k+1} - \m_k \|^2 = \frac{\sigma_k^2}{\rho^2} \| \la_k - \m_k \|^2  \leq \frac{\delta_k^2}{\| \la_k - \m_k \|^2 + 2 + (1/\| \la_k - \m_k \|^2)} \leq \frac{\delta_k^2}{4}. \notag
\end{equation} 
Subtracting $\m_{k+1}$ from $\la_k$ yields
\[ 
\|  \la_k - \m_{k+1} \|= \left\| \la_{k} - \m_{k} -  \frac{\sigma_{k}}{\rho}(\la_{k} - \m_{k}) \right\| = \left(1-\frac{\sigma_{k}}{\rho}\right) \| \la_{k} - \m_{k} \|.
\] 
Squaring both sides of the above inequality yields the relation \eqref{eq:lem_iter_rel_plada_2}. 

By the $\la$-update \eqref{eq:lambda_update}, we have
\begin{equation} \label{eq:lem_iterates_relations_e1}
\begin{aligned}
\| \la_{k+1}- \la_k \| 
& \leq  \|\m_{k+1}-\m_k \| + \rho  \| g(\x_{k+1}) + \uu_{k+1} - g(\x_{k}) - \uu_{k} \| \notag \\
& \leq  \|\m_{k+1}-\m_k \| + \rho M_g \| \x_{k+1} - \x_{k} \| + \rho \| \uu_{k+1} - \uu_{k} \|, 
 \end{aligned}
\end{equation}
which, along with $( a + b + c )^2 \leq 3(a^2 + b^2 + c^2)$ and  \eqref{eq:lem_iter_rel_plada_1}, provides
the relation \eqref{eq:lem_iter_rel_plada_3}. 
\end{proof}

\begin{lemma} [{Approximate Decrease of $\mathcal{L}_{\alpha\beta}$}] \label{lem_one_iter_ppl}
Let $\left\{ \ww_k \right\}$ be the sequence generated by Algorithm \ref{alg:plada}. Under Assumptions \ref{assumption_lipschitz_f}, \ref{assumption_bounded_domain} and \ref{assumption_bounded_subgradient}, the P-Lagrangian $\mathcal{L}_{\alpha\beta}$ \eqref{eq:ppl} satisfies:
\begin{equation}
\mathcal{L}_{\alpha\beta}(\ww_{k+1}) - \mathcal{L}_{\alpha\beta}(\ww_{k}) \leq - C_1 \| \x_{k+1} - \x_k \|^2 - C_2 \| \uu_{k+1} -\uu_{k}\|^2 + \widehat{\delta}_k,
\end{equation}
where $C_1 := \frac{1}{2}\left(\frac{1}{\eta} - {L_f} - 3 \rho M_g^2\right) > 0$, $C_2 := \frac{1}{2}\left( \frac{1}{\tau} -3\rho \right) >0$, and $\widehat{\delta}_k := \frac{\delta_k^2}{2{\rho}} + \frac{\delta_k}{\rho}$.
\end{lemma}

\begin{proof}
First, note that 
\begin{equation}
\begin{aligned} 
   \mathcal{L}_{\alpha\beta}(\x_k,\uu_k,\z_k,\la_k,\m_k)    
   & = f(\x_k)  + \left\langle \la_k, g(\x_k) + \uu_k \right\rangle  - \left\langle \la_k - \m_k,\z_k \right\rangle \\
   & \quad + \frac{\alpha}{2} \| \z_k \|^2 - \frac{\beta}{2} \| \la_k - \m_k \|^2 + r(\x) \\   
   & = f(\x_k)  + \left\langle \la_k, g(\x_k) + \uu_k \right\rangle - \frac{1}{2\rho} \| \la_k - \m_k \|^2 + r(\x) \\   
   &  = \mathcal{L}_{\alpha\beta}(\x_k,\uu_k,\widehat{\z}(\la_k,\m_k),\la_k,\m_k),  \notag
\end{aligned}
\end{equation}
where $\rho = \alpha/(1+\alpha\beta)$, and thus 
\[
\mathcal{L}_{\alpha\beta}(\x_{k+1},\uu_{k+1},\z_k,\la_k,\m_k) = \mathcal{L}_{\alpha\beta}(\x_{k+1},\uu_{k+1},\widehat{\z}(\la_k,\m_k),\la_k,\m_k).
\]

Then the difference of two successive sequences of $ \mathcal{L}_{\alpha\beta}$ can be divided into two parts:
\begin{equation}\label{eq:lem_one_iter_ppl1}
\begin{aligned}
   & \mathcal{L}_{\alpha\beta}(\x_{k+1},\uu_{k+1},\z_{k+1},\la_{k+1},\m_{k+1}) - \mathcal{L}_{\alpha\beta}(\x_k,\uu_k,\z_k,\la_k,\m_k)    \\
   & = \left[  \mathcal{L}_{\alpha\beta}(\x_{k+1},\uu_{k+1},\z_k,\la_k,\m_k) - \mathcal{L}_{\alpha\beta}(\x_k,\uu_k,\z_k,\la_k,\m_k) \right]  \\
   & \quad+ \left[ \mathcal{L}_{\alpha\beta}(\x_{k+1},\uu_{k+1},\widehat{\z}(\la_{k+1},\m_{k+1}),\la_{k+1},\m_{k+1}) - \mathcal{L}_{\alpha\beta}(\x_{k+1},\uu_{k+1},\widehat{\z}(\la_k,\m_k),\la_k,\m_k) \right].  
\end{aligned}
\end{equation}
Consider the first part \eqref{eq:lem_one_iter_ppl1}. Since $\x_{k+1}$ and $\uu_{k+1}$ are the results of the subproblems \eqref{eq:x_update}
 and \eqref{eq:u_update}, respectively, we have that for any $\x \in \mathcal{X}$ and for any $u \in U$,
\begin{equation}
\begin{aligned}
& \left\langle \nabla f(\x_k), \x_{k+1} - \x \right\rangle + \left\langle \la_k, g(\x_{k+1}) - g(\x) \right\rangle \notag \\
& \quad +\frac{1}{2\eta} \left(\| \x_{k+1} - \x_k \|^2 - \| \x - \x_k \|^2\right) + r(\x_{k+1})-r(\x)\leq 0,  \label{eq:lem_one_iter_ppl_1}
\end{aligned}
\end{equation}
and 
\begin{equation} \label{eq:lem_one_iter_ppl_2}
\begin{aligned} 
\left\langle \nabla_{u} \mathcal{L}_{\alpha\beta}(\ww_{k}), \uu_{k+1} -u \right\rangle + \frac{1}{2\tau} ( \|\uu_{k+1} - \uu_k \|^2 - \|\uu - \uu_k \|^2) \leq 0. 
\end{aligned}
\end{equation} 
By taking $\x = \x_{k}$ in \eqref{eq:lem_one_iter_ppl_1}, $u = \uu_{k}$ in \eqref{eq:lem_one_iter_ppl_2}, and using $\nabla_{u} \mathcal{L}_{\alpha\beta}(\ww_{k}) = \la_k$, we have
 \begin{align}
\left\langle \nabla f(\x_k), \x_{k+1} - \x_k \right\rangle + \left\langle \la_k, g(\x_{k+1}) - g(\x_k) \right\rangle + r(\x_{k+1})-r(\x) \leq -\frac{1}{2\eta} \| \x_{k+1} - \x_k \|^2, \notag \label{eq:lem_one_iter_ppl_3}
\end{align}
and 
\begin{equation} \label{eq:lem_one_iter_ppl_4}
\left\langle \la_k , \uu_{k+1} -\uu_k \right\rangle \leq - \frac{1}{2\tau} \| \uu_{k+1} - \uu_k \|^2. \notag
\end{equation}   
By adding and subtracting the term $\left\langle \nabla  f(\x_k), \x_{k+1} - \x_k \right\rangle$, we obtain
\begin{equation}
\begin{aligned}
& \mathcal{L}_{\alpha\beta}(\x_{k+1},\uu_{k+1},\z_{k},\la_{k},\m_{k}) - \mathcal{L}_{\alpha\beta}(\x_k,\uu_{k},\z_{k},\la_{k},\m_{k}) \\
& = \left[f(\x_{k+1}) + \left\langle \la_k, g(\x_{k+1}) + \uu_{k+1} \right\rangle + r(\x_{k+1}) \right]  
- \left[f(\x_{k}) + \left\langle \la_k, g(\x_{k}) + \uu_{k} \right\rangle +r(\x_k) \right] \\
& = \left\langle \la_k, g(\x_{k+1}) - g(\x_{k}) \right\rangle + \left\langle \la_k, \uu_{k+1} - \uu_{k}\right\rangle 
+ \left[f(\x_{k+1}) - f(\x_{k}) \right] + [r(\x_{k+1})-r(\x_k)] \\
& = \left[ \left\langle \nabla f(\x_k), \x_{k+1} - \x_k \right\rangle + \left\langle \la_k, g(\x_{k+1}) - g(\x_{k}) \right\rangle + r(\x_{k+1})-r(\x_k)\right] \\
& \quad \ 
+ \left[  f(\x_{k+1}) - f(\x_{k}) - \left\langle \nabla f(\x_k), \x_{k+1} - \x_k \right\rangle  \right] 
 + \left\langle \la_k, \uu_{k+1} - \uu_{k} \right\rangle \\
 &  \leq 
- \frac{1}{2}\left(\frac{1}{\eta} - L_f \right) \| \x_{k+1} - \x_{k} \|^2 - \frac{1}{2\tau} \| \uu_{k+1} -\uu_{k}\|^2. 
\label{eq:lem_one_iter_ppl_p1}
\end{aligned}
\end{equation}

Next, we derive an upper bound for the second part. We start by noting that
\begin{equation} 
\begin{aligned}
& \mathcal{L}_{\alpha\beta}(\x_{k+1},\uu_{k+1},\widehat{\z}(\la_{k+1},\m_{k+1}),\la_{k+1},\m_{k+1}) - \mathcal{L}_{\alpha\beta}(\x_{k+1},\uu_{k+1},\widehat{\z}(\la_{k},\m_{k}),\la_k,\m_k)           \\ 
& = \frac{1}{\rho} \left\langle \la_{k+1} - \la_{k}, g(\x_{k+1}) + \uu_{k+1} \right\rangle
- \frac{1}{2\rho}  \left(\| \la_{k+1} - \m_{k+1} \|^2 -  \| \la_k - \m_k \|^2\right). \notag
\end{aligned}
\end{equation}
Using the facts that $g(\x_{k+1}) + \uu_{k+1} = \frac{1}{\rho}(\la_{k+1}-\m_{k+1})$ and $\left\langle a, b \right\rangle =\frac{1}{2} \| a \|^2 + \frac{1}{2} \| b \|^2 - \frac{1}{2} \| a - b \|^2$ for any $a,b \in \mathbb{R}^m$, we have  
$$\frac{1}{\rho}\left\langle \la_{k+1} - \la_{k}, \la_{k+1} - \m_{k+1}  \right\rangle  
 = \frac{1}{2\rho} \left(\| \la_{k+1}-\la_k \|^2 +  \| \la_{k+1} -\m_{k+1} \|^2 - \| \m_{k+1} -\la_k \|^2 \right).$$ 
Hence,
\begin{equation}
\begin{aligned}
& \mathcal{L}_{\alpha\beta}(\x_{k+1},\uu_{k+1},\widehat{\z}(\la_{k+1},\m_{k+1}),\la_{k+1},\m_{k+1}) - \mathcal{L}_{\alpha\beta}(\x_{k+1},\uu_{k+1},\widehat{\z}(\la_{k},\m_{k}),\la_k,\m_k) \\
& \overset{(a)}{\leq}
\frac{1}{2\rho}  \left( 3\rho^2 M_g^2 \| \x_{k+1} - \x_k \|^2 + 3\rho^2 \| \uu_{k+1} - \uu_{k} \|^2 + 3\sigma_k^2 \|\la_k - \m_k\|^2 \right) \\
& \quad + \frac{1}{2\rho}\left( 1- (1 - {\sigma_k})^2\right) \| \la_k - \m_k \|^2 \\
& = \frac{1}{2} \left( 3\rho M_g^2 \| \x_{k+1} - \x_k \|^2 + 3\rho \| \uu_{k+1} - \uu_{k} \|^2  \right)  + \frac{3\sigma_k^2}{2\rho} \| \la_k - \m_k \|^2 \\
& \quad + \frac{1}{2\rho}\left( {2\sigma_k} - {\sigma_k^2}\right) \| \la_k - \m_k \|^2 \\
&  \overset{(b)}{\leq}
\frac{1}{2} \left( 3\rho M_g^2 \| \x_{k+1} - \x_k \|^2 + 3\rho \| \uu_{k+1} - \uu_{k} \|^2  \right) + \frac{2\delta_k^2 + \delta_k}{\rho}, \label{eq:lem_one_iter_ppl_p2}
\end{aligned}	
\end{equation}
where $(a)$ is from \eqref{eq:lem_iter_rel_plada_2} and \eqref{eq:lem_iter_rel_plada_3}, and $(b)$ holds by
$ {\sigma_k}\| \la_k - \m_k \|^2 \leq \frac{\delta_k}{1+ ({1}/{\| \la_k - \m_k \|^2})} \leq \delta_k$. 
Combining  \eqref{eq:lem_one_iter_ppl_p1} and \eqref{eq:lem_one_iter_ppl_p2} yields the desired result: 
\begin{equation}
\begin{aligned} 
 & \mathcal{L}_{\alpha\beta}(\ww_{k+1}) -\mathcal{L}_{\alpha\beta}(\ww_k) \notag \\ & \leq 
 - \frac{1}{2}\left(\frac{1}{\eta} - L_f -  3 \rho M_g^2 \right) \| \x_{k+1} - \x_k \|^2  - \frac{1}{2}\left( \frac{1}{\tau} -3\rho \right) \| \uu_{k+1} -\uu_{k}\|^2 + \frac{2\delta_k^2 + \delta_k}{\rho}, \notag
\end{aligned}
\end{equation}
which completes the proof.
\end{proof}

\begin{lemma} [{Subgradient Error Bound}] \label{lem_upper_bound_primal_gradient_plada}
Let $\{\ww_k\}$ be the sequence generated by Algorithm \ref{alg:plada}, and let $\{\pp_k := (\x_k,\uu_{k},\z_{k})\}$ be the primal sequence. Under Assumptions \ref{assumption_lipschitz_f} and \ref{assumption_bounded_domain}, there exists a constant $d_1 > 0$ such that for the primal subgradient $\boldsymbol{\zeta}_{\pp}^{k+1}:=(\zeta_{\x}^{k+1},\zeta_{\uu}^{k+1},\0) \in \partial_{\pp} \mathcal{L}_{\alpha\beta} (\ww_{k+1})$,
\begin{align} 
\| \boldsymbol{\zeta}_{\pp}^{k+1} \| \leq 
d_1 \left( \| \x_{k+1} - \x_k \| + \| \uu_{k+1} -\uu_k \| \right) +  (M_g + 1) \delta_k, \notag
\end{align}
where
$$d_1 = \max \{ L_f + 1/\eta + \rho (M_g^2  + M_g),  \ \rho( M_g +1) + {1}/{\tau} \}.$$
\end{lemma}

\begin{proof}
Writing down the optimality condition for the update of $\x_{k+1}$ in \eqref{eq:x_update}, we have 
\begin{equation} \label{eq:subgradient_for_x_1}
 0 \in \nabla f(\x_k) + \partial g(\x_{k+1})^{\top} \la_{k} + \frac{1}{\eta}(\x_{k+1} - \x_k) + v_{k+1}, \ \ v_{k+1} \in \partial r(\x_{k+1})
\end{equation}
Using the subdifferential calculus rules, we have
\begin{align} 
\nabla f(\x_{k+1}) + \partial g(\x_{k+1})^{\top} \la_{k+1} + v_{k+1} \in \partial_\x \mathcal{L}_{\alpha\beta}(\ww_{k+1})  \label{eq:subgradient_for_x_2}
\end{align}
By defining the quantity
\begin{align}
\boldsymbol{\zeta}_{\x}^{k+1}  =  \nabla f(\x_{k+1}) - \nabla f(\x_k) + \partial g(\x_{k+1})^{\top} (\la_{k+1} - \la_{k}) - \frac{1}{\eta}(\x_{k+1} - \x_k) 
\end{align}
and using \eqref{eq:subgradient_for_x_1} and \eqref{eq:subgradient_for_x_2}, we obtain that 
$\boldsymbol{\zeta}_{\x}^{k+1} \in \partial_{\x} \mathcal{L}_{\alpha\beta}(\ww_{k+1}). $

Next, define the quantity
\[
\boldsymbol{\zeta}_{\uu}^{k+1} := \uu_{k+1} - \Pi_{U}[\uu_{k+1}  - \la_{k+1} ],
\]
which is equivalent to the {\em projected gradient} of $\mathcal{L}_{\alpha\beta}$ in $\uu$. It is a measure of optimality for the update of $\uu_{k+1}$  \cite{nesterov2012efficiency}:
\begin{equation}
\begin{aligned}                    
\widetilde{\nabla}_{\uu}\mathcal{L}_{\alpha\beta}(\ww_{k+1}) 
& := \uu_{k+1} - \underset{v \in U}{\mathrm{argmin}}\left\lbrace \left\langle \nabla_{\uu}\mathcal{L}_{\alpha\beta}(\ww_{k+1}), v - \uu_{k+1} \right\rangle + \frac{1}{2}\| v - \uu_{k+1} \|^2\right\rbrace  \notag \\ 
& \; = \uu_{k+1} - \widetilde{\uu}_{k+1}. \notag
\end{aligned}
\end{equation}
where we define $\widetilde{\uu}_{k+1}:={\mathrm{argmin}}_{v \in U}\left\lbrace \left\langle \nabla_{\uu}\mathcal{L}_{\alpha\beta}(\ww_{k+1}), v - \uu_{k+1} \right\rangle + \frac{1}{2}\| v - \uu_{k+1} \|^2\right\rbrace$.

From the update of $\z_{k+1}$ in \eqref{eq:z_update}, we have
\begin{align}               
\nabla_{\z}\mathcal{L}_{\alpha\beta}(\ww_{k+1})       & = - (\la_{k+1} -\mu_{k+1}) + \alpha \z_{k+1} = 0. \notag 
\end{align}
Hence, we obtain 
\begin{equation} \label{eq:sub_gradient_plada}
\boldsymbol{\zeta}_{\mathbf{p}}^{k+1} := \begin{pmatrix}
\boldsymbol{\zeta}_{\x}^{k+1}  \\
\boldsymbol{\zeta}_{\uu}^{k+1} \\
0  \\
\end{pmatrix}
\quad \text{where} \quad
\begin{pmatrix}
\boldsymbol{\zeta}_{\x}^{k+1} & \in  \partial_{\x}\mathcal{L}_{\alpha\beta}(\x_{k+1},\uu_{k+1},\z_{k+1},\la_{k+1},\mu_{k+1}) \\
\boldsymbol{\zeta}_{\uu}^{k+1} & = \widetilde{\nabla}_{\uu}\mathcal{L}_{\alpha\beta}(\x_{k+1},\uu_{k+1},\z_{k+1},\la_{k+1},\mu_{k+1})   \\ 
0 & =  \nabla_{\z}\mathcal{L}_{\alpha\beta}(\x_{k+1},\uu_{k+1},\z_{k+1},\la_{k+1},\mu_{k+1})  
\end{pmatrix}. \notag
\end{equation}

We derive an upper estimate for $\boldsymbol{\zeta}_{\mathbf{p}}^{k+1}$. A direct calculation gives
\begin{equation}
\begin{aligned}
\| \boldsymbol{\zeta}_{\x}^{k+1} \|
& \leq  \| \nabla f(\x_{k+1}) -\nabla f(\x_k) \| + (1/{\eta}) \| \x_k - \x_{k+1} \| +  \| \partial g(\x_{k+1})\| \| \la_{k+1} -\la_{k} \| \\
& \leq  (L_f + 1/\eta) \| \x_{k+1} - \x_{k} \| +  M_g \| \la_{k+1} -\la_{k} \| \\
& \leq (L_f + 1/\eta) \| \x_{k+1} - \x_{k} \| + \rho M_g^2\| \x_{k+1} - \x_k \| + \rho M_g \| \uu_{k+1} - \uu_k \| + M_g \delta_k \\
& \leq (L_f + 1/\eta + \rho M_g^2)\| \x_{k+1} - \x_{k} \| + \rho M_g \| \uu_{k+1} - \uu_k \| + M_g \delta_k 
\label{eq:sub_gradient_x_plada}
\end{aligned}
\end{equation}

Next, we estimate an upper bound for the component $\boldsymbol{\zeta}_{\uu}^{k+1}$. 
The first-order optimality condition implies that 
\begin{equation} \label{eq:lem_subgradient_u_1}
 \left\langle \nabla_{\uu}\mathcal{L}_{\alpha\beta}(\uu_{k+1}) + (\widetilde{\uu}_{k+1} - \uu_{k+1}),  \uu - \widetilde{\uu}_{k+1} \right\rangle \geq 0. 
\end{equation}
Here, $\nabla_{\uu}\mathcal{L}_{\alpha\beta}(\ww_{k+1})$ is denoted by $\nabla_{\uu}\mathcal{L}_{\alpha\beta}(\uu_{k+1})$. By the definition $\uu_{k+1}$ in \eqref{eq:u_update}, we have 
\begin{equation} \label{eq:lem_subgradient_u_2}
\left\langle \nabla_{\uu}\mathcal{L}_{\alpha\beta}(\uu_{k}) + \frac{1}{\tau}(\uu_{k+1} - \uu_{k}),  {\uu} - \uu_{k+1} \right\rangle \geq 0, 
\end{equation}
where $\nabla_{\uu}\mathcal{L}_{\alpha\beta}(\uu_{k})= \nabla_{\uu}\mathcal{L}_{\alpha\beta}(\x_{k},\uu_{k},\z_{k},\la_{k},\mu_{k})$ for simplicity. Combining \eqref{eq:lem_subgradient_u_1} and \eqref{eq:lem_subgradient_u_2}, with settings $\uu = \uu_{k+1}$ in \eqref{eq:lem_subgradient_u_1} and $\uu = \widetilde{\uu}_{k+1}$ in \eqref{eq:lem_subgradient_u_2}, yields
\begin{equation} 
\left\langle  \nabla_{\uu}\mathcal{L}_{\alpha\beta}(\uu_{k}) - \nabla_{\uu}\mathcal{L}_{\alpha\beta}(\uu_{k+1}) + \frac{1}{\tau}(\uu_{k+1} - \uu_{k}) - (\widetilde{\uu}_{k+1} - \uu_{k+1}),  \widetilde{\uu}_{k+1} - \uu_{k+1} \right\rangle \geq 0, \notag 
\end{equation}
equivalently,  
$$\left\langle  \nabla_{\uu}\mathcal{L}_{\alpha\beta}(\uu_{k}) - \nabla_{\uu}\mathcal{L}_{\alpha\beta}(\uu_{k+1}) + \frac{1}{\tau}(\uu_{k+1} - \uu_{k}),  \widetilde{\uu}_{k+1} - \uu_{k+1} \right\rangle \geq \| \widetilde{\uu}_{k+1} - \uu_{k+1}\|^2.$$
By applying the Cauchy-Schwarz inequality and triangle inequality yields
\begin{equation} 
\left( \|\nabla_{\uu}\mathcal{L}_{\alpha\beta}(\uu_{k}) - \nabla_{\uu}\mathcal{L}_{\alpha\beta}(\uu_{k+1})\| + \frac{1}{\tau} \| \uu_{k+1} - \uu_{k}\|\right)  \| \widetilde{\uu}_{k+1} - \uu_{k+1} \| \geq \| \widetilde{\uu}_{k+1} - \uu_{k+1}\|^2 \notag
\end{equation}
and 
\begin{align}
\|\nabla_{\uu}\mathcal{L}_{\alpha\beta}(\uu_{k}) - \nabla_{\uu}\mathcal{L}_{\alpha\beta}(\uu_{k+1})\| 
&\leq \| \la_{k} - \la_{k+1} \| \notag \\
& \leq  \rho M_g\| \x_{k+1} - \x_k \| + \rho\| \uu_{k+1} - \uu_k \| + \delta_k . \notag 
\end{align}
Therefore, 
\begin{equation} \label{eq:sub_gradient_u_plada}
\| \boldsymbol{\zeta}_{\uu}^{k+1} \| 
= \| \widetilde{\uu}_{k+1} - \uu_{k+1}\| 
\leq \rho M_g\| \x_{k+1} - \x_k \|  + \left( \rho + {1}/{\tau} \right) \|\uu_{k+1} - \uu_{k}\| + \delta_k.
\end{equation}
Combining \eqref{eq:sub_gradient_x_plada} and \eqref{eq:sub_gradient_u_plada}, we obtain 
\begin{equation}
\| \boldsymbol{\zeta}_{\mathbf{p}}^{k+1}\|
\leq d_1 (\|\x_{k+1} - \x_k \| + \|\uu_{k+1} - \uu_k \|) + (M_g + 1) \delta_k, \notag
\end{equation}
where 
$d_1 = \max \{ L_f + 1/\eta + \rho (M_g^2  + M_g)  + {1}/{\eta},  \ \rho( M_g +1) + {1}/{\tau} \}.$
This inequality, along with $\boldsymbol{\zeta}_{\mathbf{p}}^{k+1} \in \partial_{\mathbf{p}} \mathcal{L}_{\alpha\beta}(\ww_{k+1})$, yields the desired result.
\end{proof}

\section{Proofs of Asymptotic Convergence for Algorithm \ref{alg:plada}}

\subsection{Proof of Primal Stationarity (Lemma \ref{lem_primal_convergence_plada})} \label{proof_lem_primal_convergence_plada}

\begin{proof}
From Lemma \ref{lem_one_iter_ppl}, we have
\begin{equation}
C_{\bf p} \left(\| \x_{k+1} - \x_k \|^2 + \| \uu_{k+1} - \uu_k \|^2 \right)  \leq \mathcal{L}_{\alpha\beta}(\ww_k) - \mathcal{L}_{\alpha\beta}(\ww_{k+1}) + \widehat{\delta}_k, \label{eq:stationary_e1_plada}
\end{equation}
where $C_{\bf p} = \max\{ C_1,C_2 \}$. Using Lemma \ref{lem_upper_bound_primal_gradient_plada} and the fact $(a+b+c)^2 \leq 3(a^2 + b^2 + c^2)$, we have
\begin{align} 
\| \boldsymbol{\zeta}_{\mathbf{p}}^{k+1} \|^2
& \leq 3 d_1^2 (\| \x_{k+1} - \x_k \|^2 + \| \uu_{k+1} - \uu_k \|^2 ) + 3(M_g+1)^2 \delta_k^2, \notag
\end{align}
which, combined with \eqref{eq:stationary_e1_plada}, yields
 \begin{align}
 \| \boldsymbol{\zeta}_{\mathbf{p}}^{k+1} \|^2 
&  \leq \frac{3 d_1^2}{C_\mathbf{p}} \left(\mathcal{L}_{\alpha\beta}(\ww_k) - \mathcal{L}_{\alpha\beta}(\ww_{k+1}) + \widehat{\delta}_k \right)  + 3 (M_g+1)^2  \delta_k^2. \notag
\end{align}
Summing up the above inequalities over $k=0,\ldots, T-1$, we obtain
\begin{align}
\sum_{k=0}^{T-1} \| \boldsymbol{\zeta}_{\mathbf{p}}^{k+1} \|^2 
 & \leq
\frac{3 d_1^2}{C_{\mathbf{p}}} \left(\mathcal{L}_{\alpha\beta}(\ww_0) - \mathcal{L}_{\alpha\beta}(\ww_{T}) +  \sum_{k=0}^{T-1}  \widehat{\delta}_k  \right) \notag  + 3 (M_g+1)^2   \sum_{k=0}^{T-1}\delta_k^2 \notag
\end{align}

Since $\sum_{k=0}^{\infty} \delta_k^2< + \infty$, we denote $B_\delta = \sum_{k=0}^{\infty} \delta_k^2$.  Therefore,
\begin{equation}
\begin{aligned}
& \frac{1}{T}\sum_{k=0}^{T-1} \| \boldsymbol{\zeta}_{\mathbf{p}}^{k+1} \|^2 \\
& \leq  \frac{\frac{3d_1^2}{C_{\mathbf{p}}}\left(\mathcal{L}_{\alpha\beta}(\ww_0) - \mathcal{L}_{\alpha\beta}(\ww_{T})\right)}{T} 
+ \frac{ \frac{3 d_1^2}{C_{\mathbf{p}}} \sum_{k=0}^{T-1} \widehat{\delta}_k}{T} 
+ \frac{3 (M_g+1)^2 \sum_{k=0}^{T-1}\delta_k^2}{T} \\
& \leq \frac{\frac{3d_1^2}{C_{\mathbf{p}}}\left(\mathcal{L}_{\alpha\beta}(\ww_0) - \underline{\mathcal{L}_{\alpha\beta}}\right)}{T} 
+ \frac{ \left( \frac{3 d_1^2}{2 \rho C_{\mathbf{p}}} + 3 (M_g+1)^2 \right) \sum_{k=0}^{T-1} {\delta}_k^2}{T} 
+ \frac{\frac{1}{\rho}\sum_{k=0}^{T-1}\delta_k}{T}, \label{eq:stationary_e2}
 \end{aligned}
\end{equation}
where the second inequality holds by the the lower boundedness of $\mathcal{L}_{\alpha\beta}(\ww_k)$, denoted by $\mathcal{L}_{\alpha\beta}$, that is from the boundedness of generated sequences, and $\widehat{\delta}_k = \frac{\delta_k^2}{2{\rho}} + \frac{\delta_k}{\rho}$. 

Note that given $\delta_k = {\kappa\cdot (k+1)^{-1}}$ and $\kappa>0$, for sufficiently large $T$, we know that  
$$\sum_{k=0}^{T-1}\delta_k \approx \kappa^{-1}\log(\kappa T).$$ Since the last term on the right-hand side (RHS) of \eqref{eq:stationary_e2} dominates the other terms and $T$ grows faster than $\log(T)$, the RHS of \eqref{eq:stationary_e2} decreases to 0 as $T$ increase.
\end{proof}

\subsection{Proof of Primal Feasibility (Lemma \ref{lem_feasibility_plada})}

\begin{proof}
From the $\m$-update \eqref{eq:mu_update}, notice that
$\m_{k+1}  = \m_0 + \frac{1}{\rho}\sum^k_{t=0} \sigma_t (\la_t - \m_t).$
Using the fact that $\| a \| -\| b\| \leq \|a + b \|$ for any $a,b \in \mathbb{R}^m$, we have
\begin{align}
\left\| \sum^{\infty}_{t=0} \sigma_t (\la_t - \m_t ) \right\| \leq \| \m_{k+1} \| + \|\m_{0} \| < + \infty, \label{eq:lem_feasibility_a1}
\end{align} 
where the last inequality hold by the boundedness of  $\{\m_{k}\}$ from Assumption \ref{assumption_bounded_multiplier} together with the boundedness of sequence $\{(\la_k -\m_k) := \rho (g(\x_k) + u_k)\}$. The convergence of the sequences $\{\x_{k}\}$ and $\{u_{k}\}$ to finite values $(\overline{\x},\overline{u})$, along with the definition of $\la_{k} = \m_{k} + \rho( g(\x_{k}) +u_{k})$, implies that  $\{ \la_{k} - \m_{k} \}$ is convergent to a finite value $(\overline{\la} -\overline{\m})$.

We prove that $\{ \la_k - \m_k\} \rightarrow 0$  by contradiction. Assume that $\{ \la_k - \m_k \}$ does not converge $0$, meaning there exists some $e \neq 0$ such that $ \{ \la_k - \m_k\}  \rightarrow e $ as $k\rightarrow \infty$. Since $\sum^{\infty}_{k=0} \sigma_k = \infty$, we see that 
\[
\left\| \sum^{\infty}_{k=0} \sigma_k (\la_k - \m_k ) \right\| = \infty, 
\]
which contradicts \eqref{eq:lem_feasibility_a1}. This contradiction leads to the desired result that $\overline{\la} - \overline{\m} = 0$. It directly follows the definitions of $\la_{k+1}$ and $u_{k+1}$ that
$$ 0 = \frac{1}{\rho}\left( \overline{\la} - \overline{\m} \right) =g(\overline{\x}) + \overline{u}  \ \ \text{and} \ \ \overline{u} \geq 0.$$
Hence, we have the feasibility of $\overline{\x}$, namely, $g(\overline{\x}) \leq 0.$
Now consider the dual stationarity defined as $\boldsymbol{\zeta}_{{\bf d}}^{k+1} :=(\zeta_{\la}^{k+1}, \zeta_{\m}^{k+1}) \in \partial_{{\bf d}} \mathcal{L}_{\rho} (\ww_{k+1})$.
Since $\la$ update step \eqref{eq:lambda_update} is an exact maximization, $\boldsymbol{\zeta}_{\la}^{k+1}=0$. Due to pointwise convergence of $\| \la_{k} - \m_{k} \|\to0$ as $k\to0$, $\boldsymbol{\zeta}_{\m}^{k+1}\to0$ as well by $\m$ update step \eqref{eq:mu_update}. 
\end{proof}

\subsection{Proof of Dual Feasibility (Lemma \ref{lem_dual_feasibility_plada})} \label{proof_lem_dual_feasibility_plada}

\begin{proof}
By the update rule $\eqref{eq:u_update}$, $\uu\in\mathbb{R}_+^m$ and $\bar{\uu} \ge \0$. And the stationarity condition with respect to $\uu$ implies that a fixed point $\bar{\uu}$ must satisfy 
\begin{equation}
    \bar{\uu} = \Pi_{\mathbb{R}_+^m}[\bar{\uu} - \tau\bar{\la}] \iff \langle \bar{\la}, \uu - \bar{\uu} \rangle \ge 0, \quad \forall\uu \in \mathbb{R}_+^m. \notag
\end{equation}
This condition can be stated as 
\begin{equation}
    \0 \in \bar{\la} + \mathcal{N}_{\mathbb{R}_+^m}(\bar{\uu}) \iff -\bar{\la} \in \mathcal{N}_{\mathbb{R}_+^m}(\bar{\uu}), \label{eq:lambda_normal_cone}
\end{equation}
where $\mathcal{N}_{\mathbb{R}_+^m}(\bar{\uu})$ is the normal cone to $\mathbb{R}_+^m$ at the limit point $\bar{\uu}$. For any component $j$, if $\bar{\uu}_j > \0$, the point is in the interior of the set along this axis, and the normal cone contains only the zero vector, i.e., $(\mathcal{N}_{\mathbb{R}_+^m}(\bar{\uu}))_j = \{\0\}$. Thus, \eqref{eq:lambda_normal_cone} implies that if $\bar{\uu}_j > \0$, we must have $\bar{\la}_j = \0$. If $\bar{\uu}_j = \0$, the point is at the boundary, and the normal cone consists of all non-positive scalars, i.e., $(\mathcal{N}_{\mathbb{R}_+^m}(\bar{\uu}))_j = (-\infty, 0]$. In any case, we conclude that $\bar{\la} \ge \0$.
\end{proof}

\subsection{Proof of Complementary Slackness (Lemma \ref{lem_complementary_slackness_plada})} \label{proof_lem_complementary_slackness_plada}

\begin{proof}
We prove that for each component $j$, $\bar{\la}_j g_j(\bar{\x})=0$. Consider two cases for $\bar{\uu}_j \ge 0$, where the non-negativity follows from the update rule \eqref{eq:u_update}. First, if $\bar{\uu}_j > \0$, $\bar{\la}_j = \0$ by the stationarity condition used in the proof of Lemma \ref{lem_dual_feasibility_plada}. Thus,
\begin{equation}
    \bar{\la}_j g_j(\bar{\x}) = 0 \cdot g_j(\bar{\x}) = 0. \notag
\end{equation}
Second, if $\bar{\uu}_j = 0$, from Lemma \ref{lem_feasibility_plada}, $g_j(\bar{\x}) = -\bar{\uu}_j = 0$. Thus,
\begin{equation}
    \bar{\la}_j g_j(\bar{\x}) = \bar{\la}_j \cdot 0 = 0, \notag
\end{equation}
concluding the proof.
\end{proof}

\section{Proofs of Non-asymptotic Convergence Rate for Algorithm \ref{alg:plada}}

\subsection{Proof of Primal Stationarity Convergence Rate (Lemma \ref{lem_rate_primal_stationarity_plada})}

\begin{proof}
The rate of average residual convergence established in Lemma \ref{lem_primal_convergence_plada} corresponds to using $\la_k$. From the definition of primal residual,
\begin{equation}
    \boldsymbol{\zeta}_{\x}^{k+1} = \nabla f(\x_{k+1}) - \nabla f(\x_k) + \partial g(\x_{k+1})^{\top}(\la_{k+1}-\la_k) - \frac{1}{\eta}(\x_{k+1}-\x_k). \notag
\end{equation}
And the optimality condition for the update of $\x+{k+1}$ in \eqref{eq:x_update} gives
\begin{equation}
    -(\nabla f(\x_k) + \partial g(\x_{k+1})^\top \la_k + \frac{1}{\eta}(\x_{k+1}-\x_k)) \in \partial r(\x_{k+1}). \notag
\end{equation}
Rearranging this gives an expression for the stationarity residual with respect to $\la_{k+1}$:
\begin{equation}
    \boldsymbol{\zeta}_{\x}^{k+1} \in \nabla f(\x_{k+1}) + \partial r(\x_{k+1}) + \partial g(\x_{k+1})^\top \la_{k+1}
\end{equation}
The difference between the stationarity residuals for $\la_k$ and $\boldsymbol{\nu}_k$ is bounded by $\|\partial g(\bar{\x}_T)^\top(\bar{\boldsymbol{\nu}}_T - \bar{\la}_T)\|$, where $\bar{\x}_T = \frac{1}{T}\sum \x_k$ and $\bar{\la}_T = \frac{1}{T}\sum \la_k$. The difference in the averaged multipliers is: 
\begin{equation} 
\bar{\boldsymbol{\nu}}_T - \bar{\la}_T = \frac{1}{T}\sum_{k=0}^{T-1} (\boldsymbol{\nu}_k - \la_k) = \frac{1}{T}\sum_{k=0}^{T-1} \frac{1}{\tau}(\uu_{k+1} - \uu_k) = \frac{1}{\tau T}(\uu_T - \uu_0). 
\end{equation} 
Since the iterates $\{\uu_k\}$ are bounded, $\|\bar{\boldsymbol{\nu}}_T - \bar{\la}_T\| = \mathcal{O}(1/T)$. This difference vanishes faster than the stationarity residual itself, which converges at $\tilde{\mathcal{O}}(1/\sqrt{T})$. Therefore, the stationarity condition holds for $\bar{\boldsymbol{\nu}}_T$ with the same convergence rate.
\end{proof}

\subsection{Proof of Primal Feasibility Convergence Rate (Lemma \ref{lem_rate_primal_feasibility_plada})} \label{proof_lem_rate_primal_feasibility_plada}

\begin{proof}
The rate of convergence for primal feasibility is independent of the choice of multiplier and can be quantified by leveraging the convergence rate of the dual variables. From the update rule \eqref{eq:lambda_update} for $\la_{k+1}$, we have the relation $g(\x_{k+1}) + \uu_{k+1} = \frac{1}{\rho}(\la_{k+1} - \m_{k+1})$. Since $\uu_{k+1} \ge \0$, the norm of the primal feasibility violation, $[g(\x_{k+1})]^{+} = \max\{\0, g(\x_{k+1})\}$, is bounded: 
\begin{equation} 
\|[g(\x_{k+1})]^{+}\| \le \|g(\x_{k+1}) + \uu_{k+1}\| = \frac{1}{\rho}\|\la_{k+1} - \m_{k+1}\|.
\end{equation} 
This inequality holds because for any component $j$, if $g_j(\x_{k+1}) \le \0$, then $([g(\x_{k+1})]^{+})_j^2 = \0 \le (g_j(\x_{k+1}) + \uu_{j,k+1})^2$. If $g_j(\x_{k+1}) > \0$, then since $\uu_{j,k+1} \ge \0$, we have $([g(\x_{k+1})]^{+})_j^2 = g_j(\x_{k+1})^2 \le (g_j(\x_{k+1}) + \uu_{j,k+1})^2$. Summing over all components yields the squared norm inequality.

The convergence results in Lemma \ref{lem_feasibility_plada} establish that $\frac{1}{T}\sum_{k=0}^{T-1} \|\la_{k+1} - \m_{k+1}\|^2 = \tilde{\mathcal{O}}(1/T)$. This allows us to bound the running average of the squared primal feasibility violation: 
\begin{equation}
\frac{1}{T}\sum_{k=0}^{T-1} \|[g(\x_{k+1})]^{+}\|^2 \le \frac{1}{\rho^2 T}\sum_{k=0}^{T-1} \|\la_{k+1} - \m_{k+1}\|^2 = \tilde{\mathcal{O}}\left(\frac{1}{T}\right).
\end{equation}
By applying Jensen's inequality, we find that the average primal feasibility violation converges at a rate of $\tilde{\mathcal{O}}(1/\sqrt{T})$. 
\end{proof}

\subsection{Proof of Complementary Slackness Convergence Rate (Lemma \ref{lem_rate_complementary_slackness_plada})}

\begin{proof}
From the optimality of $\uu$-update \eqref{eq:u_update} and the construction of $\boldsymbol{\nu}$, $\uu_{k+1} = \Pi_{\mathbb{R}_+^m}[\uu_{k+1} - \tau \boldsymbol{\nu}_{k}]$ implies the complementarity condition $\langle \boldsymbol{\nu}_k, \uu - \uu_{k+1} \rangle \ge 0$ for all $\uu \ge \0$. By choosing $\uu = \0$, we have $-\langle \boldsymbol{\nu}_k, \uu_{k+1} \rangle \ge 0$. Since both $\boldsymbol{\nu}_k$ and $\uu_{k+1}$ are non-negative vectors, they are orthogonal component-wise:
\begin{equation}
    \nu_{j,k} u_{j,k+1} = 0 \quad \forall j \in [m]. \label{eq:elementwise_complementarity}
\end{equation}
Now consider the approximate complementary slackness at each iteration defined as Definition \ref{def_epsilon-kkt}. By the $\la$-update \eqref{eq:lambda_update} and \eqref{eq:elementwise_complementarity},
\begin{equation}
    \nu_{j,k} g_j(\x_{k+1}) = \nu_{j,k} \left(-u_{j,k+1} + \frac{1}{\rho}(\lambda_{j,k+1} - \mu_{j,k+1})\right) = \frac{\nu_{j,k}}{\rho}(\lambda_{j,k+1} - \mu_{j,k+1}). \notag
\end{equation}
By applying Cauchy-Schwarz inequality on the approximate complementary slackness,
\begin{equation}
    \sum_{j=1}^m |\nu_{j,k} g_j(\x_{k+1})| \le \frac{1}{\rho} \sum_{j=1}^m |\nu_{j,k}| |\lambda_{j,k+1} - \mu_{j,k+1}| \le \frac{1}{\rho} \|\boldsymbol{\nu}_k\| \|\la_{k+1} - \m_{k+1}\|.
\end{equation}
Since $\{\la_k\}$ and $\{\uu_k\}$ are bounded, $\{\boldsymbol{\nu}_k\}$ is bounded by a constant $B_{\nu}:=\max_{k\ge1}\{\boldsymbol{\nu}_k\}$. Thus,
\begin{equation}
    \sum_{j=1}^m |\nu_{j,k} g_j(\x_{k+1})| \le \frac{B_{\nu}}{\rho} \|\la_{k+1} - \m_{k+1}\|. \notag
\end{equation}
Now, we can analyze the running average:
\begin{equation}
    \frac{1}{T}\sum_{k=0}^{T-1} \sum_{j=1}^m |\nu_{j,k} g_j(\x_{k+1})| \le \frac{B_{\nu}}{\rho T} \sum_{k=0}^{T-1} \|\la_{k+1} - \m_{k+1}\|. \notag
\end{equation}
Using Jensen's inequality and the established rate for the dual residual from , we get:
$$\frac{1}{T}\sum_{k=0}^{T-1} \|\la_{k+1} - \m_{k+1}\| \le \sqrt{\frac{1}{T}\sum_{k=0}^{T-1} \|\la_{k+1} - \m_{k+1}\|^2} = \sqrt{\tilde{\mathcal{O}}\left(\frac{1}{T}\right)} = \tilde{\mathcal{O}}\left(\frac{1}{\sqrt{T}}\right).$$
This then establishes that the average of the sum of absolute values converges at the required rate of $\tilde{\mathcal{O}}(1/\sqrt{T})$.
\end{proof}

\section{Supporting Lemmas for Convergence Analysis of Algorithm \ref{alg:ppala}}
Now, we establish several key properties of Algorithm \ref{alg:ppala}, including important relations among the primal and dual sequences, the approximate decrease of the Proximal-Perturbed Augmented Lagrangian ($L_\rho$), and error bounds for its subgradient in the primal variables. These intermediate results are crucial stepping stones for proving the algorithm's overall convergence to an $\epsilon$-KKT point. We first provide basic yet crucial relations on the sequences ${{\la^k}}$, ${{\m^k}}$, and ${{\mathbf{x}^k}}$.

\begin{lemma} \label{lem_iterates_relations} 
Let $\{(\x_k,\uu_k,\z_k,\la_k,\m_k)\}$ be the sequence generated by Algorithm \ref{alg:ppala} with the choice of the sequence $\{\delta_k\}$ as in \eqref{eq:delta_k}. Under Assumption \ref{assumption_bounded_domain}, for any $k \geq 1$, the following relations hold: 
\begin{subequations}
\begin{align}
\| \m_{k+1} - \m_k \|^2 & = \sigma_k^2 \| \la_k - \m_k \|^2 \leq {\delta_k^2}/{4}, \label{eq:lem_iter_rel_1-1} \\
\sigma_k \| \la_k - \m_k \|^2 & \leq \delta_k, \label{eq:lem_iter_rel_1-2} \\
\| \m_{k+1} - \la_k \|^2 & = (1 - \sigma_k)^2 \| \la_k - \m_k \|^2,  \label{eq:lem_iter_rel_2} \\
\| \la_{k+1}-\la_k \|^2 & \leq 3 \rho^2 M_g^2  \| \x_{k+1} - \x_k \|^2 + 3\rho^2 \| \uu_{k+1} - \uu_k \|^2 + {3\delta_k^2}/{4}, \label{eq:lem_iter_rel_3}
\end{align}
\end{subequations}
where $M_g$ denotes the Lipschitz constant of $g$ from~\eqref{eq:bound_jacobian_gx}.
\end{lemma}

\begin{proof}
Relations \eqref{eq:lem_iter_rel_1-1}, \eqref{eq:lem_iter_rel_2} and \eqref{eq:lem_iter_rel_3} follow the same proofs as \eqref{eq:lem_iter_rel_plada_1}, \eqref{eq:lem_iter_rel_plada_2} and \eqref{eq:lem_iter_rel_plada_3}, respectively.

By the definitions $\sigma_k = \frac{\delta_k}{\| \la_k - \m_k \|^2 + 1} \leq 1 $ and $\delta_k \in (0,1]$, we know that $\sigma_k \leq 1$. Thus, we obtain the relation \eqref{eq:lem_iter_rel_1-2}:
\begin{equation} \label{eq:lem_iterates_relations_1_1}
	\sigma_k \| \la_k - \m_k \|^2 = 
	\frac{\delta_k}{1 + \frac{1}{\| \la_k - \m_k \|^2} } \leq \delta_k. \notag
\end{equation}
Subtracting $\m_{k+1}$ from $\la_k$ yields
\[ 
\|  \la_k - \m_{k+1} \|= \| \la_{k} - \m_{k} -  \sigma_{k}(\la_{k} - \m_{k}) \| = (1-\sigma_{k}) \| \la_{k} - \m_{k} \|.
\] 
Squaring both sides of the inequality yields the relation \eqref{eq:lem_iter_rel_2}. 
\end{proof}

The relations in Lemma \ref{lem_iterates_relations} are critical to our technique for proving convergence, bypassing the need for the surjectivity of the Jacobian $\nabla g(\x)$ (or subgradient mapping $\partial g(\mathbf{x})$) as in \cite{bolte2018nonconvex,boct2020proximal}. 

\begin{lemma}\label{lem_descent_lemma} 
Let $\{\x^k\}$ be the sequence generated by Algorithm \ref{alg:ppala}. Under Assumptions \ref{assumption_lipschitz_f}, \ref{assumption_lipschitz_g} and \ref{assumption_bounded_domain}, there exists a constant $L_\mathcal{\ell}>0$ such that   
\begin{align}
\mathcal{\ell}_{\rho}(\x_{k+1}) &\leq \mathcal{\ell}_{\rho}(\x_k) + \left\langle \nabla_{\x} \mathcal{\ell}_{\rho}(\x_k), \x_{k+1} - \x_k \right\rangle + \frac{L_\mathcal{\ell}}{2} \| \x_{k+1} - \x_k \|^2,
\end{align}
where $L_\mathcal{\ell} :=  L_f + L_g B_{\la} + \rho (L_g B_{\uu} + L_g B_g + M_g^2)$  with 
$B_{\la} = \max_{k \geq 0} \|\la_k \| $, $B_{\uu} = \max_{k \geq 0} \|\uu_k \| $, $ B_g = \max_{\x \in \mathrm{dom}(r)} \| g(\x) \|$  and $ M_g = \max_{\x \in \mathrm{dom}(r)} \| \nabla g(\x) \|$ from \eqref{eq:bound_jacobian_gx}.
\end{lemma}

In the statement of Lemma \ref{lem_descent_lemma}, we omitted $(\uu_k,\z_k,\la_k,\m_k)$ in the argument of $\ell_{\rho}(\cdot)$ for simplicity.

\begin{proof}
Note that 
$\nabla_{\x} \mathcal{\ell}_{\rho}(\x,\uu,\z,\la,\m) = \nabla f(\x) + \nabla g(\x)( \la + \rho (g(\x)+ \uu))$. A direct computation gives
\begin{equation}
\begin{aligned}
\| \nabla_{\x} \mathcal{\ell}_{\rho}(\x_{k+1}) - \nabla_{\x} \mathcal{\ell}_{\rho}(\x_k) \| & \leq  \| \nabla f(\x_{k+1}) - \nabla f(\x_k) \| \notag + \| \left(\nabla g(\x_{k+1}) -\nabla g(\x_k) \right) (\la_k + \rho \uu_{k}) \| \notag \\ 
& \quad  + \rho \| \nabla g(\x_{k+1}) g(\x_{k+1}) - \nabla g(\x_k) g(\x_{k+1}) \| \notag \\
& \quad+ \rho \| \nabla g(\x_k) g(\x_{k+1}) - \nabla g(\x_k) g(\x_k) \|  \notag \\
& \leq  L_f \| \x_{k+1} - \x_k \| + L_g (B_{\la} + \rho B_{\uu}) \| \x_{k+1} - \x_k \| \notag \\
& \quad + \rho L_g B_g \| \x_{k+1} - \x_k \| + \rho M_g^2 \| \x_{k+1} - \x_k \|  \notag \\
& \leq  \left( L_f + L_g B_{\la} + \rho (L_g B_{\uu} + L_g B_g + M_g^2) \right) \| \x_{k+1} - \x_k \|. \notag
\end{aligned}
\end{equation}
Hence, by the descent lemma  \cite[Proposition A.24]{bertsekas1999nonlinear}, we obtain the desired result.
\end{proof}

Now, we establish key properties that lead to our main convergence results.

\begin{lemma} \label{lem_ppal_decrease_converge}
Let the sequence $\left\{\ww_{k}\right\}$ be generated by Algorithm \ref{alg:ppala}. Under Assumptions \ref{assumption_lipschitz_f}, \ref{assumption_lipschitz_g} and \ref{assumption_bounded_domain}, the Proximal-Perturbed Augmented Lagrangian $\mathcal{L}_{\rho}$ \eqref{eq:ppal} satisfies:
\begin{enumerate}[itemsep=1ex, label=(\alph*), ref= {(\alph*)}]

\item \label{lem_one_iter_ppal} {\bf (Approximate Decrease of $\mathcal{L}_{\rho}$)}
\begin{equation}
\mathcal{L}_{\rho}(\ww_{k+1}) - \mathcal{L}_{\rho}(\ww_k) \leq - c_1 \| \x_{k+1} - \x_k \|^2 - c_2 \| \uu_{k+1} - \uu_{k} \|^2 + \widehat{\delta}_k, \notag 
\end{equation}
where $c_1 = \frac{1}{2}\left(\frac{1}{\eta} - L_{\ell} -  3 \rho M_g^2 \right) >0 $, $c_2 = \left(\frac{1}{\tau} - 2 \rho \right) >0$, and $\widehat{\delta}_k := \frac{\delta_k^2}{4{\rho}} + \frac{\delta_k}{\rho}$. 

\item \label{lem_lagrangian_convergence} {\bf (Convergence of $\mathcal{L}_{\rho}$)}  
the sequence $\{\mathcal{L}_{\rho}(\ww_{k})\}$ is convergent, i.e.,
$\lim_{k \rightarrow \infty} \mathcal{L}_{\rho}(\ww_{k+1}) := \underline{\mathcal{L}_{\rho}} > -\infty.$
\end{enumerate}
\end{lemma}

\begin{proof}
\ref{lem_one_iter_ppal} 
The difference between two consecutive sequences of $\mathcal{L}_{\rho}$ can be divided into four parts:
\begin{subequations}
\begin{align}
\mathcal{L}_{\rho}(\ww_{k+1}) - \mathcal{L}_{\rho}(\ww_k) &= \left[  \mathcal{L}_{\rho}(\x_{k+1},\uu_{k},\z_{k},\la_{k},\m_{k}) -        \mathcal{L}_{\rho}(\ww_{k}) \right] 
\label{eq:lem_one_iter_ppal1} \\
& + [  \mathcal{L}_{\rho}(\x_{k+1},\uu_{k+1},\z_{k},\la_{k},\m_{k}) - \mathcal{L}_{\rho}(\x_{k+1},\uu_{k},\z_{k},\la_k,\m_k)] 
\label{eq:lem_one_iter_ppal2} \\
& + [ \mathcal{L}_{\rho}(\x_{k+1},\uu_{k+1},\z_{k},\la_{k+1},\m_{k+1}) - \mathcal{L}_{\rho}(\x_{k+1},\uu_{k+1},\z_{k},\la_{k},\m_{k}) ] 
\label{eq:lem_one_iter_ppal3} \\
& + \left[ \mathcal{L}_{\rho}(\ww_{k+1}) - \mathcal{L}_{\rho}(\x_{k+1},\uu_{k+1},\z_{k},\la_{k+1},\m_{k+1}) \right]. 
\label{eq:lem_one_iter_ppal4}
\end{align}
\end{subequations}
First, we consider \eqref{eq:lem_one_iter_ppal1}. Writing $\mathcal{L}_{\rho}(\x_{k+1})=\mathcal{L}_{\rho}(\x_{k+1},\uu_{k},\z_k,\la_k,\m_k)$,  and using Lemma \ref{lem_descent_lemma}, we have
\begin{equation} \label{eq:lem_one_iter_ppal_p1_1}
\begin{aligned} 
\ell_{\rho}(\x_{k+1}) & \leq \ell_{\rho}(\x_{k})
+ \left\langle \nabla_\x \ell_{\rho}(\x_k), \x_{k+1}- \x_k \right\rangle + \frac{L_{\ell}}{2} \| \x_{k+1} - \x_k \|^2.
\end{aligned}
\end{equation}   
From the definition of $\x_{k+1}$ in \eqref{eq:x_update_ppala}, it follows that
\begin{displaymath}
\begin{aligned}
  \mathcal{L}_{\rho}(\x_{k}) & \geq
\ell_{\rho}(\x_k) + \left\langle \nabla_x \ell_{\rho}(\x_k), \x_{k+1} - \x_k \right\rangle + \frac{1}{2\eta}\| \x_{k+1} - \x_k \|^2 + r(\x_{k+1}), 
\end{aligned}
\end{displaymath}
implying 
    $\left\langle \nabla_x \ell_{\rho}(\x_k), \x_{k+1} - \x_k \right\rangle + r(\x_{k+1})  \leq -\frac{1}{2\eta}\| \x_{k+1} - \x_k\|^2 + r(\x_k).$
Combining the this expression with \eqref{eq:lem_one_iter_ppal_p1_1} yields 
\begin{equation} \label{eq:lem_one_iter_ppal_p1}
\begin{aligned}
\mathcal{L}_{\rho}(\x_{k+1},\uu_k,\z_k,\la_k,\m_k) 
    -  \mathcal{L}_{\rho}(\x_k,\uu_k,\z_k,\la_k,\m_k) \leq - \frac{1}{2}\left(\frac{1}{\eta} -L_{\ell} \right) \| \x_{k+1} - \x_k \|^2. 
\end{aligned}
\end{equation}

Next, consider the second part \eqref{eq:lem_one_iter_ppal2}. Noting that $\nabla_{\uu} \mathcal{L}_{\rho}$ is $\rho$-Lipschitz continuous, we have
\begin{align}
\mathcal{L}_{\rho}(\uu_{k+1})   
& \leq \mathcal{L}_{\rho}(\uu_{k}) 
 + \left\langle \nabla_{\uu} \mathcal{L}_{\rho}(\uu_{k}), \uu_{k+1} - \uu_{k} \right\rangle + \frac{\rho}{2}  \| \uu_{k+1} - \uu_{k} \|^2. \notag
\end{align}
By using the property of the projection operator, $\left\langle \Pi_{[0,U]}[\A] - \A, \bb - \Pi_{[0,U]}[\A] \right\rangle \geq 0 \ \text{for} \ \bb \in \Pi_{[0,U]}, \ \forall \A \in \mathbb{R}^m$, with $\A =\uu_k - \tau \nabla_{\uu} \mathcal{L}_{\rho}(\uu_{k})$, and $\bb=\uu_k$, we get  
\[
\left\langle \uu_{k+1}-\uu_{k} + \tau \nabla_{\uu} \mathcal{L}_{\rho}(\uu_{k}),  \uu_k - \uu_{k+1} \right\rangle \geq 0,\]
from which we have 
$\left\langle \nabla_{\uu} \mathcal{L}_{\rho}(\uu_{k}),  \uu_{k+1} - \uu_{k} \right\rangle \leq - \frac{1}{\tau} \| \uu_{k+1} - \uu_{k} \|^2.$
Therefore,
\begin{equation}
    \mathcal{L}_{\rho}(\x_{k+1},\uu_{k+1},\z_k,\la_k,\m_k) - \mathcal{L}_{\rho}(\x_{k+1},\uu_k,\z_k,\la_k,\m_k) \leq - \left( \frac{1}{\tau} - \frac{\rho}{2} \right)\| \uu_{k+1} - \uu_{k} \|^2.
\label{eq:lem_one_iter_ppal_p2} 
\end{equation}
Now consider \eqref{eq:lem_one_iter_ppal3}. We start by noting that
\begin{equation}
\begin{aligned}
& \quad \mathcal{L}_{\rho}(\x_{k+1},\uu_{k+1},\z_{k},\la_{k+1},\m_{k+1}) - \mathcal{L}_{\rho}(\x_{k+1},\uu_{k+1},\z_{k},\la_{k},\m_{k})  \\ 
& = \underbrace{\left\langle \la_{k+1} - \la_k, g(\x_{k+1}) + \uu_{k+1}  \right\rangle}_{\mathrm{(I)}} + \underbrace{\left\langle (\la_k - \m_k) - (\la_{k+1} - \m_{k+1}), \z_{k} \right\rangle}_{\mathrm{(II)}} \\
&  \quad - \frac{\beta}{2} \| \la_{k+1} - \m_{k+1} \|^2 + \frac{\beta}{2} \| \la_k - \m_k \|^2.  \label{eq:lem_one_iter_ppal_p3_1}
\end{aligned}
\end{equation}
Using the updates $\la_{k+1}  = \m_{k+1} + \rho (g(\x_{k+1}) + \uu_{k+1})$ and $\z_k = \frac{1}{\alpha}(\la_k - \m_k)$, and the fact that $\left\langle \A -\bb, \A \right\rangle =\frac{1}{2} \| \A -\bb \|^2 + \frac{1}{2} \| \A \|^2 - \frac{1}{2} \| \bb \|^2$  with $\A= \la_{k} -\m_{k}$ and $\bb=\la_{k+1} -\m_{k+1}$, we have
\begin{equation}  
\begin{aligned}
\mathrm{(I)} & = \frac{1}{2\rho} \| \la_{k+1}-\la_k \|^2 
 + \frac{1}{2\rho} \| \la_{k+1} -\m_{k+1} \|^2 - \frac{1}{2\rho} \| \m_{k+1} -\la_k \|^2, \\
\mathrm{(II)} & =  \frac{1}{2\alpha} \| (\la_{k+1} - \m_{k+1})-(\la_k - \m_k) \|^2 
 + \frac{1}{2\alpha} \| \la_{k} -\m_{k} \|^2 - \frac{1}{2\alpha} \| \la_{k+1} -\m_{k+1} \|^2 \\
& =  \frac{\alpha}{2} \| \z_{k+1} - \z_{k} \|^2 
 + \frac{1}{2\alpha} \| \la_{k} -\m_{k} \|^2 - \frac{1}{2\alpha} \| \la_{k+1} -\m_{k+1} \|^2. 
\label{eq:lem_one_iter_ppal_p3_2}
\end{aligned}
\end{equation}
Substituting \eqref{eq:lem_one_iter_ppal_p3_2} into \eqref{eq:lem_one_iter_ppal_p3_1} yields
\begin{equation}
\begin{aligned}
& \mathcal{L}_{\rho} (\x_{k+1},\uu_{k+1},\z_k,\la_{k+1},\m_{k+1}) - \mathcal{L}_{\rho} (\x_{k+1},\uu_{k+1},\z_k,\la_k,\m_k) \\
& \leq 
\frac{1}{2\rho} \| \la_{k+1} - \la_k \|^2 
- \frac{1}{2\rho} \| \m_{k+1} - \la_k \|^2 + \frac{1}{2\rho} \| \la_k - \m_k \|^2 + \frac{\alpha}{2} \| \z_{k+1} - \z_{k} \|^2 \\
& \overset{(\mathrm{i})}{\leq}  
\frac{1}{2\rho} \left( 3 \rho^2 M_g^2  \| \x_{k+1} - \x_k \|^2 + 3 \rho^2 \| \uu_{k+1} - \uu_k \|^2 + 3 \|\m_{k+1} -\m_{k} \|^2 \right) \\
& \quad + \frac{1}{2\rho}\left( 2\sigma_k - \sigma_k^2 \right) \| \la_k - \m_k \|^2  +\frac{\alpha}{2} \|  \z_{k+1} - \z_{k} \|^2  \\
& \overset{(\mathrm{ii})}{\leq} 
\frac{1}{2} \left( 3\rho M_g^2 \| \x_{k+1} - \x_k \|^2 + 3\rho \| \uu_{k+1} - \uu_k \|^2 \right) \\ 
& \quad + \frac{1}{2\rho}\left( 2\sigma_k + 2\sigma_k^2 \right) \| \la_k - \m_k \|^2 +\frac{\alpha}{2} \| \z_{k+1} - \z_{k} \|^2  \\ 
& \overset{(\mathrm{iii})}{\leq} 
\frac{1}{2}  \left( 3\rho M_g^2 \| \x_{k+1} - \x_k \|^2 + 3\rho \| \uu_{k+1} - \uu_k \|^2 \right) + \frac{\delta_k^2}{4\rho} + \frac{\delta_k}{\rho} +\frac{\alpha}{2} \| \z_{k+1} - \z_{k} \|^2, \label{eq:lem_one_iter_ppal_p3}
\end{aligned}	
\end{equation}
where $(\text{i})$ follows from \eqref{eq:lem_iter_rel_2}  and \eqref{eq:lem_iter_rel_3} in Lemma \ref{lem_iterates_relations}; $(\text{ii})$ follows from \eqref{eq:lem_iter_rel_1-1} in Lemma \ref{lem_iterates_relations}; and
$(\text{iii})$ is from \eqref{eq:lem_iter_rel_1-1} and \eqref{eq:lem_iter_rel_1-2} in Lemma \ref{lem_iterates_relations}.
\vspace{0.01in}

Lastly, we consider \eqref{eq:lem_one_iter_ppal4}. Write down $\mathcal{L}_{\rho}(\z_{k+1}) = \mathcal{L}_{\rho}(\x_{k+1},\uu_{k+1},\z_{k+1},\la_{k+1},\m_{k+1})$ for notational simplicity. From the $\alpha$-strong convexity of $\mathcal{L}_{\rho}$ in $\z$, we have
\begin{align} 
\mathcal{L}_{\rho} (\z_{k}) & \geq \mathcal{L}_{\rho} (\z_{k+1}) + \left\langle \nabla_{\z}\mathcal{L}_{\rho}(\z_{k+1}), \z_{k} - \z_{k+1} \right\rangle + \frac{\alpha}{2}\| \z_{k+1} - \z_k \|^2. \notag     	
\end{align}
Since  $\z_{k+1}$ minimizes $\mathcal{L}_{\rho}(\x_{k+1},\uu_{k+1},\z,\la_{k+1},\m_{k+1})$, we have that $\nabla_{\z}\mathcal{L}_{\rho}(\z_{k+1}) = \0$.  Thus,
\begin{align} 
\mathcal{L}_{\rho} (\z_{k+1}) - \mathcal{L}_{\rho} (\z_{k})
\leq - \frac{\alpha}{2}\| \z_{k+1} - \z_k \|^2.  \label{eq:lem_one_iter_ppal_p4}
\end{align}
Combining \eqref{eq:lem_one_iter_ppal_p1}, \eqref{eq:lem_one_iter_ppal_p2}, \eqref{eq:lem_one_iter_ppal_p3}, and \eqref{eq:lem_one_iter_ppal_p4} yields the desired result.

\ref{lem_lagrangian_convergence}
By using the update of $\z_{k+1}=\frac{\la_{k+1} - \m_{k+1}}{\alpha}$, we deduce
\[\begin{aligned} 
\mathcal{L}_{\rho}(\ww_{k+1}) & = f(\x_{k+1})  + \left\langle \la_{k+1}, g(\x_{k+1}) + \uu_{k+1} \right\rangle \\
& \quad \ \underbrace{- \frac{1}{2\rho}\| \la_{k+1} - \m_{k+1} \|^2 + \frac{\rho}{2}\| g(\x_{k+1}) + \uu_{k+1}  \|^2}_{=0} + r(\x_{k+1}) \\
& = f(\x_{k+1}) + \frac{1}{2\rho}\|  \la_{k+1} \|^2 + \frac{1}{2\rho}\| \la_{k+1} - \m_{k+1} \|^2 - \frac{1}{2\rho}\|  \m_{k+1} \|^2 + r(\x_{k+1})  > -\infty,        
\end{aligned}\]	
where the last inequality holds by the boundedness of $\{\m_k\}$ (Assumption \ref{assumption_bounded_multiplier}) and the lower boundedness of $f$ and $r$ over  $\text{dom}(r)$ (Assumption \ref{assumption_bounded_domain}). Given the step sizes $0 < \eta < {1}/({L_{\ell} +  3\rho M_g^2})$ and $0< \tau < 1/{2\rho}$, we already know the sequence $\{\mathcal{L}_{\rho}(\ww_{k+1})\}$ is {approximately nonincreasing} (Lemma \ref{lem_ppal_decrease_converge}\ref{lem_one_iter_ppal}); Although it may not decrease monotonically at every step, it tends to decrease over iterations. As $\{\delta_k\}$ goes to 0 as $k \rightarrow  \infty$,  $\{\mathcal{L}_{\rho}(\ww_{k+1})\}$ converges to a finite value $\underline{\mathcal{L}_{\rho}}> - \infty$.
\end{proof}

\begin{lemma} [Subgradient Error Bound] \label{lem_upper_bound_primal_gradient}
Let the sequence $\{\ww_k := (\x_{k},\uu_{k},\z_{k},\la_{k},\m_{k})\}$ be generated by Algorithm \ref{alg:ppala}, and let $\{\pp_k := (\x_k,\uu_{k},\z_{k})\}$ be the primal sequence. Under Assumptions \ref{assumption_bounded_domain}, \ref{assumption_lipschitz_f} and \ref{assumption_lipschitz_g}, there exists a constant $d_1 > 0$ such that for the primal subgradient $\boldsymbol{\zeta}_{\pp}^{k+1}=(\zeta_{\x}^{k+1},\zeta_{\uu}^{k+1},\0)  \in \partial_{\pp} \mathcal{L}_{\rho} (\ww_{k+1})$,
\begin{align} 
\| \boldsymbol{\zeta}_{\pp}^{k+1} \| \leq 
d_1 \left( \| \x_{k+1} - \x_k \| + \| \uu_{k+1} -\uu_k \| \right) +  (M_g + 1) \delta_k, \notag
\end{align}
where
\begin{equation}
d_1 = \max \left\lbrace L_f + B_{\la} L_g  + \rho ( M_g +  L_g( B_g +  B_{\uu} ) + 2M_g^2) + {1}/{\eta}, \  2\rho( M_g + 1) + {1}/{\tau}\right\rbrace. \notag
\end{equation}
\end{lemma}

\begin{proof}
From the proof of Lemma \ref{lem_upper_bound_primal_gradient_plada}, we have that for all $k \geq 0$
\begin{equation}
\begin{gathered}
\zeta_{\x}^{k+1} := \nabla_{\x} \ell_{\rho} (\ww_{k+1}) -\nabla_{\x} \ell_{\rho} (\ww_{k})+ \frac{1}{\eta}(\x_k - \x_{k+1}) \in \partial_{\x} \mathcal{L}_{\rho}(\ww_{k+1}), \notag \\
\zeta_{\uu}^{k+1} := \uu_{k+1} - \Pi_{[0,U]}[\uu_{k+1}  - (\la_{k+1} + \rho (g(\x_{k+1}) + \uu_{k+1})] = \widetilde{\nabla}_{\uu} \mathcal{L}_{\rho}(\ww_{k+1}), \notag \\
\nabla_{\z}\mathcal{L}_{\rho}(\ww_{k+1}) = \alpha \z_{k+1} - (\la_{k+1} -\m_{k+1}) = \0. \notag  
\end{gathered}
\end{equation}
Hence, we obtain 
\begin{equation} \label{eq:sub_gradient}
\boldsymbol{\zeta}_{\pp}^{k+1} := 
\begin{pmatrix}
\zeta_{\x}^{k+1} & \in  \partial_{\x}\mathcal{L}_{\rho}(\x_{k+1},\uu_{k+1},\z_{k+1},\la_{k+1},\m_{k+1}) \\
\zeta_{\uu}^{k+1} & = \widetilde{\nabla}_{\uu} \mathcal{L}_{\rho}(\x_{k+1},\uu_{k+1},\z_{k+1},\la_{k+1},\m_{k+1})   \\ 
\0 & =  \nabla_{\z}\mathcal{L}_{\rho}(\x_{k+1},\uu_{k+1},\z_{k+1},\la_{k+1},\m_{k+1})  \\ 
\end{pmatrix}. \notag
\end{equation}

We derive upper estimates for $\zeta_{\x}^{k+1}$ and $\zeta_{\uu}^{k+1}$. A straightforward calculation yields
\begin{subequations}
\begin{align}
\| \zeta_{\x}^{k+1} \| 
\leq & \| \nabla f(\x_{k+1}) -\nabla f(\x_k) \| + (1/{\eta}) \| \x_k - \x_{k+1} \| \notag \\
     & + \| \nabla g(\x_{k+1})(\la_{k+1} + \rho (g(\x_{k+1}) + \uu_{k+1}) - \nabla g(\x_k) (\la_{k} + \rho (g(\x_{k}) + \uu_{k}) \| \notag \\
\leq &  (L_f + 1/\eta) \| \x_{k+1} - \x_{k} \| \notag \\
     & + \|  \nabla g(\x_{k+1}) \la_{k+1} -\nabla g(\x_k) \la_{k+1} + \nabla g(\x_k) \la_{k+1} - \nabla g(\x_k) \la_k \| \label{eq:xi_x1} \\
     & + \rho \| \nabla g(\x_{k+1})g(\x_{k+1}) - \nabla g(\x_{k})g(\x_{k+1}) + \nabla g(\x_{k})g(\x_{k+1})-  \nabla g(\x_k) g(\x_{k}) \|  \label{eq:xi_x2} \\
     & + \rho \| \nabla g(\x_{k+1}) \uu_{k+1} - \nabla g(\x_{k}) \uu_{k+1} + \nabla g(\x_{k}) \uu_{k+1} - \nabla g(\x_k)  \uu_{k} \|,  \label{eq:xi_x3}      
\end{align}
\end{subequations}
in which \eqref{eq:xi_x1}, \eqref{eq:xi_x2}, and \eqref{eq:xi_x3} can be bounded by   
\begin{equation}
\begin{aligned} 
\eqref{eq:xi_x1}  
& \leq B_{\la} L_g \| \x_{k+1} - \x_k \|  + M_g \| \la_{k+1} - \la_k \| \notag \\
& \leq B_{\la} L_g \| \x_{k+1} - \x_k \| + \rho M_g^2 \| \x_{k+1} - \x_k \| \notag \\ & \quad + \rho M_g \| \uu_{k+1} - \uu_k \| + M_g \| \m_{k+1} - \m_k \| \notag \\
& \leq \left(B_{\la} L_g  + \rho M_g^2 \right) \| \x_{k+1} - \x_k \| \notag \\ & \quad + \rho M_g \| \uu_{k+1} - \uu_k \|+  M_g {\delta_k}; \notag \\
\eqref{eq:xi_x2} 
& \leq  (\rho B_g L_g + \rho M_g^2) \| \x_{k+1} - \x_{k} \|; \notag \\
\eqref{eq:xi_x3}
& \leq \rho B_{\uu} L_g \| \x_{k+1} - \x_{k} \| + \rho M_g \| \uu_{k+1} - \uu_{k} \|.  \notag
\end{aligned}
\end{equation}
where for bounding \eqref{eq:xi_x1}, we used the $\la$-update and $\| \m_{k+1} - \m_k \| = \frac{\delta_k}{\| \la_k - \m_k \| + \frac{1}{\| \la_k - \m_k \|}} \leq {\delta_k}.$ 
Hence,
\begin{equation}
\begin{aligned}
\| \zeta_{\x}^{k+1} \| &\leq  (L_f + {1}/{\eta}  + B_{\la} L_g  + \rho L_g ( B_g + B_{\uu} + 2M_g^2) ) \| \x_{k+1} -\x_{k} \| \notag \\
&\quad+ 2\rho M_g \| \uu_{k+1} - \uu_{k} \| + M_g {\delta_k}. \label{eq:sub_gradient_x}
\end{aligned}
\end{equation}

Next, we estimate an upper bound for the component $\zeta_{\uu,k+1}$. From the proof of Lemma \ref{lem_upper_bound_primal_gradient_plada}, we have
\begin{multline} 
\left( \|\nabla_{\uu}\mathcal{L}_{\rho}(\uu_{k}) - \nabla_{\uu}\mathcal{L}_{\rho}(\uu_{k+1})\| + {\tau}^{-1} \| \uu_{k+1} - \uu_{k}\|\right) \cdot \| \widetilde{\uu}_{k+1} - \uu_{k+1} \| \geq \| \widetilde{\uu}_{k+1} - \uu_{k+1}\|^2, \notag
\end{multline}
where
\begin{equation}
\begin{aligned}
\|\nabla_{\uu}\mathcal{L}_{\rho}(\uu_{k}) - \nabla_{\uu}\mathcal{L}_{\rho}(\uu_{k+1})\| &= \|\nabla_{\uu}\mathcal{L}_{\rho}(\x_{k+1},\uu_{k}, \z_{k},\la_{k},\m_{k}) \notag \\ & \hspace{0.5in}  - \nabla_{\uu}\mathcal{L}_{\rho}(\x_{k+1},\uu_{k+1},\z_{k+1},\la_{k+1},\m_{k+1})\| \notag \\
&\leq \| \la_{k} + \rho (g(\x_{k+1}) + \uu_{k})  - \la_{k+1} -\rho (g(\x_{k+1}) + \uu_{k+1}) \| \notag \\
& \leq \rho ( M_g\| \x_{k+1} - \x_k \| + 2\| \uu_{k+1} - \uu_k \| + \delta_k ). \notag
\end{aligned}
\end{equation}
Therefore, 
\begin{align} 
\| \zeta_{\uu}^{k+1} \|  = \| \widetilde{\uu}_{k+1} - \uu_{k+1}\| \leq \rho M_g\| \x_{k+1} - \x_k \| + \left(2 \rho + {\tau}^{-1} \right) \|\uu_{k+1} - \uu_{k}\| + \delta_k. \label{eq:sub_gradient_u}
\end{align}
Combining \eqref{eq:sub_gradient_x} and \eqref{eq:sub_gradient_u}, we obtain 
\begin{equation}
\| \boldsymbol{\zeta}_{\pp}^{k+1}\|
\leq d_1 (\|\x_{k+1} - \x_k \| + \|\uu_{k+1} - \uu_k \|) + (M_g + 1) \delta_k, \notag
\end{equation}
where 
$d_1 = \max \{ L_f + B_{\la} L_g  + \rho ( M_g + B_g L_g + B_{\uu} L_g + 2M_g^2) + {1}/{\eta},   2\rho (M_g +1) + {1}/{\tau} \}.$
This inequality, combined with $\boldsymbol{\zeta}_{\pp}^{k+1} \in \partial_{\pp} \mathcal{L}_{\rho}(\ww_{k+1})$, yields the desired result.
\end{proof}
It can be easily verified that if $\frac{1}{T}\sum_{k=0}^{T-1}\| \boldsymbol{\zeta}_{\pp}^{k+1} \| \rightarrow 0$, then a point that satisfies stationarity in the KKT conditions \eqref{eq:def_kkt}, 
\begin{equation} \label{eq:kkt_stationary}
\0 \in \nabla f(\x^\ast) + \partial r(\x^\ast) + \nabla g(\x^\ast) \la^\ast, \notag
\end{equation}
is obtained. Specifically,
\begin{equation} 
\begin{aligned}
& \begin{cases}
    \0 \in \nabla f(\overline{\x}) + \partial r(\overline{\x}) + \nabla g(\overline{\x}) \overline{\la}, \\  
    \0 = \overline{\uu} - \Pi_{\mathbb{R}_+^m} [\overline{\uu} - (\overline{\la} + \rho (g(\overline{\x}) + \overline{\uu})],
\end{cases} \\ 
\Longleftrightarrow 
& \quad \0 \in \nabla f(\overline{\x}) + \partial r(\overline{\x}) + \nabla g(\overline{\x}) \overline{\la}.
\notag
\end{aligned}
\end{equation}
We will use this part to establish convergence to a KKT point in Theorem \ref{thm_kkt_ppala}. Note that we need not consider the gradient of $\mathcal{L}_\rho$ with respect to $\la$, i.e., $\xi_{\la}^{k+1} := \nabla_{\la} \mathcal{L}_\rho(\mathbf{w}_{k+1})$, since we know from the $\la$-update step \eqref{eq:lambda_update} that ${\nabla}_{\la}\mathcal{L}_{\rho}(\ww_{k+1}) = g(\x_{k+1}) + \uu_{k+1} - \z_{k+1} - \beta(\la_{k+1} -\m_{k+1}) = \0.$ 

\section{Proofs of Asymptotic Convergence for Algorithm \ref{alg:ppala}}

\subsection{Proof of Primal Stationarity (Lemma \ref{proof_lem_primal_convergence_ppala})} \label{proof_lem_primal_convergence_ppala}

\begin{proof}
Analogous to the proof of Lemma \ref{lem_primal_convergence_plada} in Section \ref{proof_lem_primal_convergence_plada}, by using Lemma \ref{lem_ppal_decrease_converge} and \ref{lem_upper_bound_primal_gradient} with $\widehat{\delta}_k = \frac{\delta_k^2}{4{\rho}} + \frac{\delta_k}{\rho}$, we have
\begin{align}
\frac{1}{T}\sum_{k=0}^{T-1} \| \boldsymbol{\zeta}_{\pp}^{k+1} \|^2 \leq \frac{\frac{3d_1^2}{c_3}\left(\mathcal{L}_{\rho}(\ww_0) - \underline{\mathcal{L}_{\rho}}\right)}{T} + \frac{ \left( \frac{3 d_1^2}{4 \rho c_3} + 3 (M_g+1)^2 \right) \sum_{k=0}^{T-1} {\delta}_k^2}{T} + \frac{\frac{1}{\rho}\sum_{k=0}^{T-1}\delta_k}{T}. \label{eq:lem_stationarity_ppala}
\end{align}
Given $\delta_k = \frac{1}{p \cdot k^q + 1}$ with $2/3 < q \leq 1$ and $p > 0$, the third term on the RHS of the above inequality dominates the second term. Moreover, for sufficiently large $T$, one can easily show that
\begin{align}
&  \sum_{k=0}^{T-1}\delta_k
\approx \begin{cases} 
p^{-1}\log(pT)  & \text{if } \  q = 1, \\
 (p - qp)^{-1}  T^{1 - q} & \text{if } \  \frac{2}{3} < q < 1. \notag 
\end{cases}
\end{align}
Thus, for $q=1$, the sum grows logarithmically, while for $2/3 < q < 1$, the sum grows polynomially with $T$.
Therefore, for each choice of $q$, the RHS of \eqref{eq:lem_stationarity_ppala} goes to 0 as $T$ increases, 
which proves that the primal sequences are convergent. 
\end{proof}

\subsection{Proof of Primal Feasibility (Lemma \ref{lem_rate_primal_feasibility_ppala})}

\begin{proof}
Follow the same proof as that of Lemma \ref{lem_feasibility_plada} in Section \ref{proof_lem_rate_primal_feasibility_plada}.
\end{proof}

\subsection{Proof of Dual Feasibility (Lemma \ref{lem_dual_feasibility_ppala})}

\begin{proof}
The proof is analogous to that of Lemma \ref{lem_dual_feasibility_plada} in Section \ref{proof_lem_dual_feasibility_plada}. By the update rule \eqref{eq:u_update_ppala}, $\uu\in\mathbb{R}_+^m$ and $\bar{\uu}\ge0$. Thus, there are two cases for $\bar{\uu}>0$ and $\bar{\uu}=0$. If $\bar{\uu}>0$, by Lemma \ref{lem_feasibility_ppala} and the stationarity of $\bar{\uu}$ and $\bar{\la}$, $\bar{\la} + \rho(g(\bar{\x})+\bar{\uu}) = \bar{\la} = 0$.
Similarly, if $\bar{\uu}=0$, $\bar{\la} + \rho(g(\bar{\x})+\bar{\uu}) = \bar{\la} \ge 0$. In any case, we conclude that $\bar{\la}\ge0$.
\end{proof}

\subsection{Proof of Complementary Slackness (Lemma \ref{lem_complementary_slackness_ppala})}

\begin{proof}
It follows the same proof as that of Lemma \ref{lem_complementary_slackness_plada} in Section \ref{proof_lem_complementary_slackness_plada}.
\end{proof}

\section{Proofs of Non-asymptotic Convergence Rate of Algorithm \ref{alg:ppala}}

\subsection{Proof of Primal Stationarity Convergence Rate (Lemma \ref{lem_rate_primal_stationarity_ppala})}

\begin{proof}
The proof is analogous to that of Lemma \ref{lem_rate_primal_stationarity_plada} using Lemma \ref{lem_stationarity_ppala} and \eqref{eq:nu_ppala}. By triangle inequality, Jensen's inequality
and the result from Lemma \ref{lem_feasibility_ppala},
\begin{align}
    \bar{\boldsymbol{\nu}}_T - \bar{\la}_T &= \frac{1}{T}\sum_{k=0}^{T-1} \left(\frac{1}{\tau}-\rho\right)(\uu_{k+1} - \uu_k) + \la_{k+1}-\m_{k+1} \\
    &\le \frac{1-\tau\rho}{\tau T}(\uu_T - \uu_0) + \frac{1}{T}\sum_{k=0}^{T-1} \|\la_{k+1}-\m_{k+1}\| \\
    &= \mathcal{O}(1/T) + \mathcal{O}(1/\sqrt{T}) = \mathcal{O}(1/\sqrt{T}).
\end{align}
Therefore, the stationarity condition holds for $\bar{\boldsymbol{\nu}}_T$ with the same convergence rate.
\end{proof}

\subsection{Proof of Primal Feasibility Convergence Rate (Lemma \ref{lem_rate_primal_feasibility_ppala})}

\begin{proof}
The proof is analogous to that of Lemma \ref{lem_rate_primal_feasibility_plada} using Lemma \ref{lem_feasibility_ppala}.
\end{proof}

\subsection{Proof of Complementary Slackness Convergence Rate (Lemma \ref{lem_rate_complementary_slackness_ppala})}

\begin{proof}
From the optimality of $\uu$-update \eqref{eq:u_update_ppala} and the construction of $\boldsymbol{\nu}$ \eqref{eq:nu_ppala}, we have $\uu_{k+1} = \Pi_{\mathbb{R}_+^m}[\uu_{k+1} - \tau \boldsymbol{\nu}_{k}]$, which implies the complementarity condition $\langle \boldsymbol{\nu}_k, \uu - \uu_{k+1} \rangle \ge 0$ for all $\uu \ge \0$. The rest of the proof is the same as that of Lemma \ref{lem_rate_complementary_slackness_plada}.
\end{proof}

\section{Proof of Corollary \ref{cor1}} \label{proof_cor1}

\begin{proof}
By using  Jensen's inequality, $\left(\frac{1}{T} \sum_{k=0}^{T-1} \| \boldsymbol{\zeta}_{\pp}^{k+1} \|\right)^2 \leq \frac{1}{T} \sum_{k=0}^{T-1} \| \boldsymbol{\zeta}_{\pp}^{k+1}\|^2$, and taking the square root, we obtain
\begin{equation}
    \frac{1}{T} \sum_{k=0}^{T-1} \| \boldsymbol{\zeta}_{\pp}^{k+1} \| \leq \frac{1}{\sqrt{T}} \sqrt{\sum_{k=0}^{T-1} \| \boldsymbol{\zeta}_{\pp}^{k+1} \|^2}. \notag
\end{equation}
Denoting the RHS of inequality \eqref{eq:lem_stationarity_ppala} by $\Delta_{T}$ and combining Lemma \ref{lem_stationarity_ppala} with the above inequality give
\begin{equation}
    \frac{1}{T} \sum_{k=0}^{T-1} \| \boldsymbol{\zeta}_{\pp}^{k+1} \| \leq \frac{\sqrt{\Delta_T}}{\sqrt{T}} \leq \epsilon, \notag
\end{equation}
which, along with the result in \eqref{eq:primal_converge_rate}, gives $\widetilde{\mathcal{O}}\left(\frac{1}{\sqrt{T}}\right)$. Therefore, the following number of iterations is required to have $\epsilon$-primal stationarity:
$$T := \left\lceil \frac{\Delta_T}{\epsilon^2} \right\rceil  = \widetilde{\mathcal{O}}\left(\frac{1}{\epsilon^2}\right).$$
\end{proof}

\section{Experimental Details}

\subsection{Classification Problems Under Non-smooth Fairness Constraints}

\subsubsection*{Experimental Setup} We benchmark PLADA against four state-of-the-art algorithms: the single-loop switching subgradient (SSG) algorithm \cite{huang2023oracle}, two double-loop inexact proximal point (IPP) algorithms (IPP-ConEx \cite{boob2022stochastic} and IPP-SSG \cite{huang2023oracle}) and the multiplier model approach \cite{narasimhan2020approximate}. For the benchmark algorithms, we followed the hyperparameter settings of \cite{huang2023oracle} and \cite{narasimhan2020approximate} and we provide detailed descriptions of hyperparameters in Table \ref{tab:hyperparameter}. Note that we only used two hyper-parameter sets for 8 different experiments, while our benchmark algorithms used different hyper-parameters for every datasets, objectives and constraints.

\begin{table}[htbp]
\footnotesize
\caption{Hyper-parameters of PLADA used in experiments}
\centering
\begin{tabular}{c|c c c c c}
    \hline
    Problem & $\eta_w$ & $\eta_u$ & $\alpha$ & $\beta$ & $\gamma_0$\\
    \hline
    Models \ref{sec:exp_dp} and \ref{sec:exp_eo}  & 0.001 & 0.1 & 10.0 & 0.1 & 0.1 \\
    Neural network \ref{sec:exp_nn_if} & 0.1 & 0.01 & 10.0 & 0.5 & 0.1 \\
    \hline
\end{tabular}
\label{tab:hyperparameter}
\end{table}

\subsubsection*{Datasets} Our evaluation uses several standard real-world datasets:  \textbf{Adult} \cite{kohavi1996scaling}, \textbf{Bank} \cite{moro2014data}, \textbf{COMPAS} \cite{angwin2022machine} and \textbf{Communities and Crime} \cite{redmond2009communities}. The descriptions of the datasets are presented in Table \ref{tab:dataset}.

\begin{table}[ht]
\footnotesize
\caption{Real-world fairness datasets used in experiments}
\centering
\begin{tabular}{c|c c c c}
    \hline
    Dataset & n & d & Label & Sensitive Group \\
    \hline
    Adult (a9a) & 48,842 & 123 & Income & Gender \\
    Bank & 41,188 & 54 & Subscription & Age \\
    COMPAS & 6,172 & 16 & Recidivism & Race \\
    Communities and Crime & 1,994 & 140 & Crime & Race \\
    MSLR-WEB10K & 1.2 M & 136 & Relevance & Quality Score \\
    \hline
\end{tabular}
\label{tab:dataset}
\end{table}




\subsection{Non-convex Multi-class Neyman-Pearson Classification}

\subsubsection*{Implementation details}
We employ a two-layer feed-forward neural network with sigmoid activation for the classition faction on the Fashion-MNIST \cite{xiao2017fashion} and CIFAR10 \cite{krizhevsky2009learning} datasets. We compare Algorithm \ref{alg:ppala} (\text{PPALA}) with \text{GDPA}, as NL-IAPIAL can only handle convex constraints. A sigmoid function $\phi(y)=1/{(1+\exp(y))}$ is used for the loss function as in \cite{lu2022single}. 
For the experiments, we use 4 classes and set $\theta=1$ with $\kappa_i = 1$ for Fashion-MNIST and $\kappa_i = 2$ for CIFAR-10.\footnote{Note that the parameter settings for GDPA, as in \cite[Section F in Appendix]{lu2022single}, lead to a lack of convergence in the neural network setting, particularly with a small value of the threshold $\kappa_i$ and a large increase ratio for updating the penalty parameter.} PPALA uses a fixed learning rate $10^{-3}$ and parameters $\alpha = 10, \beta = 0.2$ for both datasets. The initial point $\x_0$ is randomly generated in each experiment. These numerical experiments were conducted using an A100 GPU and were implemented with Pytorch \cite{paszke2019pytorch}.

\section{Additional Experiments} \label{sec:additional_experiments}
In this section, we further validate the claims on hyperparameter sensitivity and dual variable convergence by conducting additional experiments. We also show the empirical performance of Algorithm \ref{alg:plada} by extending the application to stochastic setting.

\begin{figure*}[t]
\hspace*{70pt}\makebox[10pt]{Loss}%
\hspace*{142pt}\makebox[10pt]{DP Violation}%
\hspace*{138pt}\makebox[10pt]{Near Stationarity}\\%
\subfigure{\includegraphics[width=0.311\textwidth]{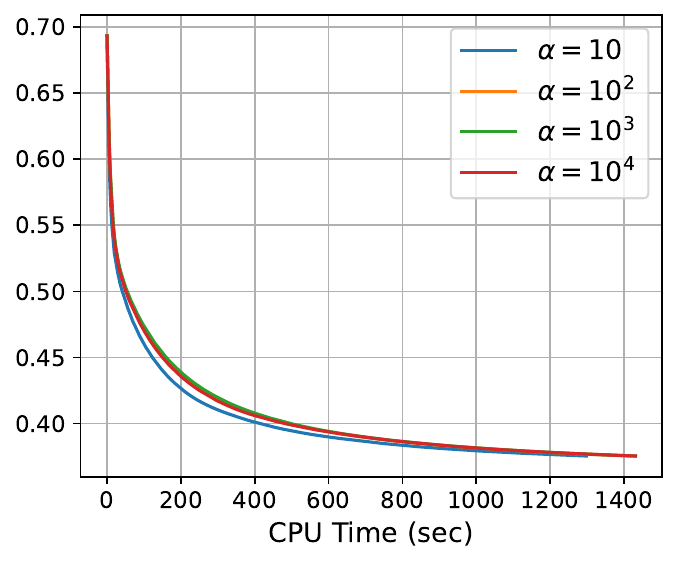} \label{fig:log_dp_a9a_alpha_obj}}
\subfigure{\includegraphics[width=0.325\textwidth]{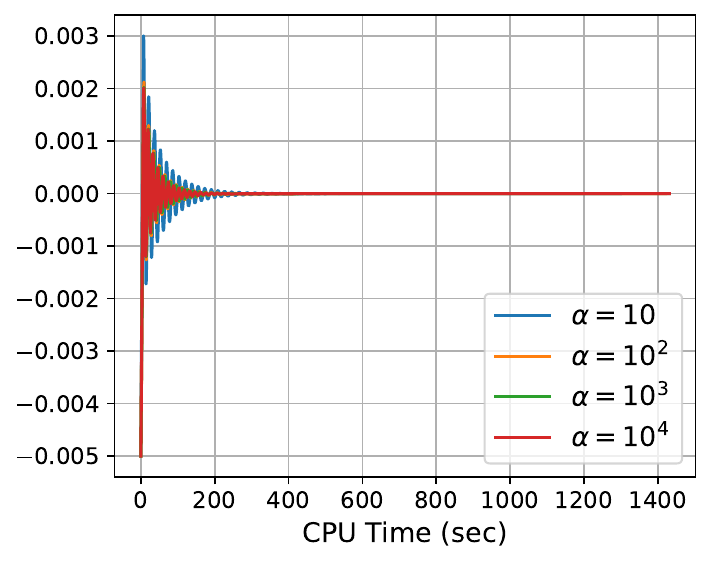} \label{fig:log_dp_a9a_alpha_cons}}
\subfigure{\includegraphics[width=0.32\textwidth]{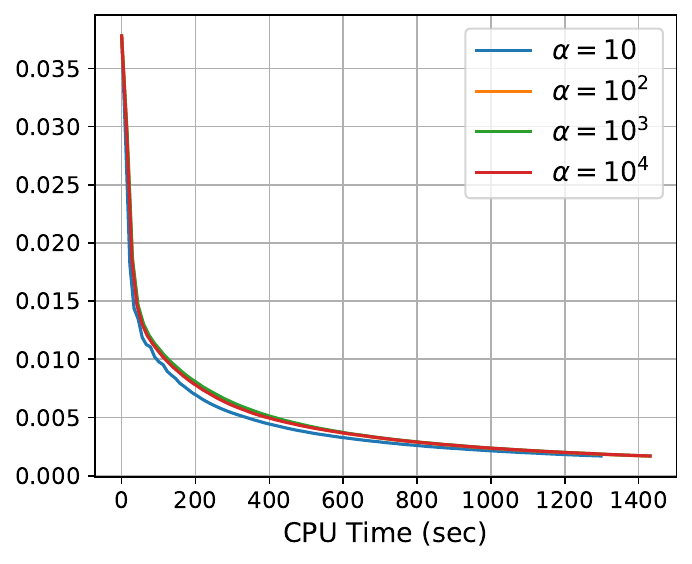} \label{fig:log_dp_a9a_alpha_stat}}
\caption{Comparison of the performance of PLADA with different $\alpha$ on the logistic loss objective with demographic parity (DP) constraint on Adult dataset. The results show the performance of PLADA is not sensitive to the value of $\alpha$ ($\beta=0.1$ is fixed).}
\label{fig:alpha_comparison}
\end{figure*}

\begin{figure*}[t]
\hspace*{68pt}\makebox[10pt]{Loss}%
\hspace*{143pt}\makebox[10pt]{DP Violation}%
\hspace*{140pt}\makebox[10pt]{Near Stationarity}\\%
\subfigure{\includegraphics[width=0.31\textwidth]{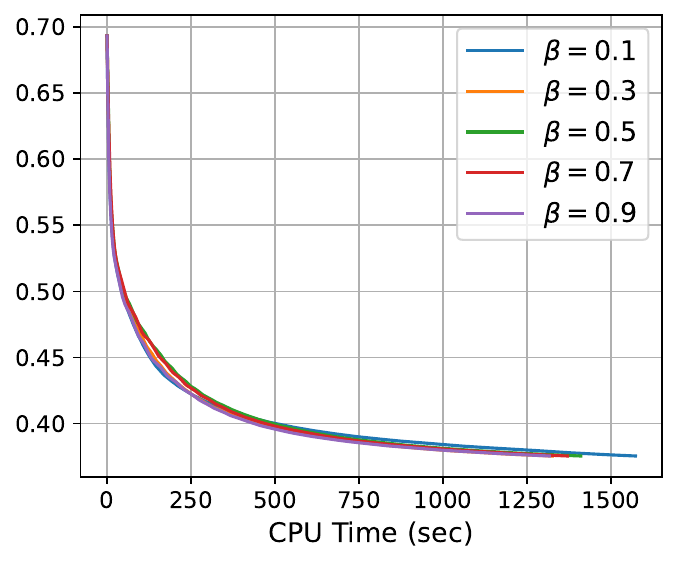} \label{fig:log_dp_a9a_beta_obj}}
\subfigure{\includegraphics[width=0.325\textwidth]{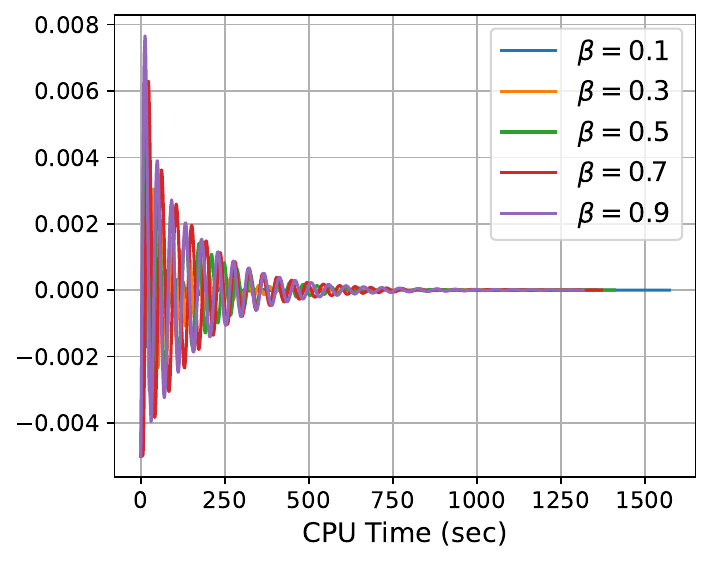} \label{fig:log_dp_a9a_beta_cons}}
\subfigure{\includegraphics[width=0.315\textwidth]{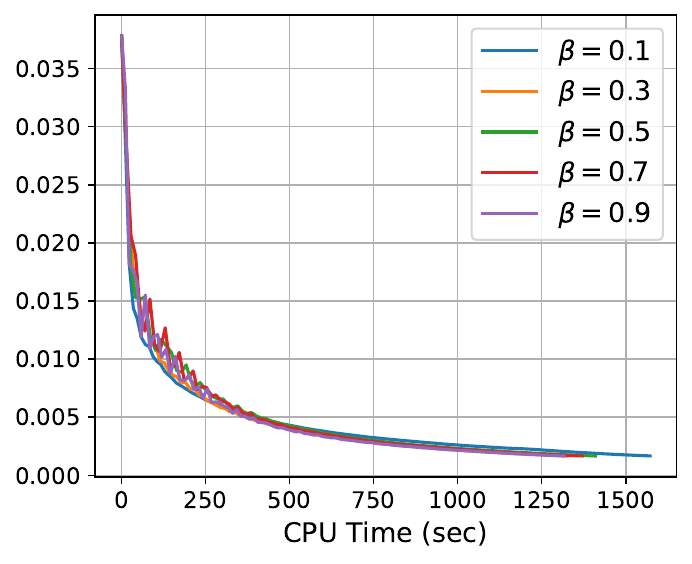} \label{fig:log_dp_a9a_beta_stat}}
\caption{Comparison of the performance of PLADA with different $\beta$ on the logistic loss objective with demographic parity (DP) constraint on Adult dataset. The results show that the performance of PLADA is very slightly sensitive to the choice of $\beta$, as it affect dual parameter defined by $\rho=\frac{\alpha}{1+\alpha\beta}$ ($\alpha=10$ is fixed).}
\label{fig:beta_comparison}
\end{figure*}

\subsubsection*{Hyperparameter robustness}
Although PLADA requires the selection of multiple hyperparameters ($\alpha, \beta, \rho, \eta, \tau$), we show that it is straightforward to select appropriate values for each hyperparameter. Figures \ref{fig:alpha_comparison} and \ref{fig:beta_comparison} show the performance of PLADA across a wide range of values for the key parameters $\alpha>0$ and $\beta>0$, respectively. The results demonstrate that the algorithm's convergence behavior is remarkably stable and can still provide solutions that minimize the objective while remaining feasible.

\begin{figure*}[t]
\hspace*{73pt}\makebox[10pt]{Adult}%
\hspace*{140pt}\makebox[10pt]{Bank}%
\hspace*{135pt}\makebox[10pt]{COMPAS}\\%
\subfigure{\includegraphics[width=0.323\textwidth]{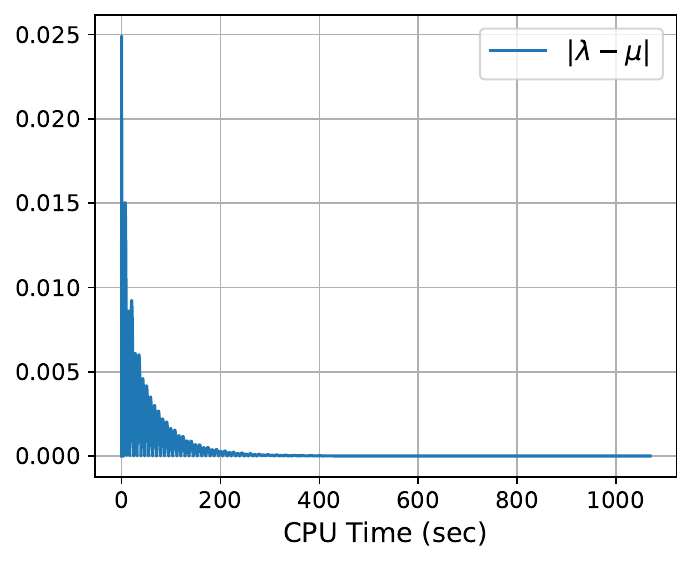} \label{fig:log_dp_a9a_dual}}
\subfigure{\includegraphics[width=0.313\textwidth]{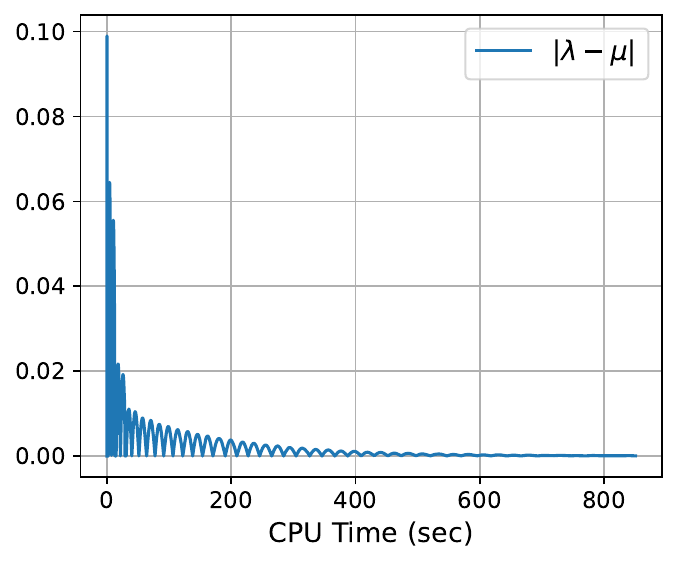} \label{fig:log_dp_bank_dual}}
\subfigure{\includegraphics[width=0.32\textwidth]{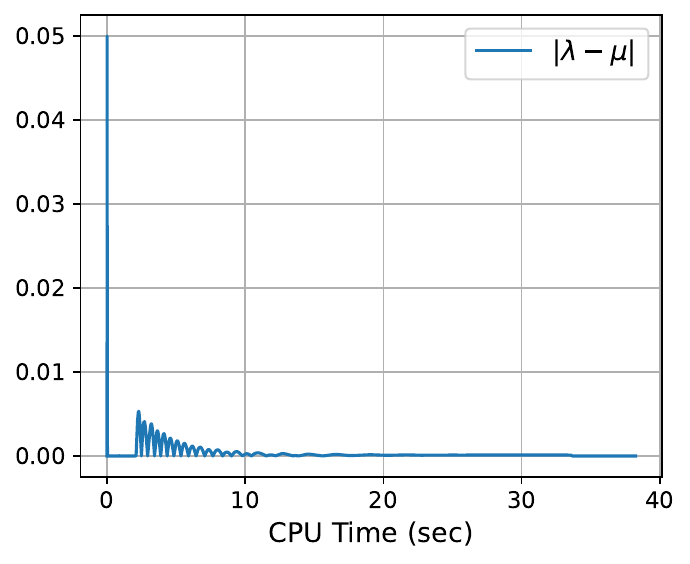} \label{fig:log_dp_compas_dual}}
\subfigure{\includegraphics[width=0.315\textwidth]{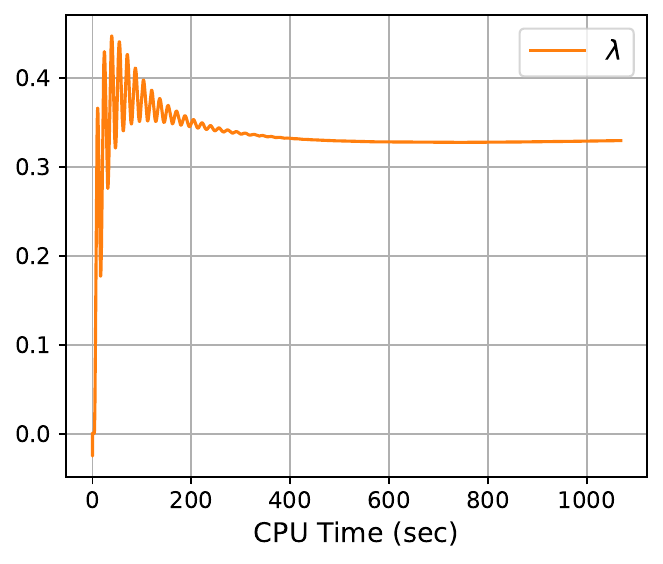} \label{fig:log_dp_a9a_lambda}}
\subfigure{\includegraphics[width=0.315\textwidth]{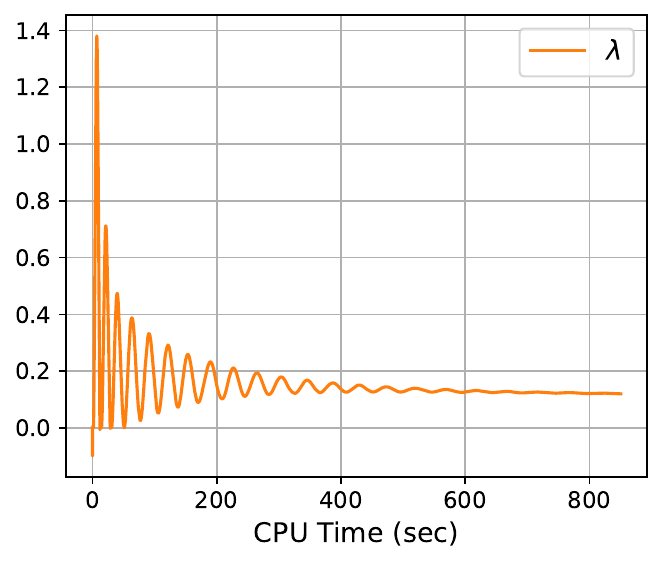} \label{fig:log_dp_bank_lambda}}
\subfigure{\includegraphics[width=0.335\textwidth]{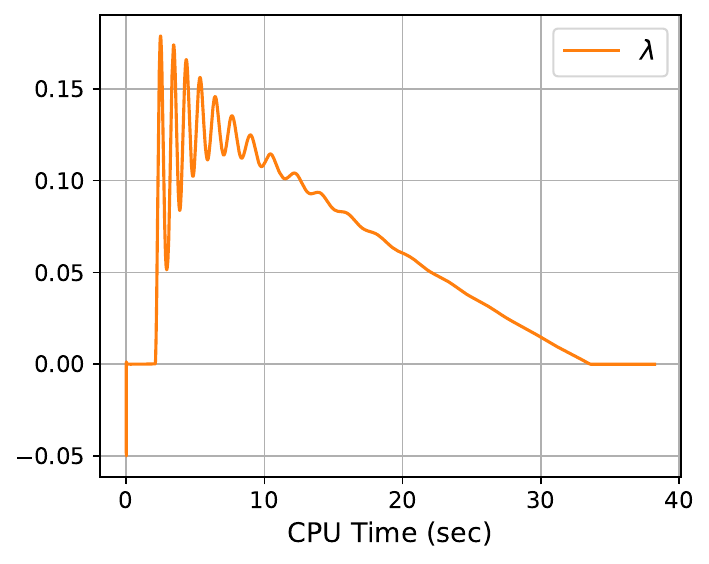} \label{fig:log_dp_compas_lambda}}
\subfigure{\includegraphics[width=0.315\textwidth]{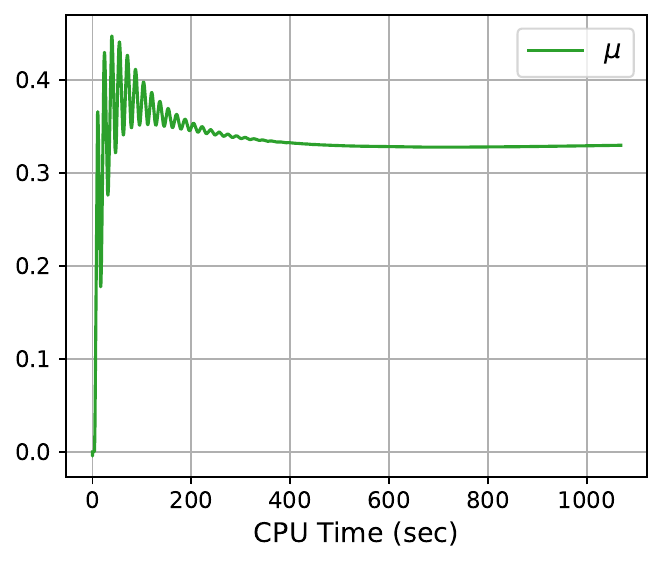} \label{fig:log_dp_a9a_mu}}
\subfigure{\includegraphics[width=0.315\textwidth]{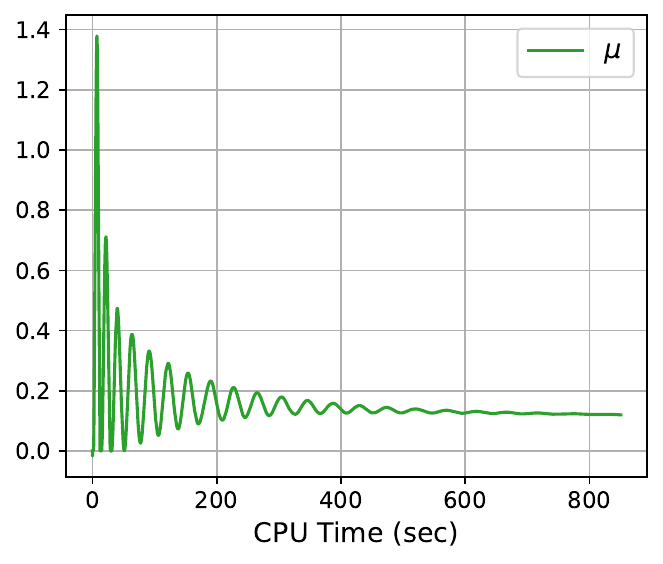} \label{fig:log_dp_bank_mu}}
\subfigure{\includegraphics[width=0.335\textwidth]{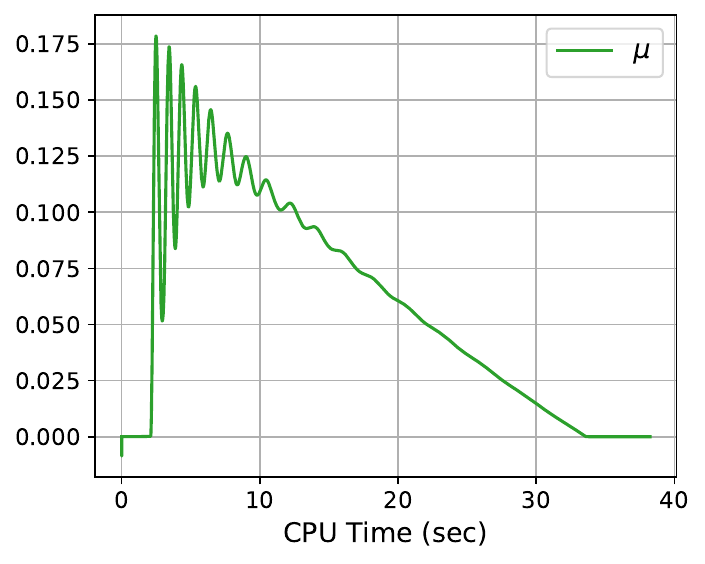} \label{fig:log_dp_compas_mu}}
\caption{The values of $|\lambda-\mu|, \lambda, \mu$ of PLADA on the logistic loss objective with demographic parity (DP) constraint. The results show the converging behavior of the dual variables and their difference.}
\label{fig:dual}
\end{figure*}

\subsubsection*{Dual variables convergence}
Lemma \ref{lem_feasibility_ppala} claims that the gap between the dual variables $\|\la-\m\|$ should converge to zero, ensuring feasibility. Figure \ref{fig:dual} provides empirical validation of this result. The plots clearly show the convergence of $|\la - \m|$ to zero, as well as the individual convergence of $\la$ and $\m$.

\begin{figure*}[t]
\centering{
\subfigure{\includegraphics[width=0.38\textwidth]{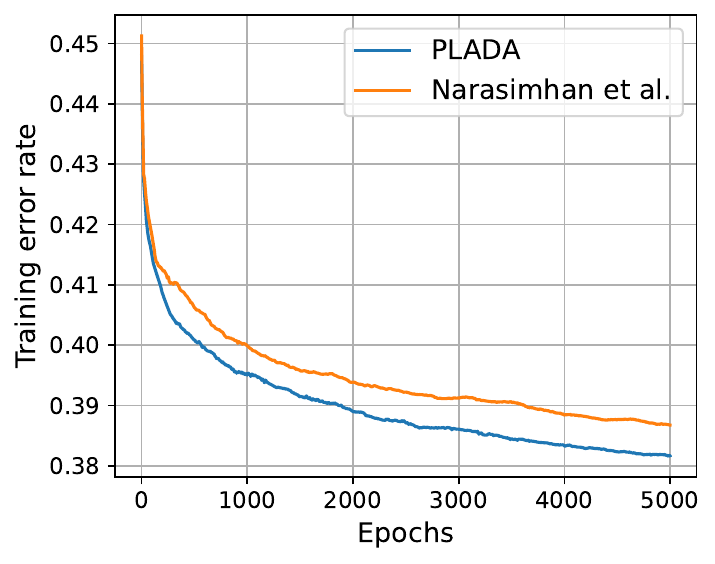} \label{fig:RF_obj}}
\subfigure{\includegraphics[width=0.39\textwidth]{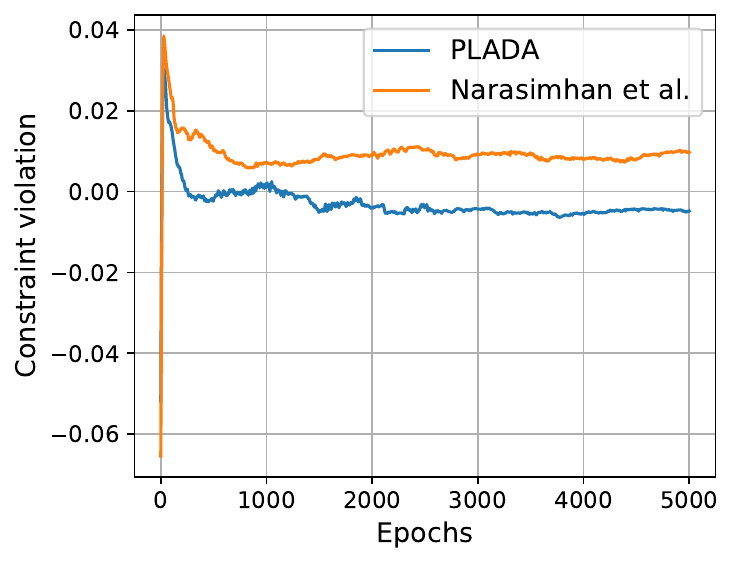} \label{fig:RF_cons}}}
\vspace{-5mm}
\caption{The average performance of PLADA and \cite{narasimhan2020approximate} on the ranking fairness versus Epochs after three repetitions. MSLR-WEB10K dataset has over 1.2M data points, from which over 470k pairs are created. PLADA achieves better constraint satisfaction with comparable error rate against approximate methods for the stochastic setting.}
\label{fig:stochastic_setting}
\end{figure*}

\subsubsection*{Highly stochastic setting} To evaluate PLADA in a more challenging setting, we applied it to a ranking fairness problem using the large-scale MSLR-WEB10K dataset, which involves over 470,000 pairwise constraints. In this highly stochastic environment, where mini-batching is inevitable, PLADA was benchmarked against the method of \cite{narasimhan2020approximate}. As shown in Figure \ref{fig:stochastic_setting}, our algorithm achieves lower error rate and better constraint satisfaction, demonstrating its effectiveness even under highly stochastic conditions.

\end{document}

%% file: math_commands.tex
\usepackage{amsmath,amsfonts,bm}

\newcommand{\x}{\mathbf{x}}
\newcommand{\y}{\mathbf{y}}
\newcommand{\z}{\mathbf{z}}

\newcommand{\uu}{\mathbf{u}}
\newcommand{\vv}{\mathbf{v}}
\newcommand{\ww}{\mathbf{w}}
\newcommand{\la}{\boldsymbol{\lambda}}
\newcommand{\m}{\boldsymbol{\mu}}
\newcommand{\0}{\mathbf{0}}
\newcommand{\A}{\mathbf{a}}
\newcommand{\bb}{\mathbf{b}}

\newcommand{\dd}{\mathbf{d}}
\newcommand{\pp}{\mathbf{p}}

\def\eqref#1{equation~\ref{#1}}

\def\1{\bm{1}}

\def\vv{{\bm{v}}}

\DeclareMathAlphabet{\mathsfit}{\encodingdefault}{\sfdefault}{m}{sl}
\SetMathAlphabet{\mathsfit}{bold}{\encodingdefault}{\sfdefault}{bx}{n}

